\documentclass[11pt]{article}

\usepackage{enumerate}
\usepackage{amsthm}
\usepackage{amssymb}
\usepackage{amsmath}
\usepackage{amsfonts}
\usepackage{cancel}
\usepackage{times}
\usepackage{bbm}
\usepackage{cite}
\usepackage{color, soul}
\usepackage{xcolor}
\usepackage{hyperref}
\definecolor{my-red}{rgb}{0.5,0.0,0.0}
\definecolor{my-blue}{rgb}{0.0,0.0,0.6}
\definecolor{my-green}{rgb}{0.0,0.5,0.0}
\definecolor{light-gray}{gray}{0.6}
\hypersetup{
	colorlinks   = true,
	citecolor    = my-green,
	linkcolor    = my-blue
}
\usepackage[T1]{fontenc}

\usepackage[paper=a4paper, top=2.5cm, bottom=2.5cm, left=2.5cm, right=2.5cm]{geometry}

\def\qed{\hfill $\vcenter{\hrule height .3mm
		\hbox {\vrule width .3mm height 2.1mm \kern 2mm \vrule width .3mm
			height 2.1mm} \hrule height .3mm}$ \bigskip}

\def \v {\omega}
\def \bound {\mathcal{B}}

\def \RR {\mathbb R}

\def \EE {\mathbb E}

\def \PP {\mathbb P}
\def \eps {\varepsilon}
\def \vphi {\varphi}

\def \cP {\mathcal P}

\def \cN {\mathcal N}

\def \cI {\mathcal I}

\def \cX {\mathcal X}

\newcommand{\al}{\alpha}
\newcommand{\be}{\beta}

\newcommand{\la}{\lambda}
\newcommand{\R}{\mathbb{R}}
\renewcommand{\P}{\mathbb{P}}
\newcommand{\E}{\mathbb{E}}
\newcommand{\one}{{\bf 1}}

\def\Ent{\mathrm{Ent}}
\newcommand{\wt}[1]{\widetilde{#1}}
\DeclareMathOperator*{\argmin}{arg\,min}

\DeclareMathOperator{\sech}{sech}
\DeclareMathOperator{\arctanh}{arctanh}
\DeclareMathOperator{\diag}{diag}
\DeclareMathOperator{\Var}{Var}
\DeclareMathOperator{\Cov}{Cov}
\DeclareMathOperator{\Cor}{Corr}
\DeclareMathOperator{\Tr}{Tr}

\DeclareMathOperator{\op}{op}

\def\sT{{\mathsf T}}
\def\reg{{\textnormal{reg}}}
\def\rr{\tiny{\sf RR}}
\def\rb{\tiny{\sf RB}}
\def\br{\tiny{\sf BR}}
\def\bb{\tiny{\sf BB}}
\def\cut{{\rm cut}}

\newtheorem{theorem}{Theorem}
\newtheorem{lemma}{Lemma}

\newtheorem{proposition}{Proposition}
\newtheorem{corollary}[theorem]{Corollary}
\theoremstyle{definition}
\newtheorem{definition}[theorem]{Definition}
\theoremstyle{remark}
\newtheorem{remark}[theorem]{Remark}
\newtheorem*{remark*}{Remark}

\long\def\symbolfootnotetext[#1]#2{\begingroup
	\def\thefootnote{\fnsymbol{footnote}}\footnotetext[#1]{#2}\endgroup}

\title{Fast mixing in Ising models with a negative spectral outlier
	via Gaussian approximation}
\date{}
\author{Dan Mikulincer\thanks{Department of Mathematics, University of Washington. Seattle, WA. \url{danmiku@uw.edu}}\and Youngtak Sohn\thanks{Division of Applied Mathematics, Brown University. Providence, RI. \url{youngtak_sohn@brown.edu}}
}

\begin{document}
	\maketitle
	\vspace{-0.8cm}
	\begin{abstract}
    We study the mixing time of Glauber dynamics for Ising models in which the interaction matrix contains a single negative spectral outlier. This class includes the anti-ferromagnetic Curie-Weiss model, the anti-ferromagnetic Ising model on expander graphs, and the Sherrington-Kirkpatrick model with disorder of negative mean.
    Existing approaches to rapid mixing rely crucially on log-concavity or spectral width bounds and therefore can break down in the presence of a negative outlier. 
    
    To address this difficulty, we develop a new covariance approximation method based on Gaussian approximation. This method is implemented via an iterative application of Stein’s method to quadratic tilts of sums of bounded random variables, which may be of independent interest. The resulting analysis provides an operator-norm control of the full correlation structure under arbitrary external fields. Combined with the localization schemes of Eldan and Chen, these estimates lead to a modified logarithmic Sobolev inequality and near-optimal mixing time bounds in regimes where spectral width bounds fail.  We complement these results by proving exponential lower bounds on the mixing time for low temperature anti-ferromagnetic Ising models on sparse random regular graphs and Erd\"{o}s-R\'{e}nyi graphs, based on the existence of gapped states as in the recent work of Sellke.
    
	\end{abstract}
	\section{Introduction}
    In this work we study the mixing time of Glauber dynamics for Ising models whose interaction matrix $J$ contains a negative outlier.  As an application, we establish $O(n\log(n))$ mixing times for the anti-ferromagnetic Ising model on a random $d$-regular graph with inverse temperature $\beta=O(1/\sqrt{d})$ uniformly for all~$d<n$.

    For a symmetric matrix $J\in \R^{n\times n}$ and an external field $h\in \R^n$, the Ising measure on $\{-1,+1\}^n$ is
	
	$$\mu_{J,h}(x) =\frac{1}{Z_{J,h}} \exp\left(\langle x,Jx\rangle + \langle h, x\rangle\right),\qquad x\in \{-1,+1\}^n,$$
    where $Z_{J,h}$ is a normalizing constant. We refer to $J$ as the \emph{interaction matrix}.
    
    Glauber dynamics is the (discrete-time) Markov chain $(X_t)_{t\geq 0}$, which at each update time $t=0,1,\ldots$, picks a uniformly random coordinate $i \in [n]$ and resamples the $i$'th coordinate $X_{t}(i)$ from the conditional law of $\mu_{J,h}$ given the remaining spins. The chain is reversible with stationary distribution $\mu_{J,h}$, and under mild assumptions ergodic. Since its introduction in \cite{glauber1963time} as a model for the time evolution of interacting spins, Glauber dynamics has become one of the canonical examples of a high-dimensional local Markov chain.
    
    The main object of study is the rate at which $X_t$ converges to stationarity, typically measured in $\ell^2$ or total variation distance. A powerful framework for analyzing convergence is provided by functional inequalities, most notably Poincar\'{e} and modified log-Sobolev inequalities (MLSI). Let $P=(P_{x,y})_{x,y\in \{-1,+1\}^n}$ be a reversible transition operator with stationary measure $\mu$. We say $P$ satisfies an MLSI with constant $\rho_{\mathrm{LS}}>0$ if for any test function $f:\{-1,1\}^n \to \RR_+$, 
   \[
   \mathrm{Ent}_\mu(Pf)\leq (1-\rho_{\mathrm{LS}})\mathrm{Ent}_{\mu}(f),
   \]
   where $Pf(x):=\sum_{y\in \{-1,+1\}^n}P_{x,y}f(y)$ and $\Ent_{\mu}(g):=\E_{\mu}[g\log(g/\E_{\mu}[g])]$.
   We denote the optimal constant by $\rho_{\mathrm{LS}}(\mu)$. In what follows, $P$ will always be the Glauber transition operator for $\mu$.
   
   It is well known that the MLSI constant controls the rate of convergence in relative entropy, and consequently, in total variation.  In particular, (see \cite{bobkov2006modified})
	\begin{align*}
	t_{\mathrm{mix}}(\varepsilon):=	\min\left\{t\geq 0:\mathrm{TV}\left(\mathrm{Law}(X_t),\mu\right)\leq \varepsilon\right\} \leq C\rho_{\mathrm{LS}}(\mu)^{-1}\left(\log\log\left(\frac{1}{\min\limits_x \mu(x)}\right) + \log\left(\frac{1}{2\varepsilon^2}\right)\right). 
	\end{align*}
	Thus, in many cases of interest, including those considered in the present work, if the MLSI constant satisfies $\rho_{\mathrm{LS}}(\mu) = \Omega(n^{-1}),$ then the mixing time has the optimal bound $O(n\log(n))$.\footnote{Throughout, for convenience we work with discrete-time Glauber dynamics. For the continuous-time version in which each site is updated according to independent rate-$1$ Poisson clocks, the corresponding MLSI constant scales by $\rho_{\mathrm{LS}}^{\mathrm{cont}}(\mu)\asymp n\rho_{\mathrm{LS}}^{\mathrm{disc}}(\mu)$. All results stated here transfer directly between the two versions after this normalization.} 
    
    Originating in the work of Dobrushin \cite{dobrushin1968problem} (later adapted to the study of Markov chains in \cite{weitz2005combinatorial}), much effort has been devoted to identifying general structural conditions on the interaction matrix $J$ that ensure rapid mixing. Recent advances \cite{baurschmidt2019simple,eldan2022spectral,chen2025localization} culminated in the \emph{spectral width condition} 
    \[
    \lambda_{\max}(J) - \lambda_{\min}(J)<\frac{1}{2}
    \]
    which implies that the MLSI constant satisfies $\rho_{\mathrm{LS}}(\mu_{J,h}) = \Omega(n^{-1})$ uniform in $h\in \R^n$, and hence optimal mixing. This spectral width condition is tight in general: for the ferromagnetic Curie-Weiss model $J = \frac{\beta}{n} {\bf 1}{\bf 1}^T, \beta>0$, Glauber dynamics exhibits a phase transition in mixing precisely at $\beta = \frac{1}{2}$~(see e.g. \cite{levin08glauber}). However, as demonstrated in \cite[Theorem 3]{AKV24}, there exist broad classes of interaction matrices $J$ for which the spectral width condition is overly conservative, potentially resulting in pessimistic predictions. This observation motivates the development of tools adapted to more structured classes of models, rather than relying solely on worst-case spectral width condition.
	
	As the ferromagnetic Curie-Weiss model illustrates, the presence of a single sufficiently large \emph{positive} eigenvalue can dramatically impede mixing of Glauber dynamics. Our main focus in this work is to treat the opposite \emph{anti-ferromagnetic} regime, where the interaction matrix $J$ has a \emph{negative} eigenvalue outlier. As we discuss below, this class encompasses several models of interest. Our main result establishes a uniform MLSI bound for this family.
    
    Note that because the measures we consider are supported on $\{-1,+1\}^n$, one may freely add a multiple of the identity to the interaction matrix without changing the model. In particular, we may assume that all remaining eigenvalues of $J$ are nonnegative, in which case the spectral width condition reduces to a constraint on the top eigenvalue.
	\begin{theorem} \label{thm:main}
		Let $J$ be an $n\times n$ positive semi-definite matrix with $\|J\|_{\op} < \frac{1}{2}$. For any $u, h\in \R^n$ such that $\|u\|_{\infty}\leq 1$, consider the Ising measure on $\{-1,1\}^n$,
		\begin{equation} \label{eq:antifer}
			\mu_{\beta,u,J,h}(x)\propto \exp\Big(-\frac{\beta}{n}\langle u, x\rangle^2 + \langle x, Jx \rangle +\langle h, x\rangle \Big),
		\end{equation}
		where $0\leq \beta \leq n^{1-\eps}$ for some $0<\eps<1$. Then,
		$$\rho_{\mathrm{LS}}(\mu_{\beta,u,J,h}) \geq (1-o(1))(1-2\|J\|_{\op})n^{-1},$$
        where the $o(1)$ depends only on $n,\eps$ and vanishes polynomially in $n$.
         Consequently, Glauber dynamics mixes in time $O(n\log(n))$.
	\end{theorem}

    Here, the condition $\beta \leq n^{1-\eps}$ is imposed primarily for simplicity, and our proofs can track the dependence on $\eps>0$ explicitly to cover $\beta\leq n^{1-o(1)}$, where the term $o(1)$ can be made quantitative. We do not attempt to optimize this dependence.
    
	We highlight some features of Theorem \ref{thm:main} (see Section \ref{sec:related} for comparisons with the literature). First, the vector $u$ may be essentially arbitrary: we do not assume $u={\bf 1}$ nor do we require any uniform control on the ratio of its coordinates. When $u={\bf 1}$, the induced quadratic term enforces negative correlation among coordinates, allowing one to use tools from the theory of negative dependence (e.g.~\cite{borcea2009negative}). However, as shown in Proposition~\ref{prop:cov:approx:entry}, once the signs of $u$ vary, negative dependence breaks down and these tools are no longer available.

    Second, after absorbing the temperature parameter into $u$ by writing $\tilde{u}=\sqrt{\beta/n}\cdot u$, the condition $0\leq \beta \leq n^{1-\eps}$ becomes $\|\tilde{u}\|_{\infty}\leq n^{-\eps/2}$. For the simplest case $J=0$, this is far weaker, especially for coordinate delocalized $\tilde{u}$, than the spectral width condition imposing $\ell^2$ constraint $\|\tilde{u}\|_2< 1/\sqrt{2}$. This illustrates a fundamental asymmetry: positive spectral outliers may destroy rapid mixing, while negative outliers are comparatively benign.
    
   Finally, we essentially allow $\beta = o(n)$, which is nearly optimal, and obtain MLSI constants that are uniform in $\beta$ throughout this regime in contrast to previous results (see Section~\ref{sec:comps}). This cannot be improved to $\beta = \Omega(n)$. Indeed, for the anti-ferromagnetic Curie-Weiss measure $\propto \exp\big(-\beta_0\big(\sum\limits_{i=1}^n x_i\big)^2\big),$
 the Glauber dynamics reweigh the transition probabilities by $e^{\beta_0}$, thus the mixing time will scale by this factor. In Lemma \ref{lem:lowerbound}, we show that the MLSI constant is at most $O(e^{-\beta_0}/n)$. Thus, there is no hope to obtain the uniform bounds of the form in Theorem \ref{thm:main} in this regime.

    The main technical challenge in proving Theorem \ref{thm:main} is to cover the near-optimal window $\beta\leq n^{1-\eps}$, where the negative outlier may lie far away from the spectrum of $J$. As we show in Section~\ref{sec:comps}, achieving this requires a highly accurate Gaussian approximation for quadratic tilts of a sum of random variables (see Eq.~\eqref{eq:CGF}). We obtain such estimates through an iterative application of Stein’s method, a result that may be of independent interest. Our key contribution is a sharp operator-norm approximation for the covariance and correlation matrices of the measures appearing in \eqref{eq:antifer} when $J=0$ and $h\in \R^n$ is arbitrary. In particular, we show that $\mathrm{Cov}(\mu_{\beta,u,0,h})$ is, up to an operator norm error of $\tilde{O}\left(\beta/n\right)$, a diagonal matrix minus a rank one matrix. The role of these approximations is discussed in Section \ref{sec:comps}.

	\subsection{Relevant cases of interest} \label{sec:corrolaries}
	In this section we present several natural classes of models to which Theorem \ref{thm:main} applies.

    \medskip\noindent\textbf{Regular graphs}: Our first case of interest lies in the \emph{anti-ferromagnetic} Ising model on a $d$-regular graph, where the interaction matrix is $J_G := -\beta A_G,$ and where $A_G$ is the adjacency matrix of some $d$-regular graph and $\beta > 0$ is the temperature parameter. These models go back to Ising's original work on the cycle graph and later the works of Peierls and Onsager, who established phase transitions on $d$-dimensional grids. Since $A_G/d$ is stochastic, the all-ones vector ${\bf 1}$ is the leading eigenvector with $A_G{\bf 1} = d{\bf 1}$. Decomposing $A_G$ on $\bf 1^\perp$ yields the following consequence of Theorem~\ref{thm:main}.
	\begin{corollary}\label{cor:dregular}
	    	Let $3 \leq d < n$ and let $G$ be a $d$-regular graph, with eigenvalues $-d\leq \lambda_n\leq \dots \leq \lambda_2\leq \lambda_1 = d$. For $h \in \RR^n$, and $\beta \geq 0$, let $J_G= -\beta A_G$ and consider $\mu_{J_G,h}$, the Ising measure on $\{-1,1\}^n$.	
		Then, if $\beta < \min\left(\frac{1}{2(\lambda_2-\lambda_n)},\frac{n^{1-\eps}}{d}\right)$, for some $\eps> 0$,
		$$\rho_{\mathrm{LS}}(\mu_{J_G,h}) \geq (1-o(1))(1-2\beta(\lambda_2-\lambda_n))n^{-1},$$
		where the $o(1)$ depends only on $n,\eps$ and vanishes polynomially in $n$.
         Consequently, Glauber dynamics mixes in time $O(n\log(n))$.
	\end{corollary} 
	One interpretation is that for anti-ferromagnetic models on $d$-regular expanders, it is the second eigenvalue that determines the potential threshold for the temperature, rather than the first eigenvalue appearing in the spectral width condition. A prominent set of examples of such expanders are the random $d$-regular graphs, which are taken to be uniform among all $ d$-regular graphs on $n$ vertices, for $dn$  even. We denote the distribution by $G_{\textnormal{reg}}(n,d)$. These random graphs enjoy an a-priori bound, holding with high probability, on their spectral gap due to Friedman's inequality, \cite{friedman2009proof} for $d$ fixed. More recently this inequality was extended to arbitrary $d$ by \cite{he2024spectral ,chen2024new}. We immediately obtain the following corollary. 
	\begin{corollary} \label{cor:randomdregular}
		Let $3\leq d < n$ for some $\eps > 0$, and let $G \sim G_{\textnormal{reg}}(n,d)$ be a random $d$-regular graph, with $h \in \RR^n$. Then, if $0\leq \beta < \min\left(\frac{1}{8\sqrt{d(1-d/n)}}, \frac{n^{1-\eps}}{d}\right)$, with $1-o(1)$ probability over $G$, 
		$$\rho_{\mathrm{LS}}(\mu_{J_G,h}) \geq (1-o(1))\left(1- 8\beta\sqrt{d(1-d/n)}\right) n^{-1},$$
		where the $o(1)$ depends only on $n,\eps$ and vanishes polynomially in $n$, and $\mu_{J_G,h}$ is as in Corollary~\ref{cor:dregular}. If $d=O(1)$, then we can take the range of $\beta$ to be $0\leq \beta < \frac{1}{8\sqrt{d-1}}$. Consequently, Glauber dynamics mixes in time $O(n\log(n))$.
	\end{corollary}
	We emphasize that in both Corollaries~\ref{cor:dregular} and~\ref{cor:randomdregular}, the degree $d$ is allowed to grow with $n$, and the MLSI bound remains uniform in $d$ throughout the admissible regime. This uniformity is one of the important strengths of Theorem~\ref{thm:main}.

We further remark that when $d\equiv d_n\to\infty$ as $n\to\infty$, and $\beta=o(d^{-1/2})$, mean-field predictions for the \emph{free energy} are known to be valid~\cite{basak17universality,augeri21transportation}. However, it is unclear how MLSI can be obtained in this regime, let alone up to $\beta=\Theta(1/\sqrt{d})$. Note that when $d=O(1)$, $\beta= \Theta(1/\sqrt{d})$ is the correct transition window for rapid mixing to exponential slow mixing in the anti-ferromagnetic Ising model on $G_{\reg}(n,d)$ (see Proposition~\ref{prop:slow:mixing:ER} below).

	\medskip\noindent\textbf{Erd\"os-R\'enyi graphs:}
	In Corollary~\ref{cor:randomdregular}, we used that the leading eigenvector of a $d$-regular graph is the all-ones vector. Theorem \ref{thm:main}, however, applies more broadly, provided the top eigenvector is sufficiently delocalized. A natural example is $G(n,d/n)$, Erd\"os-R\'enyi random graphs. For these graphs, it's well known that the top eigenvector is delocalized, with the exact bound depending on $d$. When $G$ is sparse, however, high-degree vertices cause the spectrum to fluctuate more substantially than in the regular case, and a spectral gap may fail to exist. For this reason our result requires $d\gtrsim \log n$.
	\begin{corollary}\label{thm:denseER}
		Let $C>0$ be a sufficiently large universal constant and suppose $C\log(n)\leq d \leq n$. Let $G \sim G(n,d/n)$. For any $h\in \R^n$, if $0\leq \beta < \min\left(\frac{c}{8\sqrt{d}},\frac{n^{1-\eps}}{d}\right)$ for some $\eps > 0$, and a constant $c>0$ depending on $C>0$ (with $c=1+o(1)$ as $C\to\infty$) then with probability $1-o(1)$ over $G$,
		$$\rho_{\mathrm{LS}}\left(\mu_{J_G,h}\right) \geq (1-o(1))\big(1-8c^{-1}\beta\sqrt{d}\big) n^{-1},$$
		where the $o(1)$ depends only on $n,\eps$ and vanishes polynomially in $n$. Consequently, Glauber dynamics mixes in time $O(n\log(n))$.
	\end{corollary}
    
	For sparse graphs with $p  =d/n$, for some $d > 0$, one can enforce concentration of the spectrum by pruning high-degree nodes, as shown in \cite{le2017concentration}, yielding comparable results to Theorem \ref{thm:denseER} for the pruned graph. However, recent work~\cite{liu2024fast} on spin glasses on sparse Erd\"os-R\'enyi graphs suggests such pruning is unnecessary. While Theorem~\ref{thm:main} does not apply directly to the full graph in this setting, it can be applied to the bulk of the graph. Treating the neighborhoods of high-degree vertices separately, as in~\cite{liu2024fast, prodromidis2025polynomial}, we expect that polynomial mixing bounds can be obtained in the sparse regime. We do not pursue this direction here. 
    
	\medskip\noindent\textbf{Sherrington-Kirkpatrick (SK) with anti-ferromagnetic interactions:} The classical SK model~\cite{sherrington_kirkpatrick75} is defined by an interaction matrix $J$ drawn from the Gaussian Orthogonal Ensemble, where $J$ is symmetric with $J_{ij}\sim \mathcal{N}(0,\beta^2/n)$ and inverse temperature $\beta>0$. By Wigner's semi-circle theorem, the spectral width criterion of~\cite{eldan2022spectral} implies that Glauber dynamics mixes rapidly whenever $\beta < \frac{1}{8}$, a result later improved to $\beta \approx 0.148$ in \cite{AKV24}. The precise threshold for mixing of the Glauber dynamics, conjecturally at $\beta  = \frac{1}{2}$, remains open (see \cite{el2025sampling} for a different sampling method). 

      Kirkpatrick and Sherrington~\cite{kirkpatrick78infinite} also considered a non-zero mean Gaussian disorder $J$, where $J_{ij}\sim \mathcal{N}(\mu/n, \beta^2/n)$ for $i<j$ and some $\mu\in \R$. Rigorous works have focused on the ferromagnetic case $\mu>0$, analyzing the free energy~\cite{chen2014mixed} and its fluctuations~\cite{baik2017fluctuations,baik2018ferromagnetic,Banerjeee20fluctuation}, and establishing a phase diagram that differs qualitatively from the standard ($\mu = 0$) model due to the emergence of the ferromagnetic phase. In contrast, the anti-ferromagnetic interaction $\mu<0$ is predicted in \cite{kirkpatrick78infinite} to exhibit the same phase diagram as the mean-zero model. Consistent with this prediction, combining Theorem \ref{thm:main} and the fact that Gaussian Orthogonal Ensemble has spectral width $4+o(1)$ with probability $1-o(1)$~\cite{AGZ2010}, we immediately obtain that the range of $\beta$ for which the SK model satisfies MLSI remains of constant order for all $\mu<0$ as long as $|\mu|\ll n$.

	\begin{corollary}\label{cor:SK}
	    	Let $J\in \R^{n\times n}$ be symmetric with independent off-diagonal entries $J_{ij}\sim \mathcal{N}(-\mu/n,\beta^2/n)$ for $i<j$. If $0\leq \mu\leq n^{1-\eps}$ and $0\leq \beta <1/8$ for some $\eps\in (0,1)$, then for any $h\in \R^n$,  with probability $1-o(1)$ over $J$,
            $$\rho_{\mathrm{LS}}(\mu_{J,h}) \geq (1-o(1))(1-8\beta) n^{-1},$$
           where the $o(1)$ depends only on $n,\eps$ and vanishes polynomially in $n$. Consequently, Glauber dynamics mixes in time $O(n\log(n))$.
	\end{corollary} 
\subsection{Exponential lower bounds on the mixing time}

    To complement Corollaries \ref{cor:randomdregular} and \ref{thm:denseER}, we show that the range $\beta=O(1/\sqrt{d})$ is tight up to a constant. To be specific, we prove that if either $G\sim G_{\reg}(n,d)$ with $d=O(1)$ or $G\sim G(n,d/n)$ with $d\leq n/2$, and if $\beta > C/\sqrt{d}$, with probability $1-o(1)$ over $G$, the mixing time is at least $e^{cn\beta \sqrt{d}}$. An aspect that makes this result interesting is that, at low temperatures, these models are expected to exhibit \emph{continuous/full replica symmetry breaking}~\cite{krzkakala2008potts} as in the SK model. Therefore, the overlap free-energy-barrier mechanism in the sense of \cite[Definition~1.2 and Theorem~1.3]{ben2018spectral} is not expected to be applicable to these models.
    Instead, our argument follows the strategy by Sellke~\cite{Sellke25}, who recently proved exponential mixing time lower bounds for the Sherrington-Kirkpatrick at low temperature using the existence of \textit{gapped states} (see \eqref{eq:def:gapped} for the definition). For Erd\"os-R\'enyi graphs $G(n,d/n)$, the existence of gapped states was previously established in for $d=\Theta(1)$ or $d=\Theta(n)$ in~\cite{minzer2024perfect, dandi2025maximal}. Here, we observe that their proof also extends to the intermediate regimes $1\ll d\ll n$. More significantly, we prove the existence of gapped states for sparse random regular graphs $G\sim G_{\reg}(n,d), d=O(1)$. The argument is rather delicate: using the coupling of Dembo, Montanari, and Sen~\cite{Dembo2017extremal}, originally developed for the extremal cut problems, we show that the local fields in the regular and Erd\"os-R\'enyi models only differ by a vanishing fraction of vertices. The key point is a non-trivial cancellation effect in the change of the local field, analogous to the global max-cut comparison in~\cite{Dembo2017extremal}. Having established the existence of gapped states in these ensembles, we apply Sellke's argument to obtain slow mixing. However, non-trivial modifications are required to adapt the argument to the non-dense setting, due to the presence of high-degree nodes in the case of $G(n,d/n)$.

    Finally, we note that this slow mixing result extends to the \emph{Kawasaki dynamics}~\cite{BBD25} in the centered slice~$\{x:\langle \one, x\rangle=0\}$. This provides a positive answer, up to universal constants, to Open Question~2 of \cite{BBD25}: Kawasaki dynamics on sparse random $d$-regular graphs is exponentially slow once $\beta \gtrsim 1/\sqrt{d}$. We refer to Proposition~\ref{prop:slow:mixing:ER} for the precise statement.
    
            \begin{proposition}\label{thm:explowerbound}
                  Consider $\mu_{J_G, 0}$ the anti-ferromagnetic Ising model with zero external field, where either $G\sim G_{\reg}(n,d)$ with $d\geq d_0$ fixed sufficiently large, or $G \sim G(n,d/n)$ with $d\leq n/2$. If $\beta \geq C/\sqrt{d}$, then with probability $1-o(1)$ over $G$, the mixing time of Glauber dynamics is at least $e^{cn\beta \sqrt{d}}$.
            \end{proposition}

	\subsection{Components of the proof} \label{sec:comps}
	
	\subsubsection{The localization framework}
	Our method of proof follows the localization framework of Chen and Eldan~\cite{chen2025localization}. While we primarily use this framework and its results as a black box, we outline the key ideas for context. The underlying principle is to study functional inequalities associated with a given measure through a measure-valued stochastic process whose mass progressively localizes in smaller regions until collapsing to an atom. When the evolution is constrained to preserve certain martingale properties, global features of the original measure can be deduced from the increasingly local structure along the trajectory.

	\medskip\noindent\textbf{Stochastic localization:} In this work we will implicitly implement two localization schemes. The first localization scheme we employ is the \emph{stochastic localization} process by Eldan \cite{eldan2013thin}. 
   Given the initial measure $\mu$ on $\R^n$, the measure-valued process $(\nu_t)_{t\geq 0}$ is defined by the change of density $\frac{d\nu_t}{d\mu}(x)=F_t(x)$, where $F_t(x)$ solves infinite system of SDEs
    \begin{equation}\label{eq:SL:SDE}
    dF_t(x)=F_t(x)\langle x- b(\nu_t), C_t dB_t\rangle\,,\quad\forall x\in \R^n\,,\quad F_0\equiv 1\,.
    \end{equation}
    Here, $b(\nu_t)\equiv \int x d\nu_t(x)$ is the barycenter of $\nu_t$, and $(C_t)_{t\geq 0}$ is adapted process of $n\times n$ positive semi-definite matrices. The uniqueness and existence of $(\nu_t)_{t}$ was established in \cite{eldan2013thin, eldan20clt}. By Ito's formula (see \cite[Fact 3.5]{chen2025localization}), this evolution is equivalent to tilting $\mu$ by a random quadratic factor
    \begin{equation} \label{eq:SL}
        \frac{d\mu_t}{d\mu}(x) = e^{-\langle x, Q_tx\rangle + \langle h_t,x\rangle}\,,\quad Q_t=\int_{0}^{t}C_s^2 ds\,,\quad h_t=\int_{0}^{t} C_s dB_s+C_s^2 b(\nu_s) ds\,.
    \end{equation}
    As $t$ increases, $\nu_t$ becomes more log-concave, which allows to apply tools like the Brascamp-Lieb or Bakry-\'Emery inequalities. Moreover, by choosing the matrix $C_t$ appropriately, e.g. $C_t=J$, the quadratic tilt $Q_t$ can partially cancel interactions coming from $J$ in $\mu=\mu_{J,h}$. This perspective has been crucial in recent advances on Ising-type models \cite{alaoui2023fast, caputo2025factorizations, chen2025rapidthreshold, AKV24}.
	
	 For Theorem \ref{thm:main}, our main use of stochastic localization is to reduce the study of the measure $\mu_{\beta,u,J,h}$ from \eqref{eq:antifer} to the study of an anti-ferromagnetic rank one model \begin{equation} \label{eq:antfernoJ}
		\mu_{\beta,u,h'}(x)=(Z_{\beta,u,h'})^{-1}\exp\left(-\frac{\beta}{n}\langle u,x\rangle^2 + \langle h',x\rangle\right),
	\end{equation}
    where $Z_{\beta,u,h'}$ is the normalizing constant called the \textit{partition function}. Here, $\mu_{\beta,u,h'}$ no longer includes $J$, and $h'$ is a different external field. For this reduction to be valid, one must control the dissipation of entropy along the localization process. As shown in \cite[Section 3]{chen2025localization}, the key quantity to control this dissipation is the operator norm $\|\Cov(\mu_{\beta,u,J,h})\|_{\op}$. We obtain the the required bounds through the Hubbard–Stratonovich transform, in a manner similar to \cite{baurschmidt2019simple}, together with a sharp operator norm bounds on $\|\Cov(\mu_{\beta,u,h'})\|_{\op}$ derived from a Gaussian approximation (see Theorem~\ref{thm:corr:approx}).
	
	\medskip\noindent\textbf{Coordinate-by-coordinate localization:}
    In contrast to the bounded matrix $J$ in Theorem \ref{thm:main} the remaining rank one matrix $-\frac{\beta}{n}uu^T$ can have a much larger operator norm of order $n^{1-\eps}$. Applying the previous stochastic localization scheme used to eliminate $J$ would, as becomes clear from the proof, introduce a multiplicative factor $e^{-\beta}$ in the MLSI constant. Since we allow $\beta$ to scale with $n$ such bounds can lead to exponential sub-optimality (see Section \ref{sec:related} for some examples in previous works). 
    
    For the negative rank-one model $\mu_{\beta,u,h'}$, we therefore employ a different, and arguably simpler, \emph{coordinate-by-coordinate localization} scheme. Let $\sigma$ be a uniformly random permutation on $[n]$, independent of $X\sim \mu$. The coordinate-by-coordinate localization is the measure-valued martingale process defined by
    $$\mu_t:= \mathrm{Law}\left(X|X_{\sigma(1)},\dots,X_{\sigma(t)}\right).$$
    Clearly $\mu_0 = \mu$ and $\mu_n$ is a (random) Dirac measure. Moreover, the conditional law $\mu_{n-1}|X$ is closely related to the transition kernel of Glauber dynamics. In this sense, the coordinate-by-coordinate localization generates Glauber dynamics. To demonstrate its usefulness, let us explain how to use the coordinate-by-coordinate localization to derive a spectral gap for $\mu$. Since we are working on $\{-1,1\}^n$ we can write $\mu_{t+1}$ as a linear tilt of $\mu_t$,

    \begin{equation} \label{eq:lineartilt}
        \frac{\mu_{t+1}(x)}{\mu_t(x)} = 1+ \langle x-b(\mu_t),Z_t\rangle,
    \end{equation}
    for an appropriate vector $Z_t$. This is the discrete analog of the infinitesimal tilt appearing in \eqref{eq:SL:SDE}. Given a test function $\vphi:\{-1,+1\}^n\to\R$, the key quantity controlling the spectral gap is the dispersion of variance along the process,
    \[
    \frac{\E\left[\mathrm{Var}_{\mu_{t+1}}(\vphi)|\mu_t\right]}{\mathrm{Var}_{\mu_{t}}(\vphi)}\geq1 - \frac{1}{n-t}\|\Cor(\mu_t)\|_{\op}\,,
    \]
    where the inequality follows from \eqref{eq:lineartilt} (see \cite[Claim 3.3]{chen2025localization}). Here, for a measure $\nu$ with covariance matrix $\Sigma$,  $\Cor(\nu)=\diag(\Sigma)^{-1/2}\Sigma \diag(\Sigma)^{-1/2}$  denotes the correlation matrix. Thus, as long as $\|\Cor(\mu_t)\|_{\op}$ is small enough for all $t\geq 0$, the localization process approximately preserves variance, yielding a spectral gap on $\mu$. In the high-dimensional expanders framework, this condition is also known as \emph{spectral independence}, known to imply a bound on the spectral gap \cite{anari2024spectral}.

    For MLSI, one must control entropy rather than variance. Entropy dissipation can be analyzed in a similar, though technically more delicate, manner, and again reduces to bounding correlation matrices. The main difference is that we require bounds on the correlation matrix of $\mu_t$ under all linear tilts, i.e., for all possible external fields. See \cite[Section 3]{chen2025localization} for the connection to \emph{entropic independence} \cite{anari2022entropic}.

    A central challenge in this approach is that it is not enough to simply bound $\Cor(\mu_{\beta,u,h})$, but it requires to obtain sharp estimate $\|\Cor(\mu_{\beta,u,h})\|_{\op} = 1 +o(1)$. The main technical contribution of this paper is to obtain such bounds for $\mu_{\beta,u,h}$ for all tilts $h\in \R^n$. When $\beta \ll n$, we derive a highly accurate structural description of the correlation matrix (see Theorem~\ref{thm:corr:approx} for the full statement)
   \begin{equation}\label{eq:informalcovariance}
	 	\sup_{h\in \R^n}\left\|\Cor(\mu_{\be,u,h})-\left(I_n-\frac{\al_{\star}}{n}w_{\star}w_{\star}^{\sT}\right)\right\|_{\op}\leq \frac{C_{\eps}\be \log(1+\beta)^{C_{\eps}}}{n}\,,
	\end{equation}
    where $w_{\star}\in \R^n$ and $\alpha_\star=\frac{2\beta}{1+2\beta\|w_\star\|^2/n}>0$ are explicit functions of $(\beta,u,h)$, and $C_{\eps}>0$ is a constant that only depends on $\eps>0$. Since $I_n-\frac{\al_{\star}}{n}w_{\star}w_{\star}^{\sT}\preceq I_n$, the condition $\beta\leq n^{1-\eps}$ implies that the error is $o(1)$. Because covariance and correlation matrices are conjugate, this simultaneously provides a near-optimal operator norm description of $\Cov(\mu_{\beta,u,h})$.
	\subsubsection{A heuristic derivation of the covariance matrix} \label{sec:heuristic:cov:approx}
  Recall the partition function $Z_{\beta,u,h}$ from~\eqref{eq:antfernoJ}. Using the identity $\Cov(\mu_{\beta,u,h})=\nabla_h^2 \log Z_{\beta,u,h}$, we now give a heuristic derivation of the matrix $I_n-\frac{\al_{\star}}{n}w_{\star}w_{\star}^{\sT}$ appearing in Eq.~\eqref{eq:informalcovariance}. Expanding $Z_{\beta,u,h}$ according to the magnetization $m=\langle u,x\rangle /n$, we obtain
	\begin{equation}\label{eq:partition:reweight}
		\begin{split}
			Z_{\beta, u,h}
			&=\sum_{m} \exp\left(-n\beta m^2\right)\sum_{x\in \{\pm 1\}^{n}:\frac{\langle u,x\rangle}{n}=m}\exp\left(\langle h,x\rangle \right)\\
			&=\sum_{m} \exp\Big(n\Big(-\beta m^2+F(\lambda)-\lambda m\Big)\Big)\sum_{x\in \{\pm 1\}^{n}:\frac{\langle u,x\rangle}{n}=m}\prod_{k=1}^{n} \frac{\exp\left((h_k +\lambda u_k) x_k\right)}{2\cosh(h_k+\lambda u_k)}\,,
		\end{split}
	\end{equation}
	where $\lambda\in \R$ is a Lagrangian multiplier and the convex function $F(\cdot)\equiv F_{u,h}(\cdot)$ is defined by
	\begin{equation}\label{eq:def:F}
		F(\lambda)=\log 2+\frac{1}{n}\sum_{i=1}^{n}\log\cosh(h_i+\lambda u_i)\,.
	\end{equation}
    Observe that the final sum in \eqref{eq:partition:reweight} equals 
    \begin{equation}\label{eq:heuristics:tilted:law}
    \P\bigg(\sum_{k=1}^{n}u_k \xi_k=nm\bigg),\quad\textnormal{where}\quad \xi_k\in \{-1,+1\}~~\textnormal{are independent}~~\textnormal{and}~~\E \xi_k=\tanh(h_k+\lambda u_k)\,.
    \end{equation}
 Suppose that for each $m\in [-1,1]$, we take $\lambda = \lambda_m:=\argmin_{\lambda}\{F(\lambda)-\lambda m\}$. This choice corresponds to usual exponential tilting used in large deviations analysis~\cite{DZ10}. Then, 
	$$m = \frac{\partial F(\lambda_m)}{\partial\lambda} = \frac{1}{n}\sum\limits_{k=1}^n u_k\tanh(h_k + \lambda_m u_k) = \frac{1}{n}\sum\limits_{k=1}^n u_k\EE[\xi_k].$$
	Hence, the event $\{\sum\limits_{k=1}^n u_k \xi_k = nm\}$ can be approximated using the local central limit theorem (CLT). Thus, at least at an exponential scale it is negligible, and we expect that
	\[
	\begin{split}
		\frac{1}{n}\log Z_{\beta,u,h}
		=\sup_{m\in [-1,1]}\inf_{\lambda\in \R}\left\{-\beta m^2 +F(\lambda)-\lambda m\right\}+o(1)=\inf_{\lambda\in \R} \left\{F(\lambda)+\frac{\lambda^2}{4\beta}\right\}+o(1)\,,
	\end{split}
	\]
	where $\sup$ and $\inf$ can be interchanged since $(m,\lambda)\mapsto -\beta m^2+F(\lambda)-\lambda m$ is concave-convex. Let $\lambda^\star_h= \argmin_{\lambda\in \R} \{F(\lambda)+\lambda^2/(4\beta)\}$. Since $\Cov(\mu_{\beta,u,h}) = \nabla^2_h\log Z_{\beta,u,h}$, taking Hessian w.r.t. $h$ on both sides leads to the heuristic approximation
	\[
	\begin{split}
		\Cov(\mu_{\beta,u,h})
		&\approx n\cdot \nabla^2_h\left[\inf_{\lambda\in \R}\left\{F(\lambda)+\frac{\lambda^2}{4\beta}\right\}\right]=n\cdot \nabla_h \left[\nabla_h F(\la^\star_h)\right]\\
		&= n\cdot \left(\nabla^2_h F(\lambda^\star_h)+\nabla_h\partial_{\la} F(\la^\star_h)\cdot \left(\nabla_h \lambda^\star_h\right)^{\sT}\right)\\
		&=n\cdot \left(\nabla^2_h F(\lambda^\star_h)-\left(\partial_{\lambda}^2 F(\lambda^\star_h)+\frac{1}{2\beta}\right)^{-1} \cdot \left(\nabla_h\partial_{\la} F(\la^\star_h)\right)\left(\nabla_h\partial_{\la} F(\la^\star_h)\right)^{\sT}\right)\,,
	\end{split}
	\]
	where the first equality is by envelope theorem, and the final inequality holds by taking gradient with respect to $h$ in $\la^\star_h/(2\be) +\partial_{\la}F(\la^\star_h)=0$. Since $\Cor(\mu_{\beta,u,h})$ is conjugate to $\Cov(\mu_{\beta,u,h})$, this heuristic suggests $\Cor(\mu_{\beta,u,h})$ is approximately $I_n-c_{\star}c_{\star}^{\sT}$ for an explicit vector $c_{\star}\in \R^n$.
	
	However, we emphasize that this derivation above is purely heuristic. It does not give any quantitative control on the error, nor does it guarantee uniformity over $h\in \R^n$, even for $\beta=O(1)$. The main challenge is to quantify the error uniformly over $h, u\in \R^n$, where the most difficult regime is when $\beta$ is polynomially large, close to order $n$.
    
    One obstacle is that the argument relies on a different tilting parameter $\lambda_m$ for each value of $m$, which is commonly used in large deviations theory~\cite{DZ10}, but makes the distribution of $(\xi_k)_{k\leq n }$ dependent on $m$.  In our setting, however, available local CLT estimates are simply insufficient to yield the uniform in $h\in \R^n$ bounds required for the estimate~\eqref{eq:informalcovariance}. More refined approximations such as Edgeworth expansions (cf.~\cite{dolgopyat22edgeworth}) could in principle provide sharper control, but they are also not tractable since we need to accommodate essentially arbitrary external fields $h\in \R^n$ and $m\in [-1,1]$.
    
    To circumvent this issue, we will instead choose a \textit{single global tilt} $\lambda$ corresponding to a magnetization value $m_\star$, defined as the solution to a fixed point equation (see Lemma \ref{lem:elementary}). The upshot of this choice is that it allows us to replace the local CLT by a functional CLT, for which we can establish surprisingly tight estimates. See Remark~\ref{rmk:single:CLT} for further discussion. 
  
    \subsubsection{Normal approximation of the magnetization}
    \label{subsec:intro:normal}
    With the choice of the tilt $\lambda=\lambda_\star$, the cumulant generating function $\Psi(s):=\log \E[\exp(sM)]$, where $M=\langle u,X\rangle/n$ is the magnetization for $X\sim \mu_{\beta,u,h}$, admits the following representation (see Lemma~\ref{lem:express:Psi})
    \begin{equation}\label{eq:CGF}
    \Psi(s) = sm_\star+ \log\Bigg(\frac{\E\Big[\exp\Big(\frac{s}{\sqrt{n}}Z-\beta Z^2\Big)\Big]}{\E\Big[\exp\Big(-\beta Z^2\Big)\Big]}\Bigg),
    \end{equation}
    where $Z = \frac{1}{\sqrt{n}}\sum_{k=1}^{n}u_k (\xi_k-\E\xi_k)$ and $\xi_k\in \{-1,+1\}$ are as in \eqref{eq:heuristics:tilted:law} with $\lambda=\lambda_\star$. It turns out that $\Cov(\mu_{\beta,u,h})$ can be written in terms of cavity versions of $\Phi(s)$ (see Lemma~\ref{lem:cov:express}), and controlling $\Psi'(s), \Psi''(s)$ for $s=O(\beta)$ is a key step in approximating $\Cov(\mu_{\beta,u,h})$. These derivatives involve quantities of the form 
    \[
    \E\left[Z^K\exp\left(\frac{s}{\sqrt{n}}Z-\beta Z^2\right)\right]\,,~~K=0,1,2.
    \]
    As $n\to \infty$ we expect $Z$ to be approximately Gaussian, and if $Z\sim \mathcal{N}(0,\sigma^2)$, then the expectations above admit closed form formulas. The main point is to make this approximation quantitative since we need uniformity in $h\in \R^n$.
    
  To illustrate, consider $K=2$ and assume for simplicity that $\Var(Z)\geq 1/2$, although we'll ultimately have to treat the full range $\Var(Z)\in [0,1]$, including cases such as $\Var(Z)=n^{-1/2}$. Since the function $f(x)= x^2 e^{-\beta x^2}$ satisfies $\int |f'(x)|dx=O(1/\beta)$, a standard application of the Berry-Esseen inequality gives an approximation error of $1/(\sqrt{n}\beta)$. However, it turns out that the correlation approximation in~\eqref{eq:informalcovariance} requires accuracy of much finer scale $1/(n\beta^{3/2})$, necessitating bounds beyond the standard one.
  
  We achieve this improvement by applying Stein's method. Because the required error bounds are very fine --- on the order $1/(n\beta ^{3/2})$ for $K=2$ --- the argument must be tailored to the specific structure of the moment functionals $x \to x^K e^{-\beta x^2}$. Our proof exploits algebraic relations among these moments, their derivatives, and the corresponding Poisson equation solutions. This structure allows us to implement Stein's method in an inductive way, where we induct both on $K$ and on the approximation error itself to obtain the following bound (see Theorem \ref{thm:mainapprox} for a more formal and complete statement). 
    \begin{theorem} \label{thm:informalCLT}
    With the same notation as above, and assuming $\Var(Z)\geq 1/2$ and that $s$ is small enough, there exists a Gaussian random variable $G$ such that
    $$\left|\E\left[Z^2\exp\left(\frac{s}{\sqrt{n}}Z-\beta Z^2\right)\right] -\E\left[G^2\exp\left(\frac{s}{\sqrt{n}}G-\beta G^2\right)\right]\right| \lesssim \frac{1}{n\beta^{\frac{3}{2}}}.$$
    \end{theorem}
    Observe that Theorem \ref{thm:informalCLT} improves on the standard Berry-Esseen rates both in the dependence on $n$ and the dependence on $\beta$. In fact, even in much more favorable settings, when $Z$ has an absolutely continuous and symmetric distribution, two conditions which are known to dramatically improve the rate of convergence yet are not satisfied above, the rate $1/(n\beta^{3/2})$ still does not follow from the accelerated Berry-Esseen bounds~\cite{johnston2023fast}. Thus, Theorem \ref{thm:informalCLT} captures a unique property of the function $x\mapsto x^2 e^{-\beta x^2}$ at hand.  In the course of proving our main estimate~\eqref{eq:informalcovariance}, we shall require the full power of Theorem~\ref{thm:informalCLT}. Any lesser bound would result in a worse final result in Theorem \ref{thm:main}. 
    
     Due to its connection to the quadratic exponential tilts appearing in the localization framework, such as in \eqref{eq:SL}, we believe that Theorem \ref{thm:informalCLT} and its proof could be of independent interest.
	
	\subsection{Further related literature and discussion} \label{sec:related}
	There are numerous works on sampling from Ising measures of the form \eqref{eq:antifer}. Bauerschmidt, Bodineau, and Dagallier~\cite{BBD25} study the Ising model on random 
	$d$-regular graphs as in Corollary \ref{cor:randomdregular}, but consider Kawasaki dynamics rather than Glauber dynamics. Kawasaki dynamics swaps two neighboring spins at each step and thus preserves magnetization. It can be viewed as an analog of Glauber dynamics for a fixed-magnetization slice of the hypercube with neighbor updates. This restriction can accelerate mixing by avoiding bottlenecks created by single-spin flips. In the anti-ferromagnetic regime, \cite{BBD25} shows that the Kawasaki dynamics mix in time $O_{d,h}(n\log^6(n))$, when $\beta < 1/(16\sqrt{d-1})$, although the constant in the big $O$ deteriorates with $d$, as well as for large fields $h$. They also ask whether Kawasaki dynamics on $G_{\reg}(n,d)$ becomes slow for $\beta>\arctanh(1/\sqrt{d-1})\sim 1/\sqrt{d-1}$, or at least for $\beta$ sufficiently large. Proposition~\ref{prop:slow:mixing:ER} answers this question affirmatively up to a universal constant, showing slow mixing once $\beta>C/\sqrt{d}$.

	A different Markov chain, the polarized walk, was introduced by Anari, Koehler, and Vuong~\cite{AKV24}. One of the main results~\cite[Theorem 86]{AKV24} establishes $O(n\log (n))$ mixing time for the polarized walk for Ising measures~\eqref{eq:antifer} where $u = {\bf 1}$, i.e., with an anti-ferromagnetic Curie-Weiss component. Crucially, they obtain bounds which are independent from the parameter $\beta$, for any choice of $\beta$, at the cost of using the more complicated polarized walk. They also raise the question of understanding the behavior of Glauber dynamics on these model and show that for anti-ferromagnetic Ising on random $d$-regular graphs, Glauber mixes in time $O(e^{e^{c\beta d}}n\log(n))$, in the same regime of Corollary \ref{cor:randomdregular}. 
	The doubly exponential dependence on $\beta d$ was later improved to a single exponential $O(e^{\beta d} n\log(n))$ up to $\beta <1/(8\sqrt{d-1})$ by Chen, Chen, Chen, Yin, and Zhang~\cite{CCCYZ25}, who employed a trickle-down theorem, inspired from the theory of high-dimensional expanders, along with the localization framework to refine the approach of \cite{AKV24}.

With respect to these two results, Theorem \ref{thm:main} essentially completes the picture concerning Glauber dynamics, by establishing uniform bounds on the mixing time within the admissible regime of $\beta\leq n^{1-\eps}$. Conceptually, the key methodological distinction is our use of coordinate-by-coordinate localization, whereas \cite{AKV24} employs stochastic localization and \cite{CCCYZ25} relies on negative field localization. Also, the results of \cite{AKV24, CCCYZ25} for $u={\bf 1}$ rely heavily on the theory of log-concave polynomials. As mentioned before, since negative association fail if the coordinates have $u$ have mixed signs as seen in Proposition~\ref{prop:cov:approx:entry}, we do not expect log-concavity based tools to extend for general $u\in \R^n$.

We also note that~\cite{AKV24} introduced a refined spectral width condition that yields quantitative improvements in many settings. Their approach analyzes the expected covariance matrix along stochastic localization trajectories. Since \eqref{eq:informalcovariance} provides an explicit covariance approximation, this framework could in principle be extended to accommodate a negative spectral outlier, potentially improving the constant $1/8$ in Corollary~\ref{cor:dregular}. We do not pursue this direction here.
	
	For more general interaction structures, Koehler, Lee, and Risteski~\cite{KLR22} propose a rejection sampling approach that can handle interaction matrices with an arbitrary number of negative eigenvalues $-\lambda_1,\dots, -\lambda_k$. Their algorithm samples in time $O(n^{d_{+}+1} \exp(c\sum\limits_{i=1}^k\lambda_i))$, where $d_{+}$ is the number of eigenvalues of the interaction matrix above $1-1/c$ for $c>1$. This again demonstrates the exponential dependence on the size of negative outliers as in \cite{CCCYZ25}. In the setting of Theorem \ref{thm:main}, this corresponds to allowing $k$ distinct outliers $u_1,\dots,u_k$, extending beyond the single outlier studied here, albeit the $\exp(c\sum_{i=1}^{k}\lambda_i)$ cost. We believe the localization framework combined with the Gaussian approximation approach developed in this work is well-suited to study the fast mixing of Glauber dynamics in the multi-outlier setting as well, although achieving near-optimal range of $\beta$ in this case is likely to require significant additional technical work.
    
	\paragraph{Organization}The rest of the paper is organized as follows. In Section \ref{sec:mainproof}, we utilize the localization framework and prove Theorem \ref{thm:main} while assuming the approximation estimate~\eqref{eq:informalcovariance} (or Theorem~\ref{thm:corr:approx}). Section~\ref{sec:covapprox} is devoted to the proof of \eqref{eq:informalcovariance}. One of the key components is the normal approximation from Theorem \ref{thm:informalCLT}, which we prove in Section~\ref{sec:CLT}. Finally, in Section \ref{sec:app} we provide proofs for the various corollaries of Theorem \ref{thm:main}, as well as mixing time lower bounds for anti-ferromagnetic models.
\paragraph{Acknowledgements} We thank Elchanan Mossel and Allan Sly for helpful discussions. D.M. was partially supported by the Brian and Tiffinie Pang Faculty Fellowship. Y.S. thanks Korean Institute for Advanced Study, where part of this work was carried out.
	\paragraph{Notation} We use  $C>0$ to denote an absolute constant whose value may change from line to line. The notation $C_{\eps}>0$ refers to a constant depending only on $\eps>0$. For quantities $f=f_{n,u,h,\eps}$ and $g=g_{n,u,h,\eps}$, we write $f=O_{\eps}(g)$ if there exists a constant $C_{\eps}>0$ such that $|f|\leq C_{\eps}g$. For a symmetric matrix $A\in \R^{n\times n}$, we let $\|A\|_{\op}:=\sup_{\|x\|_2\leq 1}\|Ax\|_2$ denote the operator norm of $A$.
	\section{Proof of Theorem~\ref{thm:main}} \label{sec:mainproof}
	As explained in the introduction, our proof of Theorem \ref{thm:main} has two distinct components.
	\begin{itemize}
		\item We first eliminate the matrix $J$ from the measure $\mu_{\beta,u,J,h}$ in \eqref{eq:antifer} leaving us with a measure of the form $\mu_{\beta,u,h'}$ as in \eqref{eq:antfernoJ}.
		\item We then estimate the possible values for $\Cov(\mu_{\beta,u,h'})$ and employ the localization framework to bound the MLSI constant.
	\end{itemize}
	
	The first step is handled in Subsection \ref{sec:elimnateJ}, where we apply the Hubbard–Stratonovich transform to obtain conditional bounds on $\mathrm{Cov}(\mu_{\beta,u,J,h})$. We then combine these bounds with the stochastic localization process to bound $\rho_{\mathrm{LS}}\left(\mu_{\beta,u,J,h}\right)$ in terms of $\rho_{\mathrm{LS}}\left(\mu_{\beta,u,h'}\right)$. 
	Next, we address the second part in Subsection \ref{sec:removingCW}, assuming appropriate bounds on $\Cov(\mu_{\beta,u,h'})$. Obtaining these bounds, which is the crux of the proof, is more technically involved, and we defer it to Section \ref{sec:covapprox}.
	\subsection{From general models to anti-ferromagnetic rank-one models} \label{sec:elimnateJ}
	We first introduce our main tool for this step, which is a consequence of the Hubbard–Stratonovich transform. This result is a straightforward generalization of \cite[Lemma 5.2]{chen2025localization}, which was based on the decomposition obtained by \cite{baurschmidt2019simple}. We include the proof for completeness.
	\begin{lemma}\label{lem:hubbard:stratonovich}
		Let $J_0$ be a symmetric $n \times n$ matrix, and suppose that there exists a constant $L>0$, for which
		\[
		\sup_{h\in \R^n} \left\|\Cov(\mu_{J_0,h})\right\|_{\op}\leq L.
		\]
		Then, for all positive semi-definite $J\in \R^{n\times n}$ with $\|J\|_\mathrm{op}<\frac{1}{2L}$, it holds that
		\[
		\sup_{h\in \R^n} \left\|\Cov(\mu_{J+J_0,h})\right\|_{\op}\leq \frac{L}{1-2L\|J\|_{\op}}.
		\]
	\end{lemma}
	For the proof, we need the following well-known result which follows from the Bakry-Em\'ery theorem \cite[Corrolary 4.8.2]{bakry2014analysis}, or the Brascamp-Lieb inequality~\cite [Theorem 4.1]{brascamb76}.
	\begin{lemma}\label{lem:bl}
		Let $f:\RR^n\to \RR$ be a smooth function with $ \alpha I_n \preceq \nabla^2 f(x)$, for some $\alpha > 0$, and every $x \in \RR^n$. If $Y$ is a random vector with density $\frac{e^{-f(x)}}{\int\limits_{\RR^n}e^{-f(x)}dx}$, then for any continuously differentiable $h:\RR^n\to \RR^n$,
		\[
		\left\|\Cov(h(Y))\right\|_{\op}\leq \frac{1}{\alpha} \E\left[\left\|\nabla h(Y)\right\|_{\op}^2\right].
		\]
	\end{lemma}
	
	\begin{proof}[Proof of Lemma~\ref{lem:hubbard:stratonovich}]
		By adding a multiple of the identity matrix, we assume with no loss of generality that $0\prec J\prec \frac{1}{2C}I_n$. By Gaussian integration (also known as the Hubbard-Stratonovich transform), for any $x\in \{-1,1\}^n$,
		\[
		\exp(\langle x, Jx\rangle)=C(J)\int_{\R^n}e^{-\frac{1}{2}\langle y, (2J)^{-1} y\rangle +\langle x,y\rangle } dy,
		\]
		where $C(J)$ does not depend on $x\in \{+1,-1\}^n$. It follows that 
		\begin{equation}\label{eq:decomp:measure}
			\begin{split}
				\mu_{J+J_0,h}(x)
				&\propto \int_{\R^n}e^{-\frac{1}{2}\langle y, (2J)^{-1} y\rangle +\langle x,h+y\rangle +\langle x, J_0 x\rangle } dy\\
				&\propto \int_{\R^n}e^{-f(y)}\cdot \mu_{J_0,h+y}(x) dy,
			\end{split}
		\end{equation}
		where $f(y)=\frac{1}{2}\langle y, (2J)^{-1}y\rangle -\log Z_{J_0, h+y}$, and $Z_{J_0, h+y} = \sum\limits_{x \in \{-1,1\}^n} e^{\langle x, J_0x\rangle + \langle x, h+y\rangle}$.
		A key observation is that $f$ is strictly convex under the stated assumptions. In fact,
		\[
		\nabla^2 f(y) = (2J)^{-1}-\Cov(\mu_{J_0,h+y})\succ \bigg(\frac{1}{\|2J\|_{\op}}-L\bigg)I_n.
		\]
		Let $\nu$ denote the probability measure $\nu(dy)\propto \exp(-f(y))dy$. By the decomposition~\eqref{eq:decomp:measure}, we can draw $X\sim \mu_{J+J_0,h}$ by first drawing $Y\sim \nu$ and then drawing $X\sim \mu_{J_0,h+Y}$ conditional on $Y$. The law of total variance then gives
		\[
		\Cov(\mu_{J+J_0,h})=\Cov_{\nu}\left(\E[\mu_{J_0,h+Y}|Y]\right)+\E_{\nu}\left[\Cov(\mu_{J_0,h+Y}|Y)\right]\preceq \Cov_{\nu}\left(\E[\mu_{J_0,h+Y}|Y]\right)+L\cdot I_n.
		\]
		Consider the function $h(y)=\E[\mu_{J_0,h+y}]$. Note that $\|\nabla h(y)\|_{\op}=\|\Cov(\mu_{J_0,h+y})\|_{\op}\leq L$. An application of Lemma~\ref{lem:bl} with $\alpha = 1/(2\|J\|_{\op})-L$ gives
		\[
		\left\|\Cov_{\nu}\left(\E[\mu_{J_0,h+Y}|Y]\right)\right\|_{\op}+L = \left\|\Cov\left(h(Y)\right)\right\|_{\op}+L \leq \frac{2L^2\|J\|_{\op}}{1-2L\|J\|_{\op}}+L=\frac{L}{1-2L\|J\|_{\op}}.
		\]
		Combining the two displays above completes the proof.
	\end{proof}\vspace{-\baselineskip}
	The upshot of Lemma \ref{lem:hubbard:stratonovich} is that we can now bound the MLSI constant of $\mu_{\beta,u,J,h}$ from \eqref{eq:antifer} in terms of the MLSI constant of the, a-priori simpler, $\mu_{\beta,u,h}$. in which we set $J = 0$. We achieve this by combining the covariance bound afforded by Lemma \ref{lem:hubbard:stratonovich} with the stochastic localization process allowing to remove quadratic terms, as in \cite[Theorem 5.1]{chen2025localization}. 
	\begin{corollary} \label{cor:stepI}
		Let $\mu_{\beta,u,J,h}$ be as in \eqref{eq:antifer} and let $\mu_{\beta,u,h'}$ be as in \eqref{eq:antfernoJ}. Suppose that $\|J\|_{\op} < \frac{1}{2}$, and suppose that
		\begin{equation} \label{eq:covbound}
			\sup\limits_{h' \in \RR^n} \|\mathrm{Cov}(\mu_{\beta,u,h'})\|_{\mathrm{op}} = 1 + \alpha,
		\end{equation}
        where $\alpha>0$ satisfies $\alpha < \frac{1}{2\|J\|_{\op}} - 1$.
		Then, for large enough $n$,
			$$\rho_{\mathrm{LS}}(\mu_{\beta,u,J,h}) \geq \overline{\rho_{\mathrm{LS}}} \exp\left(-\int\limits_0^1\frac{2(1+\alpha)\|J\|_{\op}}{1- 2(1+\alpha)(1-\lambda)\|J\|_{\op}}d\lambda\right),$$   
		where 	
		\begin{equation} \label{eq:uniformLSbound}
			\overline{\rho_{\mathrm{LS}}} := \inf\limits_{h' \in \RR^n} \rho_{\mathrm{LS}}(\mu_{\beta,u,h'}).
		\end{equation}
	\end{corollary}
	\begin{proof}
		For $\lambda \in [0,1],$ and $h' \in \RR^n$ define the measures $\mu_{\lambda,h'}$ by 
		$$\frac{d\mu_{\lambda,h'}}{d\mu_{\beta,u,J,h}}(x) = \exp\left(-\lambda \langle x, Jx\rangle + \langle h',x\rangle\right).$$
		Concretely,
		$$\mu_{\lambda,h'}(x) = \exp\left(-\beta \left(\langle x, u\rangle \right)^2 + (1-\lambda)\langle x, Jx\rangle + \langle h + h',x\rangle\right),$$
		and so we note that $\mu_{1,h'} = \mu_{\beta,u,h+h'}$ and $\mu_{0,0} = \mu_{\beta,u,J,h}$. Set now
		$$\alpha_\lambda := \sup\limits_{h' \in \RR^n} \|\Cov(\mu_{\lambda,h'})\|_{\mathrm{op}},$$
		and observe that
		$$\overline{\rho_{\mathrm{LS}}} = \inf\limits_{h' \in \RR^n} \rho_{\mathrm{LS}}(\mu_{1,h'}).$$
		With this notation, according to \cite[Theorem 5.1]{chen2025localization}, we have the following bound
		\begin{equation} \label{eq:stocloc}
			\rho_{\mathrm{LS}}(\mu_{\beta,u,J,h})\geq \overline{\rho_{\mathrm{LS}}} \exp\left(-2\|J\|_{\mathrm{\op}}\int\limits_0^1\alpha_\lambda d\lambda\right).
		\end{equation}
		Thus we now need to bound $\alpha_\lambda$ for $\lambda \in [0,1]$.
		Taking $J_0 = uu^T$ in Lemma \ref{lem:hubbard:stratonovich} combined with the assumption
		$$\sup\limits_{h' \in \RR^n} \|\mathrm{Cov}(\mu_{\beta,u,h'})\|_{\mathrm{op}} = 1 + \alpha < \frac{1}{2\|J\|_{\op}},$$
		shows that, since $\|J\|_{\mathrm{op}} < \frac{1}{2(1+\alpha)}$, for large enough $n$,
		\begin{align*}
			\alpha_\lambda = \sup\limits_{h' \in \RR^n} \|\mathrm{Cov}(\mu_{\lambda,h'})\|_{\mathrm{op}} &= \sup\limits_{h' \in \RR^n} \|\mathrm{Cov}(\mu_{(1-\lambda)J + J_0,h'})\|_{\mathrm{op}} \leq \frac{1+\alpha}{1-2(1+\alpha)(1-\lambda)\|J\|_{\mathrm{op}}}.
		\end{align*}
		Combined with \eqref{eq:stocloc} we arrive at
        	$$\rho_{\mathrm{LS}}(\mu_{\beta,u,J,h}) \geq \overline{\rho_{\mathrm{LS}}} \exp\left(-\int\limits_0^1\frac{2(1+\alpha)\|J\|_{\op}}{1- 2(1+\alpha)(1-\lambda)\|J\|_{\op}}d\lambda\right),$$   
		which is the desired bound.
	\end{proof}\vspace{-\baselineskip}
	\subsection{Handling anti-ferromagnetic rank-one models} \label{sec:removingCW}
	In light of Corollary \ref{cor:stepI}, our next step is to bound $\rho_{\mathrm{LS}}(\mu_{\beta,u,h'})$ for an arbitrary external field $h' \in \RR^n$. This step is facilitated by bounding the correlation matrix as in \eqref{eq:informalcovariance}, which also implies the condition \eqref{eq:covbound}. However, remark that naively applying the bound from \eqref{eq:stocloc} with $J$ replaced by $-\beta uu^T$ would necessarily result in the sub-optimal and $\beta$-dependent bound
	$\rho_{\mathrm{LS}}(\mu_{\beta,u,h'}) \geq e^{-2\beta},$ leading to potential exponential bounds, and missing the main estimates of Theorem \ref{thm:main}.
	Instead, we replace the use of the stochastic localization process, which results in the bound \eqref{eq:stocloc}, and replace it with an analysis based on coordinate-by-coordinate localization, exemplified by the following result from \cite{chen2025localization}. For completeness, we outline the main steps of the proof.
	\begin{proposition} \label{prop:coor-by-coord}
		Let $\mu$ be a measure on $\{-1,1\}^n$ and for $h' \in \RR^n$ define the tilted measure $$\frac{d\mathcal{T}_{h'}\mu}{d\mu}(x)\propto \exp(\langle h,x\rangle).$$ 
		Suppose that for some $\delta \in (0,1)$, the following uniform bound on the corelation matrices holds:
		\begin{equation} \label{eq:corbound}
			\sup\limits_{h'\in \RR^n} \|\Cor\left(\mathcal{T}_{h'}\mu\right)\|_{\mathrm{op}} \leq 1+\delta.
		\end{equation}
		Then $\rho_{\mathrm{LS}}(\mu) \geq \left(n^{1+\delta}\Gamma(1-\delta)\right)^{-1}$, where $\Gamma$ stands for the Gamma function.
	\end{proposition}
	\begin{proof}
		By \cite[Lemma 3.17]{chen2025localization} the condition \eqref{eq:corbound} implies that $\mu$ is $(1+\delta)$-entropically stable with respect to the binary entropy function (see \cite[Definition 3.10]{chen2025localization}). 
		According to \cite[Lemma 3.16]{chen2025localization}, since $\mu$ is entropically stable, we have the following bound for single-site Glauber dynamics,
		$$\rho_{\mathrm{LS}}(\mu) \geq \prod\limits_{k=2}^n \left(1-\frac{1+\delta}{k}\right).$$
		The product has a closed form using the  functional equation of the Gamma function, $\Gamma(x)x = \Gamma(1+x),$
		$$\prod\limits_{k=2}^n \left(1-\frac{1+\delta}{k}\right) =  \prod\limits_{k=2}^n \frac{k-1-\delta}{k} = \frac{1}{n!}\prod\limits_{k=2}^n (k-1-\delta) = \frac{\Gamma(n-\delta)}{n!\Gamma(1-\delta)}.$$
		Finally, applying Wendel's inequality for Gamma functions,
		$$\frac{\Gamma(n-\delta)}{n!} = \frac{1}{n}\frac{\Gamma(n-\delta)}{\Gamma(n)} \geq \frac{1}{n}\frac{1}{(n-\delta)^\delta}\geq \frac{1}{n^{1+\delta}}.$$
	\end{proof}\vspace{-\baselineskip}
    In Section \ref{sec:covapprox} we will establish an appropriate bound on $\Cor(\mathcal{T}_{h'}\mu)$ validating the condition in \eqref{eq:corbound}. This bound is contained in Theorem \ref{thm:corr:approx}, which in particular gives the estimate \eqref{eq:informalcovariance}, that we restate as
    $$\sup_{h\in \R^n}\left\|\Cor(\mu_{\be,u,h})\right\|_{\op}\leq 1+ \frac{C_{\eps}\be \log(1+\beta)^{C_{\eps}}}{n}\, ,$$
    where $C_\eps$ depends only on $\eps$.
	Combining Proposition \ref{prop:coor-by-coord} with the above estimate we can now bound $\rho_{\mathrm{LS}}(\mu_{\beta,u,h'})$.
	\begin{corollary} \label{cor:CWLSbound}
		Let $\mu_{\beta,u,h}$ be defined as in \eqref{eq:antfernoJ}, where $h \in \RR^n$ is arbitrary, $u\in \RR^n$ satisfies $\|u\|_{\infty} \leq 1$ and $0\leq -\beta\leq n^{1-\eps}$ for some $\eps > 0$. Then,
		$$\rho_{\mathrm{LS}}(\mu_{\beta,u,h}) \geq \left(1-\frac{C_\eps}{n^{\frac{\eps}{4}}}\right)n^{-1},$$
        where $C_\eps>0$ depends only on $\eps.$
	\end{corollary}
	\begin{proof}
		To apply Proposition \ref{prop:coor-by-coord} we first need to verify the condition \eqref{eq:corbound}. For $h' \in \RR^n$, by Theorem \ref{thm:corr:approx}, and in particular Eq. \eqref{eq:informalcovariance}, along with the triangle inequality,
		\begin{equation} \label{eq:corrdecomp}
			\|\Cor(\mathcal{T}_{h'} \mu_{\beta,u,h})\|_{\mathrm{op}} = \|\Cor(\mu_{\beta,u,h+h'})\|_{\mathrm{op}} \leq \|I_n-\frac{\alpha_*}{n}w_*w_*^T\|_{\mathrm{op}} + \frac{C_\eps\beta \log(1+\beta)^{C_\eps}}{n},
		\end{equation}
		where $C_\eps$ depends only on $\eps$. Moreover, since $\alpha_* = \frac{2\beta}{1+\frac{2\beta}{n}\|w_*\|_2^2}$, we get 
		$$0\preceq \frac{\alpha_*}{n}w_*w_*^T \preceq  \frac{2\beta \|w_*\|_2^2}{n(1+\frac{2\beta}{n}\|w_*\|_2^2)}I_n\preceq I_n \implies \|I_n-\frac{\alpha_*}{n}w_*w_*^T\|_{\mathrm{op}} \leq 1.$$ 
        Since $\beta \leq n^{1-\eps}$, we get that
		and \eqref{eq:corbound} holds with $\delta = \frac{C_\eps\beta \log(1+\beta)^{C_\eps}}{n}$, and that for large enough $n$, $\delta < 1/2$. So, Proposition \ref{prop:coor-by-coord} implies
		$$\rho_{\mathrm{LS}}(\mu_{\beta,u,h}) \geq \frac{1}{\Gamma(1-\delta)n^{\delta}}n^{-1}.$$
        To simplify the expression, for large $n$, we bound $\delta \leq \frac{C'_\eps}{n^{\frac{\eps}{2}}}$, for some other constant $C'_\eps>0$.
        With a Taylor expansion $\Gamma(1-\delta) \leq 1 + 10\delta \leq 1+\frac{10C_\eps'}{n^{\frac{\eps}{2}}}$, for small enough $\delta$. Similarly, we can bound $n^{\delta}\leq 1 + \frac{C'_\eps}{n^{\frac{\eps}{4}}}$.
		Combining the bounds yields the result
	\end{proof}\vspace{-\baselineskip}
	\subsection{Proof of Theorem \ref{thm:main}}
	Given Corollary \ref{cor:stepI} and Corollary \ref{cor:CWLSbound}, the proof of Theorem \ref{thm:main} is now immediate.
	\begin{proof}[Proof of Theorem \ref{thm:main}]
		By Corollary \ref{cor:CWLSbound} if $\overline{\rho_{\mathrm{LS}}}$ is defined as in \eqref{eq:uniformLSbound}, then $\overline{\rho_{\mathrm{LS}}} \geq \left(1-\frac{C_\eps}{n^{\frac{\eps}{4}}}\right)n^{-1}$. This in particular implies the $\sup\limits_{h' \in \RR^n}\|\Cov(\mu_{\beta,u,h'})\|_{\op} = \left(1+\frac{C_\eps}{n^{\frac{\eps}{4}}}\right)$, as required in \eqref{eq:covbound}. Alternatively, since $\mu_{\beta,u,h'}$ is supported on $\{-1,1\}^n$
		$$\Cov(\mu_{\beta,u,h'}) \leq \Cor(\mu_{\beta,u,h'})\leq 1+\frac{C_\eps}{n^{\frac{\eps}{4}}},$$
		where the second inequality follows from the bound in \eqref{eq:corrdecomp}.
		Applying the above two bounds in Corollary \ref{cor:stepI} yields with $\alpha = \alpha_{n,\eps}  = \frac{C_\eps}{n^{\frac{\eps}{4}}}$ for large enough $n$,
        \begin{align*}
            \rho_{\mathrm{LS}}(\mu_{\beta,u,J,h)}) 	&\geq (1-\alpha_{n,\eps}) \exp\left(-\int\limits_0^1\frac{2(1+\alpha_{n,\eps})\|J\|_{\op}}{1- 2(1+\alpha_{n,\eps})(1-\lambda)\|J\|_{\op}}d\lambda\right)n^{-1}\\
            &= (1+o(1))\exp\left(-\int\limits_0^1\frac{2\|J\|_{\op}}{1- 2(1-\lambda)\|J\|_{\op}}d\lambda\right)n^{-1} = (1+o(1))(1-2\|J\|_{\op})n^{-1}, 
        \end{align*}
       where we've used that for every $\eps > 0$, $\alpha_{n,\eps} \xrightarrow{n\to \infty} 0,$ and the applied dominated convergence in the integral.
	\end{proof}\vspace{-\baselineskip}

	\section{Covariance approximation of negative rank one Ising models} \label{sec:covapprox}
	In this section, we prove the correlation approximation~\eqref{eq:informalcovariance}. Throughout, we fix $\beta\leq n^{1-\eps}$, $u\in \R^{n}$ such that $\|u\|_{\infty}\leq 1$, and consider arbitrary $h\in \R^{n}$. Define $m_\star\equiv m_{\star,u,h}\in [-1,1]$ as the unique solution to the fixed point equation
	\begin{equation}\label{eq:def:m:star}
		m_{\star}=\frac{1}{n}\sum_{i=1}^{n}u_i\tanh(h_i-2\beta m_{\star} u_i)\,.
	\end{equation}
	The existence and uniqueness of $m_\star$ is established in Lemma~\ref{lem:elementary} below. Next, define the vectors $w_{\star,h} \equiv w_{\star} = (w_{\star,i})_{i\leq n} \in \R^n, v_{\star,h} \equiv v_\star = (v_{\star,i})_{i\leq n}\in \R^n$, and the scalar $\al_{\star,h}\equiv\al_\star$ by
	\begin{equation}\label{eq:def:v:w:alpha}
	v_{\star,i}=\sech^2(h_i-2\beta m_\star u_i)\,,\quad\quad w_{\star,i}=u_i\sech(h_i-2\beta m_{\star} u_i)\,,\quad\quad\al_{\star}=\frac{2\be}{1+\frac{2\be}{n}\sum_{i=1}^{n}w_{\star,i}^2}\,,
	\end{equation}
	where we suppress the dependence on $h$ for simplicity. The main result of this section is as follows.
	\begin{theorem}\label{thm:corr:approx}
		Let $0<\beta\leq n^{1-\eps}$ for some $\eps>0$.  There exists a constant $C_\eps > 0$, depending only on $\eps$, such that for $n\geq C_{\eps}$,
		\[
		\sup_{h\in \R^n}\left\|\Cov(\mu_{\be,u,h})-\left(\diag(v_{\star})-\frac{\al_{\star}}{n}\left(v_{\star}\odot u\right)\left(v_{\star}\odot u\right)^{\sT}\right)\right\|_{\op}\leq \frac{C_{\eps}\be \log(1+\beta)^{C_{\eps}}}{n},
		\]
		where $\odot$ denotes the Hadamard (entry-wise) product. Moreover,
		\[
		\sup_{h\in \R^n}\left\|\Cor(\mu_{\be,u,h})-\left(I_n-\frac{\al_{\star}}{n}w_{\star}w_{\star}^{\sT}\right)\right\|_{\op}\leq \frac{C_{\eps}\be \log(1+\beta)^{C_{\eps}}}{n}\,.
		\]
	\end{theorem}	
	The proof of Theorem~\ref{thm:corr:approx} is based on the following result. 
	\begin{proposition}\label{prop:cov:approx:entry}
		Uniformly over $0<\beta \leq n^{1-\eps}$, $u\in \R^{n}$ such that $\|u\|_{\infty}\leq 1$, $h\in \R^n$, and $i\neq j \in [n]$, we have for $X=(X_i)_{i\leq n} \sim \mu_{\beta,u,h}$ that
		\begin{align}
			\Var(X_i)
			&=v_{\star,i} \left(1+O_{\eps}\left(\frac{\beta\log(1+\beta)^{C_{\eps}}}{n}\right)\right)\,,\\
			\Cov(X_i,X_j)
			&=-\frac{\alpha_{\star}}{n}u_iu_j v_{\star,i}v_{\star,j} \left(1+O_{\eps}\left(\frac{\beta\log(1+\beta)^{C_{\eps}}}{n}\right)\right)\,.
		\end{align}
	\end{proposition}
	\begin{proof}[Proof of Theorem~\ref{thm:corr:approx}]
		We establish the approximation for the correlation matrix.
	By bounding the off-diagonal contribution in the correlation matrix by its Frobenius norm, we have
		\[
		\begin{split}
			&\left\|\Cor(\mu_{\be,u,h})-\left(I_n-\frac{\al_{\star}}{n}w_{\star}w_{\star}^{\sT}\right)\right\|_{\op}\\
			&\leq \left\|\diag\left(\frac{\alpha_{\star}}{n}w_{\star,i}^2\right)_{i\in [n]}\right\|_{\op}+\left(\sum_{i\neq j}\left(\Cor(X_i,X_j)+\frac{\alpha_{\star}}{n}w_{\star,i}w_{\star,j}\right)^2\right)^{1/2}.
		\end{split}
		\]
		Since $|w_{\star,i}|\leq 1$ and $\alpha_{\star}\leq 2\beta$, the first term is at most $2\beta/n$. For the second term, Proposition~\ref{prop:cov:approx:entry} yields
		\[
		\begin{split}
			\left(\sum_{i\neq j}\left(\Cor(X_i,X_j)+\frac{\alpha_{\star}}{n}w_{\star,i}w_{\star,j}\right)^2\right)^{1/2}
			&\leq	\frac{\alpha_{\star}}{n}\left(\sum_{i\neq j} w_{\star,i}^2 w_{\star,j}^2\right)^{1/2}\cdot O_{\eps}\left(\frac{\beta\log(1+\beta)^{C_{\eps}}}{n}\right)\\
			&\leq
			\frac{2\beta\|w_{\star}\|^2}{n+2\beta \|w_{\star}\|^2}\cdot O_{\eps}\left(\frac{\beta\log(1+\beta)^{C_{\eps}}}{n}\right).
		\end{split}
		\]
        Since $2\beta \|w_{\star}\|^2/(n+2\beta \|w_\star\|^2)\leq 1$, this proves the correlation approximation. The covariance approximation follows by the same argument where we use $v_{\star,i}\leq 1$ and $\sum_{i=1}^{n}u_i^2 v_{\star,i}=\|w_{\star}\|^2$.
	\end{proof}\vspace{-\baselineskip}
	
	\subsection{Fundamental Identities}
	The rest of this section is devoted to the proof of Proposition~\ref{prop:cov:approx:entry}. For $X\sim \mu_{\beta,u,h}$, define the \textit{magnetization} $M:=\langle u, X \rangle/n$. We relate $\Var(X_i)$ and $\Cov(X_i,X_j)$ with the law of the magnetization in cavity systems where one or two coordinates are removed. For $I\subset [n]$, let
	\[
	u^{(I)}=(u_i)_{i\in [n]\setminus I}\in \R^{n-|I|},\quad h^{(I)}=(h_i)_{i\in [n]\setminus I}\in \R^{n-|I|}.
	\]
	For $I=\{i\}$ or $I=\{i,j\}$, we simply write $(i)$ and $(ij)$. Denote the cavity magnetization
	\[
	M^{(I)}=\frac{\langle u^{(I)},X^{(I)}\rangle}{n}, \quad\textnormal{where}\quad X^{(I)}\sim \mu_{\frac{n-|I|}{n}\beta,u^{(I)},h^{(I)}}\,.
	\]
	Note that the normalization for $ \mu_{\frac{n-|I|}{n}\beta,u^{(I)},h^{(I)}}$ is taken so that 
    \[
    \mu_{\frac{n-|I|}{n}\beta,u^{(I)},h^{(I)}}(x)\propto \exp\Big(-\frac{\beta}{n}\big(\langle u^{(I)}, x\rangle\big)^2+\langle h^{(I)},x\rangle\Big),\quad x\in \{+1,-1\}^{n-|I|},
    \]
    and that the normalization for $M^{(I)}$ remains $1/n$ rather than $1/(n-|I|)$. Denote the cumulant generating function of $M^{(I)}$ by
	\[
	\Psi^{(I)}(s)=\log \E \left[\exp\left(sM^{(I)}\right)\right],\quad s\in \R.
	\]
	
	\begin{lemma}\label{lem:cov:express}
		For $u=(u_i)_{i\leq n}\in \R^n$ and $h=(h_i)_{i\leq n}\in \R^n$, let $X=(X_i)_{i\leq n}\sim \mu_{\beta,u,h}$. For $i\in [n]$, 
		\[
		\Var(X_i)=\frac{4\exp\big(\Psi^{(i)}(2\beta u_i)+\Psi^{(i)}(-2\beta u_i)\big)}{\Big(\exp\big(h_i+\Psi^{(i)}(-2\beta u_i)\big)+\exp\big(-h_i+\Psi^{(i)}(2\beta u_i)\big)\Big)^2}.
		\]
		For $i\neq j \in [n]$, we have
		\[
		\Cov(X_i,X_j)=\frac{4\Xi}{\Theta^2},
		\]
		where $\Xi\equiv \Xi_{ij}$ and $\Theta\equiv\Theta_{ij}$ are defined as follows. Letting $p_{ij}=2\beta(u_i+u_j)$ and $q_{ij}=2\beta(u_i-u_j)$,
		\[
		\begin{split}
			&\Xi=\exp\left(-\frac{4\beta}{n}u_iu_j+\Psi^{(ij)}(p_{ij})+\Psi^{(ij)}(-p_{ij})\right)-\exp\left(\frac{4\beta}{n}u_iu_j+\Psi^{(ij)}(q_{ij})+\Psi^{(ij)}(-q_{ij})\right),\\
			&\Theta = \exp\left(-\frac{2\beta}{n}u_iu_j+h_i+h_j+\Psi^{(ij)}(-p_{ij})\right)+\exp\left(-\frac{2\beta}{n}u_iu_j-h_i-h_j+\Psi^{(ij)}(p_{ij})\right)\\
			&\quad\quad+\exp\left(\frac{2\beta}{n}u_iu_j+h_i-h_j+\Psi^{(ij)}(-q_{ij})\right)+\exp\left(\frac{2\beta}{n}u_iu_j-h_i+h_j+\Psi^{(ij)}(q_{ij})\right).
		\end{split}
		\]
	\end{lemma}
	\begin{proof}
		We only prove the second identity since the first identity follows from a simpler argument. For $x_i,x_j\in \{\pm 1\}$, $\P(X_i=x_i, X_j=x_j)$ is proportional to
		\[
		\begin{split}
			&\sum_{(x_k)_{k\notin\{i,j\}}\in \{\pm 1\}^{n-2}}\exp\bigg(-\frac{\beta}{n}\Big(u_ix_i+u_jx_j+\sum_{k\notin\{i,j\}}u_k x_k\Big)^2+h_ix_i+h_jx_j+\sum_{k\notin\{i,j\}}h_k x_k\bigg)\\
			&\propto e^{-\frac{2\beta}{n}u_iu_j x_i x_j +h_i x_i+h_j x_j}\sum_{x_0\in \{\pm 1\}^{n-2}}e^{-\frac{2\beta}{n}(u_ix_i+u_j x_j)\sum_{k\notin\{i,j\}}u_k x_k}\cdot \mu_{\frac{n-2}{n}\beta, u^{(ij)},h^{(ij)}}(x_0)\\
			&=e^{-\frac{2\beta}{n}u_iu_j x_i x_j +h_i x_i+h_j x_j} \E \left[e^{-2\beta(u_ix_i+u_jx_j)M^{(ij)}}\right]=:\Gamma_{x_i,x_j},
		\end{split}
		\]
		where we used $x_i^2=x_j^2=1$ in the first proportionality. Since $\Cov(X_i,X_j)=\frac{4(\Gamma_{1,1}\Gamma_{-1,-1}-\Gamma_{1,-1}\Gamma_{-1,1})}{(\Gamma_{1,1}+\Gamma_{-1,-1}+\Gamma_{1,-1}+\Gamma_{-1,1})^2}$, the conclusion follows.
	\end{proof}\vspace{-\baselineskip}
	Using Lemma~\ref{lem:cov:express}, we approximate $\Var(X_i)$ and $\Cov(X_i,X_j)$ by estimating the derivatives of $\Psi^{(I)}(\cdot)$ for $|I|\leq 2$. The most delicate contribution is the term $\Xi$, which requires second order control on $\Psi^{(ij)}$, in contrast with the first order estimates being sufficient for the remaining terms. Furthermore, while the factors $\frac{2\beta}{n}u_iu_j$ in $\Theta$ are negligible, the corresponding terms $\pm\frac{4\beta}{n}u_iu_j$ in $\Xi$ are essential and must be retained to obtain the correct leading-order covariance.
    
	We begin by introducing the global tilt parameter $\lambda_\star$ that we'll use. Recall the convex function $F(\lambda)\equiv F_{u,h}(\lambda)$ in \eqref{eq:def:F}. Define $\la_\star\equiv \la_\star(u,h)$ and $m_\star\equiv m_\star(u,h)$ by
	\begin{equation}  \label{eq:lambdastardef}
		\la_\star =\argmin_{\la\in \R}\left\{F(\la)+\frac{\la^2}{4\beta}\right\},\quad m_\star\equiv -\frac{\la_\star}{2\beta}.
	\end{equation}
	\begin{lemma}\label{lem:elementary}
		For $u,h\in \R^n$ such that $u\neq 0$, let $(\la_\star,m_\star)$ be defined as above. Then, $m_\star\in [-1,1]$ is the unique solution to the fixed point equation~\eqref{eq:def:m:star}, and $m_\star= F'(\la_\star)$ holds.
	\end{lemma}
	\begin{proof}
		Note that $\la\to F(\la)+\la^2/(4\beta)$ is strictly convex which tends to infinity as $|\la|\to\infty$ since $F(\cdot)$ is bounded below and convex. Thus, $\la=\la_\star$ is the unique solution to $F'(\la)+\la/(2\beta)=0$. Note that the fixed point for $m=m_\star$ in equation~\eqref{eq:def:m:star} can be written as $m=F'(-2\beta m)$. This shows that $\la$ solves $F'(\la)+\la/(2\beta)=0$ if and only if $m=-\la/(2\beta)$ solves~\eqref{eq:def:m:star}, which concludes the proof.
	\end{proof}\vspace{-\baselineskip}
	Consider independent random variables $\xi_i\in \{\pm 1\}$ distributed as
	\[
	\P(\xi_i=1)=1-\P(\xi_i=-1)=\frac{\exp(h_i+\la_\star u_i)}{2\cosh(h_i+\la_\star u_i)}.
	\]
	By definition of $(\la^\star,m^\star)$,
	\[
	m_\star=\frac{1}{n}\sum_{i=1}^{n} u_i\E\left[\xi_i\right],\qquad F''(\la_\star)=\frac{1}{n}\sum_{i=1}^{n}u_i^2 \Var(\xi_i)= \frac{1}{n}\sum_{i=1}^{n}w_{\star,i}^2,
	\]
    where the last equality holds since $w_{\star,i}=u_i^2\sech^2(h_i+\la_\star u_i)$ (cf. \eqref{eq:def:v:w:alpha}). For $I\subset [n]$, define 
	\[
	Z\equiv Z^{(I)}=\frac{1}{\sqrt{n}}\left(\sum_{i\in [n]\setminus I}u_i \xi_i -nm_\star \right).
	\]
	Although $Z$ doesn't have mean zero in general, we have $\E [Z]= O(1/\sqrt{n})$ for $|I|=O(1)$, which will be sufficient for our purposes. The following lemma provides a key representation of $\Psi^{(I)}$, equivalently, the law of $M^{(I)}$.
	
	\begin{lemma}\label{lem:express:Psi}
		For $I\subset[n]$, let $Z\equiv Z^{(I)}$ as defined above. For $s\in \R$, we have
		\[
		\Psi^{(I)}(s)=sm_\star+ \log\Bigg(\frac{\E\Big[\exp\big(\frac{s}{\sqrt{n}}Z-\beta Z^2\big)\Big]}{\E\Big[\exp\big(-\beta Z^2\big)\Big]}\Bigg).
		\]
	\end{lemma}
	\begin{proof}
		Fix $u,h\in \R^n$ and $I\subset [n]$. Define 
		\[
		F^{(I)}(\la)\equiv F^{(I)}_{u^{(I)},h^{(I)}}(\la)=\log 2+\frac{1}{n}\sum_{i\in [n]\setminus I}\log\cosh(h_i+\la u_i).
		\]
		The partition function $Z_{\frac{n-|I|}{n}\beta, u^{(I)},h^{(I)}}$ for the measure $\mu_{\frac{n-|I|}{n}\beta, u^{(S)},h^{(S)}}$ can be expressed analogously as in $I=\emptyset$ case in~\eqref{eq:partition:reweight}. Taking $\la=\la_\star$ (not depending on $I$), we have
		\[
		Z_{\frac{n-|I|}{n}\beta, u^{(I)},h^{(I)}}=\sum_{m}\exp\bigg(n\Big(-\beta m^2+F^{(I)}(\la_\star)-\la_\star m \Big)\bigg)\P\bigg(\sum_{i\in [n]\setminus I} u_i \xi_i = nm\bigg).
		\]
        Since $\la_\star = -2\beta m_\star$, we have $-\beta m^2 -\la_\star m=-\beta(m-m_\star)^2+\beta m_\star^2$. Hence,
		\[
		Z_{\frac{n-|I|}{n}\beta, u^{(I)},h^{(I)}}=e^{nF^{(I)}(\la^\star)+\beta n m_\star^2}\sum_m \exp\Big(-\beta n(m-m_\star)^2\Big)\P\bigg(\sum_{i\in [n]\setminus I} u_i \xi_i = nm\bigg).
		\]
		A similar computation with adding the factor $e^{sm}$ shows that
		\[
		\begin{split}
			&\sum_{x\in \{\pm 1\}^{n-|I|}}\exp\bigg(s\frac{\langle u^{(I)},x\rangle}{n}-\frac{\beta}{n}\big(\langle u^{(I)}, x\rangle\big)^2+\langle h^{(I)},x\rangle\bigg)\\
			&=e^{nF^{(I)}(\la^\star)+\beta n m_\star^2}\sum_m \exp\Big(sm-\beta n (m-m_\star)^2\Big)\P\bigg(\sum_{i\in [n]\setminus I} u_i \xi_i = nm\bigg).
		\end{split}
		\]
		Recalling that $\frac{1}{n}\sum_{i\in [n]\setminus I}u_i \xi_i=m_\star+\frac{1}{\sqrt{n}}Z$, taking the ratios of the previous two displays yields the claimed expression for $\Psi^{(I)}(s)\equiv \log \E [\exp(sM^{(I)})]$.
	\end{proof}\vspace{-\baselineskip}
	\begin{remark}\label{rmk:single:CLT}
	    A major benefit of Lemma~\ref{lem:express:Psi} is that it expresses the law of the cavity magnetization $M^{(I)}$ as a perturbation of $m_\star$ through a single random variable $Z$ subject to a negative quadratic tilt. Consequently, derivative estimates for $\Psi^{(I)}$ reduce to controlling the moments of $Z$ under this tilted law. This is in sharp contrast with the local CLT approach discussed in Section~\ref{sec:heuristic:cov:approx}, which requires analyzing a different tilted law for each possible value of the magnetization $m$.
	\end{remark}
	
	\subsection{Cumulant generating function of the magnetization}
	As a consequence of Lemma~\ref{lem:express:Psi}, the derivatives $\Psi^{(I)}(\cdot)$ can be expressed in terms of the expectations under a quadratic tilt. In a more general setting, given a bounded random variable $W$ and $t\in \R$, introduce the tilted measure $\P^{W}_{t,\beta}$ where by for an event $A$,
	\[
	\P^{W}_{t,\beta}(A)= \frac{\E\left[\one_{A}\exp\left(tW-\beta W^2\right)\right]}{\E\left[\exp\left(tW-\beta W^2\right)\right]},
	\]
	and denote the corresponding expectation (resp. variance) by $\E^{W}_{t,\beta}$ (resp. $\Var^{W}_{t,\beta}$).
    The usefulness of this representation comes from a strong \emph{universality} result. When $W$ is a sum of independent bounded random variables, the tilted measure $\P^{W}_{t,\beta}$ depends only weakly on the distribution of $W$. In particular, its moments can be estimated to high accuracy if we replace $W$ with a Gaussian matching its variance. This is formalized in the next theorem.
	\begin{theorem}\label{thm:gaussian:approx:reweighted}
		Let $0<\beta \leq n^{1-\eps}$ for some $\eps \in (0,1)$. Consider independent random variables $(Y_i)_{i\leq n}$ such that $\E Y_i=0$ and $|Y_i|\leq C$ almost surely. Let
		\begin{equation} \label{eq:rvW}
			W=\frac{1}{\sqrt{n}}\sum_{i=1}^{n}Y_i,\qquad\zeta=\Var(W).
		\end{equation}
		Then, uniformly over $t\in \R$ such that $|t|\leq C\beta/\sqrt{n}$, 
		\[
		\left|\Var^{W}_{t,\beta}\left(W\right)-\frac{\zeta}{1+2\beta\zeta}\right| \leq \frac{C_{\eps}\log(1+\beta)^{C_\eps}}{n(1+2\beta\zeta)}\,,\qquad \Big|\E^{W}_{t,\beta}[W]\Big|\leq \frac{C_{\eps}\log(1+\beta)^{C_\eps}}{\sqrt{n}}\,.
		\]
	\end{theorem}
    It is worth noting that $\Var^G_{0,\beta}(G)=\frac{\zeta}{1+2\beta\zeta}$, where $G \sim \mathcal{N}(0,\zeta).$ Thus, Theorem \ref{thm:gaussian:approx:reweighted} is a universality result allowing to approximate the tilted moments of $W$ by those of a \emph{centered} Gaussian, even when $W$ is not centered. In light of this, the precision of $\frac{\log(1+\beta)^{C_{\eps}}}{n(1+2\beta\zeta)}$ is specific to the tilted variance and does not extend to $\E_{t,\beta}^{W}[W^2]$. Indeed, if $W$ were replaced by $G$ and $\zeta=1$ for simplicity, then under the tilted measure $    \E^{G}_{t,\beta}[G^2]=\big(\frac{ t}{1+2\beta}\big)^2+ \frac{1}{1+2\beta }$. Taking $t=\beta/\sqrt{n}$ and $\beta=n^{1-\eps}$ shows that $|\E^{G}_{t,\beta}[G^2] - \E^{G}_{0,\beta}[G^2]|$ is much larger than $O(1/(n\beta))$ scale of Theorem~\ref{thm:gaussian:approx:reweighted}. Thus, the variance estimate relies on a non-trivial cancellation between $\E^{W}_{t,\beta}[W^2]$ and $\E^{W}_{t,\beta}[W]^2$, yielding an additional factor $1/\beta$.
    
    The proof of Theorem \ref{thm:gaussian:approx:reweighted} is technically involved and is based on a careful implementation of Stein's method. We therefore postpone the proof and further discussion to Section \ref{sec:CLT}. For now, we state its corollary that provides estimates on the derivatives of $\Psi^{(I)}$. For $I\subset [n]$, define $F_{I}(\cdot)\equiv F_{I,u,h}(\cdot)$ by
	\[
	F_I(\la)=F(\la)-\frac{1}{n}\sum_{i\in I} \log \cosh(h_i+\la u_i)= \log 2+\frac{1}{n}\sum_{i\in [n]\setminus I} \log \cosh(h_i+\la u_i).
	\]
	\begin{lemma}\label{lem:psi:estimates}
		Let $0<\beta\leq n^{1-\eps}$. Consider $I\subset [n], s\in \R$ such that $|I|\leq C$ and $|s|\leq C\beta$. Then, we have the estimates
		\[
		\begin{split}
			&\left|\Psi^{(I)}(s)-sm_\star\right|
			\leq\frac{C_{\eps}\beta\log(1+\beta)^{C_{\eps}}}{n}\,,\qquad
			\bigg|\left(\Psi^{(I)}\right)^\prime (s)-m_\star\bigg|
			\leq\frac{C_{\eps}\log(1+\beta)^{C_{\eps}}}{n}\,,\\
			&\qquad~~~~\bigg|\left(\Psi^{(I)}\right)^{\prime\prime}(s)-\frac{F_I''(\la_\star)}{n\big(1+2\beta F_I''(\la_\star)\big)}\bigg|
			\leq\frac{C_{\eps}\log(1+\beta)^{C_{\eps}}}{n^2\big(1+2\beta F_I''(\la_\star)\big)}.
		\end{split}
		\]
		In particular, taking $s=0$, we have
		\[
		\left|\E M^{(I)}-m_\star\right|\leq \frac{C_{\eps}\log(1+\beta)^{C_{\eps}}}{n}\,,\qquad \bigg|\Var(M^{(I)})-\frac{F_I''(\la_\star)}{n\big(1+2\beta F_I''(\la_\star)\big)}\bigg| \leq\frac{C_{\eps}\log(1+\beta)^{C_{\eps}}}{n^2\big(1+2\beta F_I''(\la_\star)\big)}.
		\]
	\end{lemma}
	\begin{proof}
		We first prove the second estimate since the first estimate follows from the second by integration. Recall the expression for $\Psi^{(I)}(s)$ in Lemma~\ref{lem:express:Psi}. Since $m_\star=F'(\la_\star)$ by Lemma~\ref{lem:elementary}, we have
		\[
		\E Z \equiv \E Z^{(I)}= -\frac{1}{\sqrt{n}}\sum_{i\in I}u_i \E[\xi_i]=:\frac{\eta}{\sqrt{n}}.
		\]
		Note that $|\eta|\leq |I|\leq C$. Also, 
		\[
		Z=\frac{\eta}{\sqrt{n}}+W, \quad\textnormal{where}\quad W=\frac{1}{\sqrt{n}}\sum_{i\in [n]\setminus I} u_i \big(\xi_i-\E[\xi_i]\big).
		\]
        Thus, $\sqrt{n}W$ is sum of $n-|I|$ bounded random variables, and $\Var(W)=F_I''(\la_\star)$. By Lemma~\ref{lem:express:Psi},
		\[
		\left(\Psi^{(I)}\right)^\prime (s)=m_\star+\frac{1}{\sqrt{n}}\cdot\frac{\E\left[Z\exp\left(\frac{s}{\sqrt{n}}Z-\beta Z^2\right)\right]}{\E\left[\exp\left(\frac{s}{\sqrt{n}}Z-\beta Z^2\right)\right]}=m_\star+\frac{1}{\sqrt{n}}\E^{W}_{t,\beta}[W],
		\]
		where $t= (s-2\eta\beta)/\sqrt{n}$, which satisfies $|t|\leq C\beta/\sqrt{n}$. Applying Theorem~\ref{thm:gaussian:approx:reweighted} yields the desired bound on $(\Psi^{(I)})'(s)$. Proceeding in the same manner for the third estimate, we have by Lemma~\ref{lem:express:Psi}
		\[
		\left(\Psi^{(I)}\right)^{\prime\prime}(s)=\frac{1}{n}\Var^{W}_{t,\beta}(W),
		\]
		and the bound follows directly from Theorem~\ref{thm:gaussian:approx:reweighted} since $\Var(W)= F_I''(\la_\star)$.
	\end{proof}\vspace{-\baselineskip}
	
	\subsection{Estimation of $\Xi, \Theta$}
	We now prove Proposition~\ref{prop:cov:approx:entry}. Introduce the error parameter
	\[
	\Upsilon = \frac{\beta \log(1+\beta)^{C_{\eps}} }{n}<\frac{1}{2},
	\]
	where $C_{\eps}>0$ denotes the constant from Lemma~\ref{lem:psi:estimates}. Since $\beta \leq n^{1-\eps}$, the inequality $\Upsilon<1/2$ holds for sufficiently large $n\geq C_{\eps}$. We henceforth assume $\Upsilon<1/2$.
	
	\begin{lemma}\label{lem:Theta}
		With the same notation of Proposition \ref{prop:cov:approx:entry}, for $X\sim \mu_{\beta,u,h}$ and $i\neq j \in [n]$, we have the estimates $\Var(X_i)=v_{\star,i}(1+O_{\eps}(\Upsilon))$, and $\Theta = 4(v_{\star,i}v_{\star,j})^{-1/2}(1+O_{\eps}(\Upsilon))$.
	\end{lemma}
	\begin{proof}
		Recall the expression for $\Var(X_i)$ in Lemma~\ref{lem:cov:express}. Using the first estimate in Lemma~\ref{lem:psi:estimates} with $I=\{i\}$ and $s\in \{\pm 2\beta u_i\}$, we have $\Psi^{(i)}(\pm 2\beta u_i)=\pm 2\beta u_i m_\star+O_{\eps}(\Upsilon)$, so $ \exp\left(\Psi^{(i)}(2\beta u_i)+\Psi^{(i)}(-2\beta u_i)\right)= 1+O_{\eps}(\Upsilon)$ and
		\[
		\exp\big(h_i+\Psi^{(i)}(-2\beta u_i)\big)+\exp\big(-h_i+\Psi^{(i)}(2\beta u_i)\big)=\big(1+O_{\eps}(\Upsilon)\big)\cdot \big(2\cosh(h_i-2\beta u_i m_\star)\big).
		\]
		Recalling $v_{\star,i}\equiv \sech^2(h_i-2\beta u_i m_\star)$, applying these estimates to the expression for $\Var(X_i)$ in Lemma~\ref{lem:cov:express} concludes the proof of the first claim. Since the terms $\pm \frac{2\beta}{n}u_iu_j$ in the definition of $\Theta$ is $O_{\eps}(\Upsilon)$, the second claim follows by the same argument where we apply Lemma~\ref{lem:psi:estimates} with $I=\{i,j\}$ and $s\in \{\pm p_{ij},\pm q_{ij}\}$.
	\end{proof}\vspace{-\baselineskip}
	It thus remains to approximate $\Xi$. To begin with, note that
	\begin{equation}\label{eq:Xi}
		\Xi=\exp\left(\frac{4\beta}{n}u_iu_j +\Psi^{(ij)}(q_{ij})+\Psi^{(ij)}(-q_{ij})\right)\left(\exp\left(-\frac{8\beta}{n}u_iu_j+\Delta_{ij}\Psi^{(ij)}\right)-1\right),
	\end{equation}
	where we recall $p_{ij}=2\beta(u_i+u_j), q_{ij}=2\beta (u_i-u_j)$, and $\Delta_{ij}\Psi^{(ij)}$ is defined by
	\[
	\Delta_{ij}\Psi^{(ij)}=\Psi^{(ij)}(p_{ij})+\Psi^{(ij)}(-p_{ij})-\Psi^{(ij)}(q_{ij})-\Psi^{(ij)}(-q_{ij}).
	\]
	This term is more delicate to approximate than $\Theta$. We note that second order Taylor approximation to $\Psi^{(ij)}$ around $0$ and applying the derivative estimates from Lemma~\ref{lem:psi:estimates} will not suffice since it could be that $u_i^2+u_j^2\gg u_i u_j$ while we need the error control in terms of $u_i u_j$ in order to prove Proposition~\ref{prop:cov:approx:entry}.
	
	\begin{lemma}\label{lem:Xi}
		We have $\Xi= -\frac{4\al_\star}{n}u_iu_j(1+O_{\eps}(\Upsilon))$.
	\end{lemma}
	\begin{proof}
		We begin by approximating $\Delta_{ij}\Psi^{(ij)}$. Rearranging terms, 
		\[
		\begin{split}
			\Delta_{ij}\Psi^{(ij)}
			&=\int_{0}^{p_{ij}-q_{ij}}\left(\Psi^{(ij)}\right)^\prime(q_{ij}+t)-\left(\Psi^{(ij)}\right)^\prime(-p_{ij}+t)dt\\
			&=\int_{0}^{p_{ij}-q_{ij}}(p_{ij}+q_{ij})\left(\Psi^{(ij)}\right)^{\prime\prime}(v_t)dt,
		\end{split}
		\]
		where for each $t$ the point $v_t$ lies between $q_{ij}+t$ and $-p_{ij}+t$ by the mean value theorem. Since $p_{ij}=2\beta(u_i+u_j)$ and $q_{ij}=2\beta(u_i-u_j)$, we have $|v_t|\leq 4\beta$. Hence, by Lemma~\ref{lem:psi:estimates}, uniformly over $t\in [0,p_{ij}-q_{ij}]$,
		\[
		\left(\Psi^{(ij)}\right)^{\prime\prime}(v_t)=\frac{F_{ij}''(\la_\star)}{n(1+2\beta F_{ij}''(\la_\star))}+O_{\eps}\left(\frac{\log(1+\beta)^{C_{\eps}}}{n^2(1+2\beta F_{ij}''(\la_\star))}\right),
		\]
		Substituting this into the integral yields
		
		\begin{equation}\label{eq:Delta:Psi:estimate}
			\Delta_{ij}\Psi^{(ij)}=\frac{16\beta^2 F_{ij}''(\la_\star)u_iu_j}{n(1+2\beta F_{ij}''(\la_\star))}+O_{\eps}\left(\frac{\beta u_i u_j}{n(1+2\beta F_{ij}''(\la_\star) )}\Upsilon\right).
		\end{equation}
		We next replace $F_{ij}''(\la_\star)$ by $F''(\la_\star)$. Since $F_{ij}''(\la_\star)=F''(\la_\star)-(w_{\star,i}^2+w_{\star,j}^2)/n$,
		\[
		\left|\frac{2\beta F_{ij}''(\la_\star)}{1+2\beta F_{ij}''(\la_\star)}-\frac{2\beta F''(\la_\star)}{1+2\beta F''(\la_\star)}\right|=\left|\frac{1}{1+2\beta F_{ij}''(\la_\star)}-\frac{1}{1+2\beta F''(\la_\star)}\right|\leq \frac{4\beta}{n(1+2\beta F''(\la_\star))}.
		\]
        In particular, we have $(1+2\beta F_{ij}''(\la_\star))^{-1}\leq C(1+2\beta F''(\la_\star))^{-1}$ for large enough $n$. Thus, the estimate~\eqref{eq:Delta:Psi:estimate} remains valid with $F_{ij}''(\la_\star)$ replaced everywhere by $F''(\la_\star)$. Substituting this estimate for $\Delta_{ij}\Psi^{(ij)}$ into~\eqref{eq:Xi}, and using the fact $\Psi^{(ij)}(q_{ij})+\Psi^{(ij)}(-q_{ij})=1+O_{\eps}(\Upsilon)$ by Lemma~\ref{lem:psi:estimates}, we obtain
		\[
		\begin{split}
			\Xi
			&=\big(1+O_{\eps}(\Upsilon)\big)\left(-\frac{8\beta}{n}u_iu_j+\frac{16\beta^2 F''(\la_\star) u_iu_j}{n(1+2\beta F''(\la_\star))}+O_{\eps}\left(\frac{\beta u_i u_j}{n(1+2\beta F_{ij}''(\la_\star) )}\Upsilon\right)\right)\\
			&=-\frac{4\al_\star}{n}u_iu_j\big(1+O_{\eps}(\Upsilon)\big),
		\end{split}
		\]
		where the last equality holds since $\al_\star=\frac{2\beta}{1+2\beta F''(\la_\star)}$. 
	\end{proof}\vspace{-\baselineskip}
	\begin{proof}[Proof of Proposition~\ref{prop:cov:approx:entry}]
		The result follows from combining Lemma~\ref{lem:cov:express}, Lemma~\ref{lem:Theta}, and Lemma~\ref{lem:Xi}.
	\end{proof}\vspace{-\baselineskip}
	
	\section{Normal approximations for moments of tilted measures} \label{sec:CLT}
	In this section, we prove Theorem \ref{thm:gaussian:approx:reweighted}. Although that theorem concerns only the first two moments of a quadratically tilted measure, the argument requires an appropriate control on \emph{all moments} of the form $\E[W^{m}e^{-\beta W^2}]$ for all $m\geq 0$ in order to carry out an induction based on Stein's method. We begin by explaining the general setting. To apply Stein's method, it is convenient to normalize $W$ to have unit variance. For $\v > 0$, let $\mathcal{S}_{\bound}^\v$ denote the family of random variables $Z$ which can be expressed as follows for some $n\geq 1$:
    \begin{equation}\label{eq:def:mathcal:S}
    Z = \frac{1}{\sqrt{\v}}\sum\limits_{i=1}^n Y_i,\qquad \mathrm{Var}\bigg(\sum\limits_{i=1}^nY_i\bigg)=\v,
    \end{equation}
    where the $Y_i$ are independent, $\E[Y_i]=0$, and $|Y_i|\leq \bound$, almost surely.    
    For $\gamma>0$ and $m\in \{0,1,2,\ldots\}$, let
    \[
    h_{m}(x) = x^me^{-\gamma x^2},
    \]
    where we suppress the dependence on $\gamma$. Our main objective is to approximate $\EE[h_{m}(Z)]$ for $Z\in \mathcal{S}_{\bound}^{\v}$ by $\EE[h_{m}(G)]$ for a standard Gaussian $G$. Since in Theorem \ref{thm:gaussian:approx:reweighted} we need to consider $t\neq 0$, potentially as large as $\beta/\sqrt{n}$, we will allow $Z$ to have a non-trivial mean. To that end, let
\[
\Delta_{m,\mu}^\v(\gamma) = \sup\limits_{Z \in \mathcal{S}^\v}\Big|\EE\big[h_{m}(Z+\mu)\big]- \E\big[h_{m}(G)\big]\Big|,~~~G\sim \mathcal{N}(0,1),\qquad \Delta_{m}^{\v}(\gamma)\equiv \Delta_{m,0}^{\v}(\gamma).
\]
    Here, we further suppress the dependence on $\bound$. Remark that the standard Gaussian $G$ remains centered, and we will prove that as long as $|\mu|\leq C/\sqrt{\v}$ and $\v >\gamma$, this addition of $\mu$ does not affect our error bounds. 	
   The connection to Theorem~\ref{thm:gaussian:approx:reweighted} is as follows. Let $W = \frac{1}{\sqrt{n}}\sum\limits_{i=1}^nY_i$ be as in \eqref{eq:rvW}. Note that $W=\sqrt{\frac{\v}{n}} Z$, and set $\gamma = \frac{\v}{n}\beta$ and $\mu = -\frac{t}{2\beta}\sqrt{\frac{n}{\v}}$. Here, $\gamma, \mu$ are taken so that $W-\frac{t}{2\beta}=\sqrt{\frac{\v}{n}}(Z+\mu)$. With these parameters, note the crucial identity
    \begin{align} \label{eq:varh}
    \Var^{W}_{t,\beta}(W)= \Var^{W-\frac{t}{2\beta}}_{0,\beta}\bigg(W-\frac{t}{2\beta}\bigg)=\frac{\v}{n}\left(\frac{\E\left[h_2(Z+\mu)\right]}{\E\left[h_0(Z+\mu)\right]}-\left(\frac{\E\left[h_1(Z+\mu)\right]}{\E\left[h_0(Z+\mu)\right]}\right)^2\right),
    \end{align}
    with a similar identity for $\E_{t,\beta}^W[W]$.
Using \eqref{eq:varh}, an upper bound on $\Delta_{m,\mu}^{\omega}(\gamma)$ for $m\in \{0,1,2\}$ translates to the approximation of $\Var_{t,\beta}^{W}(W)$.

	By standard central limit theorem arguments, it is clear that $\Delta_m^{\v}(\gamma) \xrightarrow{\v \to \infty} 0.$ However, Theorem~\ref{thm:gaussian:approx:reweighted} asks us for an effective rate of convergence, which is the content of the main result of this section. 
	\begin{theorem} \label{thm:mainapprox}
		For every $m,M \in \mathbb{N}$, there exists a constant $C_{m,M,\bound}>0$ such that, if the variance satisfies  $\v > (M+1)\bound^2$ and $\v > \gamma\bound^2$, and the mean satisfies $|\mu| \leq C\bound/ \sqrt{\v}$,  then
		$$\Delta_{m,\mu}^{\v}(\gamma) \leq C_{m,M,\bound}\left(\frac{\log(2+2\gamma)^{M}}{\sqrt{\v}(1+2\gamma)^{\frac{m+2}{2}}} +  \frac{1}{\gamma^{\frac{m}{2}}}\left(\frac{\gamma}{\v}\right)^{\frac{M}{2}}\right),$$
		when $m$ is odd, and
		$$\Delta_{m,\mu}^{\v}(\gamma) \leq C_{m,M,\bound}\left(\frac{\log(2+2\gamma)^{M}}{\v(1+2\gamma)^{\frac{m+1}{2}}} +  \frac{1}{\gamma^{\frac{m}{2}}}\left(\frac{\gamma}{\v}\right)^{\frac{M}{2}}\right),$$
		when $m$ is even.
	\end{theorem}
     The bounds in Theorem~\ref{thm:mainapprox} differ by parity of $m$, a distinction that underlies the different bounds of $\Var_{t,W}^{W}(W)$ and $\E_{t,W}^{W}[W]$ in Theorem~\ref{thm:gaussian:approx:reweighted}. As will be clear in the proof, the parameter $M\in \mathbb{N}$ serves as an induction parameter: the argument proceeds by induction on $M$ based on $M+2$'th order Taylor approximation of the solution to the Poisson equation associated with $h_m$. 

Recalling $\gamma=\beta\v/n$, we have $\gamma/\v=\beta/n$. Under the assumption $\beta\le n^{1-\eps}$, we have $\gamma/\v\leq n^{-\eps}\lesssim \v^{-\eps}$ since $\v\leq \bound^2 n$. In this case where $\gamma/\v\lesssim \v^{-\eps}$, the second term in the right-hand side above is at most $\v ^{-\eps M/2}$. Choosing $M=m+C_\eps$, we obtain the following corollary providing a more accessible bound. 

	\begin{corollary} \label{corr:epsapprox}
		With the same assumptions of Theorem \ref{thm:mainapprox}, suppose further that $\frac{\gamma}{\v} \leq \frac{1}{\v^\eps}$, for some $\eps > 0.$ Then, for a constant $C_{m,\eps,\bound} > 0$ such that if $\v \geq C_{m,\eps}$ and $|\mu|\leq C/\sqrt{\v}$, 
		$$\Delta^\v_{m,\mu}(\gamma) \leq C_{m,\eps,\bound}\frac{\log(2+2\gamma)^{C_{m,\eps,\bound}}}{\sqrt{\v}(1+2\gamma)^{\frac{m+2}{2}}},$$
		when $m$ is odd, and 
		$$\Delta^\v_{m,\mu}(\gamma) \leq C_{m,\eps,\bound}\frac{\log(2+2\gamma)^{C_{m,\eps,\bound}}}{\v(1+2\gamma)^{\frac{m+1}{2}}},$$
		when $m$ is even.
	\end{corollary}
        We make a few remarks concerning Corollary~\ref{corr:epsapprox}.
        \begin{itemize}
                \item As noted in Section~\ref{subsec:intro:normal}, a standard Berry-Esseen inequality yields $|\Delta_m(\gamma)|\lesssim 1/(\sqrt{\v}\gamma^{m/2})$. Corollary~\ref{thm:corr:approx} therefore improves this by a factor $1/\gamma$ for odd $m$, and by a factor $1/\sqrt{\v \gamma}$ for even $m$.
                 \item    For Theorem~\ref{thm:gaussian:approx:reweighted}, we only require the cases $m\in \{0,1,2\}$, but we require these bounds at full strength where any weakening (aside from log factors) would lead to a strictly smaller range of $\beta$ in Theorem~\ref{thm:main}. At the same time, the inductive structure of our proof of Theorem~\ref{thm:mainapprox} necessitates establishing bounds for \textit{all} $m\geq 0$, since each induction step at level $M$ depends on uniform control of the errors for every moments.
         
            \item  Since $\E[G^m e^{-\gamma G^2}]=C_m(1+2\gamma)^{-(m+1)/2}$ where $C_m$ is the $m$'th moment of a standard Gaussian, the bound in Corollary~\ref{corr:epsapprox} for even $m$ can be written as
        \[
        \E[Z^m e^{-\gamma Z^2}]= \Big(1+\tilde{O}_{\eps}\left(\v^{-1}\right)\Big)\cdot \E [G^m e^{-\gamma G^2}]\,,
        \]
        where $\tilde{O}_{\eps}$ hides log factors and dependence on $m$. We expect that this multiplicative error is optimal up to log factors: for instance, when $Y_1,\ldots, Y_n \sim \textnormal{Unif}\{+1,-1\}$ then Stirling's approximation already yields a relative error of order $O(1/n)$. 
            \item This improved error rate $O(\v^{-1})$ instead of $O(\v^{-1/2})$ ultimately stems from the fact that $h_{m}(x)$ is an even function for even $m$. In many settings, capitalizing on the symmetry requires, at the very least, some symmetry of the random variable $Z$. Such symmetries are unavailable here as we are only assuming $Y_i$ are bounded. Furthermore, the addition of a non-trivial mean $\mu$ amounts to shifting $h_m(\cdot)\to h_m(\cdot + \mu)$, and destroys the symmetry of $h_m$.
            
            One of the key ideas that allows us to overcome these obstacles is to boost the normal approximation through induction. At a high level, we can bound $\Delta_{m,\mu}(\gamma)$ by a linear combination of terms involving $\Delta_{m'}(\gamma)$ and $|\E[h_{m'}^{(\ell)}(G)\}|$ over a suitable range of indices $m'$ and $\ell$. $\Delta_{m'}(\gamma)$ is handled through induction, while Gaussian integration by parts yields a good bound on $|\E[h_{m'}^{(\ell)}(G)]|$.
        \end{itemize}

We first show how Theorem~\ref{thm:gaussian:approx:reweighted} follows from Corollary~\ref{corr:epsapprox}. The proof of Theorem~\ref{thm:mainapprox} is based on an inductive implementation of Stein’s method, presented in Section~\ref{sec:stein}.
    
	\begin{proof} [Proof of Theorem \ref{thm:gaussian:approx:reweighted}]
	   Recall that $\zeta =\Var(W)=\v/n$. We use Corollary~\ref{corr:epsapprox} by setting
       \[
       Z= \frac{1}{\sqrt{\zeta}} W\,,\qquad
       \gamma = \zeta \beta\,,\qquad \mu = -\frac{t}{2\beta}\sqrt{\frac{n}{\v}}\,,\qquad \bound=2.
       \]
       We first establish the bound on the variance using the identity~\eqref{eq:varh}. To begin with, let's first consider the case $\v \geq C_{\eps} \log(2+\beta)^{C_\eps}$. Since $|t| \leq C\beta/\sqrt{n}$, we have $\mu \leq C/\sqrt{\v}$, so Corollary~\ref{corr:epsapprox} applies.
       
       Using that $\E[h_1(G)] = 0$ since $h_1$ is odd, Corollary \ref{corr:epsapprox} for $m=1$ yields
        $$\left|\E\left[h_1(Z+\mu)\right]\right| \leq C_{\eps} \frac{\log(2+\beta)^{C_\eps}}{\sqrt{\v}(1+2\zeta\beta)^{\frac{3}{2}}},$$
        where we used $\gamma = \zeta\beta$. 
        Moreover, since $\EE[h_0(G)] = (1+2\gamma)^{-1/2}$, Corollary \ref{corr:epsapprox} for $m=0$ implies
        $$\left|\EE[h_0(Z+\mu)]\right| \geq \left|\EE[h_0(G)]\right| - \Delta_{0,\mu}(\gamma) \geq \frac{1}{\sqrt{1+2\zeta\beta}} - \frac{C_\eps\log(2+\beta)^{C_\eps}}{\v \sqrt{1+2\zeta\beta}} \geq \frac{1}{2}\frac{1}{\sqrt{1+2\zeta\beta}},$$
        where we used the lower bound on $\v$.
        Combining the two bounds, we deduce
        \begin{equation}\label{eq:bound:first:tilted:order}
            \frac{\omega}{n}\left(\frac{\E\left[h_1(Z+\mu)\right]}{\E\left[h_0(Z+\mu)\right]}\right)^2 \leq 4\frac{C_\eps^2\log(2+\beta)^{2C_\eps}}{n(1+2\zeta\beta)^2}.
        \end{equation}
        It remains to consider $\frac{\omega}{n}\frac{\E[h_2(Z+\mu)]}{\E[h_0(Z+\mu)]}$ term in \eqref{eq:varh}. Since $\frac{\E\left[h_2(G)\right]}{\E\left[h_0(G)\right]} =\frac{(1+2\gamma)^{-3/2}}{(1+2\gamma)^{-1/2}}= \frac{1}{(1+2\zeta\beta)}$,
       another application of Corollary~\ref{corr:epsapprox} gives
        \begin{align*}
            \left|\frac{\v}{n}\frac{\E\left[h_2(Z+\mu)\right]}{\E\left[h_0(Z+\mu)\right]} - \frac{\zeta}{1+2\beta\zeta}\right| &= \frac{\v}{n}\left|\frac{\E\left[h_2(Z+\mu)\right]}{\E\left[h_0(Z+\mu)\right]} - \frac{\E\left[h_2(G)\right]}{\E\left[h_0(G)\right]}\right|\\
            &\leq \frac{\v}{n}\frac{\Delta_{2,\mu}(\gamma)}{\left|\EE\left[h_0(Z+\mu)\right]\right|} +  \frac{\v}{n}\frac{\EE[h_2(G)]}{\EE[h_0(G)]}\frac{\Delta_{0,\mu}(\gamma)}{\left|\EE\left[h_0(Z+\mu)\right]\right|}\\
            &\leq  \frac{2\v}{n} \frac{C_\eps\log(2+\beta)^{C_\eps}\sqrt{1+2\zeta\beta}}{\v(1+2\zeta\beta)^{\frac{3}{2}}} + \frac{2\v}{n(1+2\zeta\beta)}\frac{C_\eps\log(2+\beta)^{C_\eps}}{\v}\\
            &\leq \frac{4 C_\eps\log(2+\beta)^{C_\eps}}{n(1+2\zeta\beta)}.
        \end{align*}
        This establishes the variance bound when $\v\ge C_\eps\log(2+\beta)^{C_\eps}$. Now consider $\v \leq C_{\eps} \log(2+\beta)^{C_\eps}$, equivalently $\zeta\leq \frac{C_{\eps}\log (2+\beta)^{C_{\eps}}}{n}$. Since $\beta \leq n^{1-\eps}$, we have the a priori bounds
        \[
        \frac{\zeta}{1+2\beta \zeta}\leq \frac{C_{\eps}\log(2+\beta)^{C_{\eps}}}{n}\,,\qquad
        1+2\beta\zeta\leq 1+\frac{2C_{\eps} \log(2+\beta)^{C_{\eps}}}{n^{\eps}}.
        \]
        Thus, recalling the identity~\eqref{eq:varh}, it suffices to prove $\frac{\v}{n}\frac{\E\left[h_2(Z+\mu)\right]}{\E\left[h_0(Z+\mu)\right]}\leq \frac{C_\eps\log(2+\beta)^{C_\eps}}{n}$.
        Observe that 
        $$\EE[h_2(Z + \mu)]  \leq \EE\big[(Z+\mu)^2\big]\EE\Big[e^{-\gamma(Z+\mu)^2}\Big] = \EE\left[(Z+\mu)^2\right]\EE\left[h_0(Z+\mu)\right],$$
        where the inequality is due to Chebyshev's association inequality \cite[Chapter IX]{mintronovic1993cassical}. Thus, 
        \begin{equation}\label{eq:chebyshev}
            \frac{\v}{n}\frac{\E\left[h_2(Z+\mu)\right]}{\E\left[h_0(Z+\mu)\right]} \leq \frac{\v(1+\mu^2)}{n}\leq\frac{C_\eps\log(2+\beta)^{C_\eps}}{n},
        \end{equation}
        where the last inequality holds since $|\mu|\leq C/\sqrt{\v}$ and we assumed $\v \leq C_{\eps}\log(2+\beta)^{C_{\eps}}$. This proves the desired bound on the variance.
        
        Finally, we bound the expectation $\E^{W}_{t,\beta}[W]$. First assume $\v\geq C_{\eps}\log(2+\beta)^{C_{\eps}}$. Using the estimate~\eqref{eq:bound:first:tilted:order}, 

        $$\E_{t,\beta}^W[W]=\sqrt{\frac{\v}{n}}\frac{\EE[h_1(Z+\mu)]}{\EE[h_0(Z+\mu)]} - \frac{t}{\beta} \leq\frac{2C_\eps\log(2+\beta)^{C_\eps}}{\sqrt{n}(1+2\zeta\beta)} + \frac{C}{\sqrt{n}},$$
        which establishes the desired bound. In the complementary regime $\v\geq C_{\eps}\log(2+\beta)^{C_{\eps}}$, Jensen's inequality yields
        \[
        \left|\E^{W}_{t,\beta}[W]\right|\leq \Big(\E^{W}_{t,\beta}[W^2]\Big)^{1/2}=  \bigg(\frac{\v}{n}\frac{\E\left[h_2(Z+\mu)\right]}{\E\left[h_0(Z+\mu)\right]}\bigg)^{1/2},
        \]
        so combining with \eqref{eq:chebyshev} concludes the proof.
	\end{proof}\vspace{-\baselineskip}
	
	\subsection{Approaching Theorem \ref{thm:mainapprox} using Stein's method}  \label{sec:stein}
	
  As discussed above, the proof of Theorem~\ref{thm:mainapprox} is based on a refined application of Stein’s method. This framework provides a flexible and effective way to approximate functionals of sums of independent random variables by their Gaussian counterparts, while allowing one to exploit structural features and symmetries of the involved function $h_m$. In particular, $h_m$, their derivatives, and the auxiliary functions arising from Stein’s method possess favorable algebraic and analytic properties that are central to the inductive argument establishing Theorem~\ref{thm:mainapprox}.
  
It is considerably simpler to first work in the centered case $\mu=0$ and analyze the corresponding deficit $\Delta_m^{\v}(\gamma)$. Once near-optimal bounds in Theorem~\ref{thm:mainapprox} for $\Delta_m^{\v}(\gamma)$ are established, extending the analysis to the general case $\mu\neq 0$ is relatively simple.

Throughout the proof, we assume w.l.o.g. that $\bound =1$. Indeed, the class $\mathcal{S}_\bound^{\v}$ in \eqref{eq:def:mathcal:S} is invariant under the rescaling $\bound\mapsto 1$ and $\v\mapsto \v/\bound^2$. We therefore fix $\bound =1$ and write $\mathcal{S}^{\v}\equiv \mathcal{S}_1^{\v}$ for simplicity.
	
	\subsubsection{Preliminaries on Stein's method} We begin by providing the necessary background to apply Stein's method.
	For a test function $h:\RR\to\RR$, consider the Poisson equation with unknown $\cP h$,
	$$(\cP h)'(x) -x\cP h(x) = h(x) - \EE[h(G)],$$
    where $G\sim \mathcal{N}(0,1)$.
	Integrating this identity yields
	\begin{equation} \label{eq:stein}
		\left|\EE\left[h(Z)\right] - \EE\left[h(G)\right]\right| = \left|\EE\left[(\cP h)'(Z)\right] - \EE\left[Z\cP h(Z)\right]\right|
	\end{equation}
	The solution $\cP h$ admits the following representation~\cite[Proposition 3.2.2]{nourdin2012normal}
	\begin{align} \label{eq:steinsol}
		\cP h(x) = e^{x^2/2}\int_{-\infty}^x\left(h(y) - \EE\left[h(G)\right]\right)e^{-y^2/2}dy. 
	\end{align}	
	Next, for each $i \in [n]$, define $Z_i:= Z - \frac{Y_i}{\sqrt{\v}}.$ Since $\EE\left[Y_i\right] = 0$ and $Y_i$ is independent from $Z_i$, we have that $\E[Z\cP h(Z)]$ equals
	\begin{align*}
	\frac{1}{\sqrt{\v}}\sum_i\EE\left[Y_i\cP h(Z)\right]
		&=\frac{1}{\sqrt{\v}}\sum_i \left(\EE\left[Y_i\cP h(Z)\right] - \EE\left[Y_i\cP h(Z_i)\right] - \EE\left[Y_i(\cP h)'(Z_i)(Z-Z_i)\right]\right)\\
		&\ \ \ \ + \frac{1}{\sqrt{\v}}\sum_i \EE\left[Y_i(\cP h)'(Z_i)(Z-Z_i)\right].
	\end{align*}
	Also, $Z-Z_i = \frac{Y_i}{\sqrt{\v}}$ and  $\frac{1}{\v}\sum\EE\left[Y_i^2\right] = 1,$ hence,
	\begin{align} \label{eq:taylorBasis}
		\EE\left[(\cP h)'(Z)\right] - \EE\left[Z\cP h(Z)\right] &= \frac{1}{\sqrt{\v}}\sum_i\left(\EE\left[Y_i\cP h(Z_i)\right] + \EE\left[Y_i(\cP h)'(Z_i)(Z-Z_i)\right]-\EE\left[Y_i\cP h(Z)\right]\right)\nonumber\\
		&+\ \ \  \frac{1}{\v}\sum_i \EE\left[Y_i^2\right]\left(\EE\left[(\cP h)'(Z)\right] - \EE\left[(\cP h)'(Z_i)\right]\right).
	\end{align} 
	The main point is that both sums are particularly amenable to a Taylor approximation, where a better control of higher derivatives translates to a better control of the approximation error. 
    
	\subsubsection{Properties of the functions $h_m$}
    To apply Stein’s method, we first collect several basic properties of the functions 
    \[
    h_m(x)\equiv x^m e^{-\gamma x^2},\qquad m\geq 0.
    \]
    These bounds will be used repeatedly to control Taylor remainders and derivatives appearing in the Poisson equation.
	\begin{lemma} \label{lem:hdervbounds}
		For $m,\ell \in \mathbb{N}$, $$\|h_m^{(\ell)}\|_{\infty} \leq \frac{3^\ell(m+1)^{m+1}}{\sqrt{\gamma}^{m-\ell}}.$$
	\end{lemma}
	\begin{proof}
    We proceed by induction on $\ell$. Since $h_m(x)$ attains maximum at $x=\sqrt{\frac{m}{2\gamma}}$ with value $\sqrt{\frac{m}{2\gamma}}^{m}e^{-m/2}$, the inequality holds for $\ell=0$. For the induction step, note that
		\begin{align*}
			|h^{(\ell)}_m(x)| &= \left|mh^{(\ell-1)}_{m-1}(x) - 2\gamma h^{(\ell-1)}_{m+1}(x)\right| \leq m\cdot \frac{3^{\ell-1}m^{m}}{\sqrt{\gamma}^{m-\ell}} + 2\gamma \frac{3^{\ell-1}(m+1)^{m+1}}{\sqrt{\gamma}^{m+2-\ell}}\\
			& \leq \frac{3^{\ell-1}(m+1)^{m+1}}{\sqrt{\gamma}^{m-\ell}} + \frac{2\cdot 3^{\ell-1}(m+1)^{m+1}}{\sqrt{\gamma}^{m-\ell}} = \frac{3^\ell(m+1)^{m+1}}{\sqrt{\gamma}^{m-\ell}},
		\end{align*}
	which concludes the proof.
	\end{proof}\vspace{-\baselineskip}
    \begin{lemma} \label{lem:hdervs}
		Let $m,\ell \in \mathbb{N}$. Then, there exists constants $c_{m,\ell,i}$, such that 
		$$h_m^{(\ell)} = \sum\limits_{i=-\ell }^{\ell} c_{m,\ell,i}\gamma^{\frac{\ell+i}{2}}{h_{m+i}}$$
       with the convention $h_{k}\equiv 0$ when $k<0$. 
	\end{lemma}
	\begin{proof}
		Note that $h_m'(x) = mx^{m-1}e^{-\gamma x^2} - 2\gamma x^{m+1}e^{-\gamma x^2} = mh_{m-1}(x) -2\gamma h_{m+1}(x)$ for $m\geq 1$ and $h_0'(x)=-2\gamma h_1(x)$. We can thus take $c_{m,1,-1} = m$, $c_{m,1,0} = 0$, and $c_{m,1,1} = -2$. The general case follows by induction on $\ell$ in the same manner as above.
	\end{proof}\vspace{-\baselineskip}
    We will bound $\Delta_m^\v(\gamma)$ via an iterative use of Stein’s method, which requires understanding the structure of the Stein solutions $\mathcal P h_m$.
	\begin{lemma} \label{lem:poissol}
    For any $m\geq 1$, 
    \[
    \cP h_m=\frac{m-1}{1+2\gamma} \cP h_{m-2} -\frac{1}{1+2\gamma} h_m,
    \]
    with the convention $h_{k}\equiv 0$ when $k<0$. 
	\end{lemma} \label{lem:poisdermoments}
	\begin{proof}
 Recall the representation \eqref{eq:steinsol} for $\cP h_m$. Using integration by parts
     	\begin{align*}
			e^{x^2/2}\int_{-\infty}^xy^me^{-\frac{(1+2\gamma)}{2} y^2}dy
            &= e^{x^2/2}\left(-\frac{x^{m-1}}{1+2\gamma}e^{-\frac{(1+2\gamma)}{2} x^2} + \frac{m-1}{1+2\gamma}\int_{-\infty}^xy^{m-2}e^{-\frac{(1+2\gamma)}{2} y^2}dy\right)\\
			&=\frac{m-1}{1+2\gamma}e^{x^2/2} \int_{-\infty}^x h_{m-2}(y)e^{-y^2/2}dy - \frac{1}{1+2\gamma}h_{m-1}.
		\end{align*}
       Moreover, Gaussian integration by parts gives
        \[
        \E[h_m(G)]=\E\big[h_{m-1}'(G)\big]=(m-1) \E[h_{m-2}(G)]-2\gamma \E[h_m(G)],
        \]
        hence $\E[h_m(G)]=\frac{m-1}{1+2\gamma}\E[h_{m-2}(G)]$. Substituting into \eqref{eq:steinsol}  yields
        \[
        \cP h_m = \frac{m-1}{1+2\gamma}e^{x^2/2}\int_{-\infty}^x \big(h_{m-2}(y)-\E[h_{m-2}(G)]\big)e^{-y^2/2}dy - \frac{1}{1+2\gamma}h_{m-1},
        \]
        which equals $\cP h_{m-2}- \frac{1}{1+2\gamma}h_{m-1}$ as claimed.
	\end{proof}\vspace{-\baselineskip}

    In light of Lemma~\ref{lem:poissol}, we introduce the auxiliary shifted function
	\begin{align*}
	\wt{h}_m:=h_m-\frac{m-1}{1+2\gamma} h_{m-2}\,,\qquad m\geq 1.
	\end{align*}
    By linearity of the Poisson operator $\cP$ and Lemma~\ref{lem:poissol}, we have
    \begin{equation}\label{eq:Phtilde}
        f_m:=\cP \wt{h}_m=-\frac{1}{1+2\gamma} h_{m-1}\,,\qquad m\geq 1.
    \end{equation}
     Combined with Lemma~\ref{lem:hdervs}, the Poisson solution $f_m\equiv \cP \wt{h}_m$ and its derivatives is a linear combination of $(h_{j})_{j\in \mathbb{N}}$, where the coefficients have explicit dependence on $\gamma$. This self-similar algebraic structure is a key ingredient in our inductive proof of Theorem~\ref{thm:mainapprox}. Define the shifted approximation errors 
    \[
    \wt{\Delta}^\v_m(\gamma) = \sup\limits_{Z \in \mathcal{S}^\v}\bigg|\EE\Big[\wt{h}_m(Z)\Big] - \EE\Big[\wt{h}_m(G)\Big]\bigg|.
    \]
    In the induction, we first bound $\wt{\Delta}_m^{\v}(\gamma)$ via Stein's method, and then recover the bounds on $\Delta_m^{\v}$ in terms of $\wt{\Delta}_m^{\v}(\gamma)$. The following lemma makes this step precise.
    \begin{lemma}\label{lem:wtDelta:to:Delta}
    There exist constants $(c_{m,j})_{m,j\geq 0}$, independent of $\gamma$, such that $c_{m,j}=0$ whenever $m$ and $j$ have the opposite parity, and
    \[
    h_m=\frac{c_{m,0}}{(1+2\gamma)^{\frac{m}{2}}}h_0+\sum_{j=1}^{m}\frac{c_{m,j}}{(1+2\gamma)^{\frac{m-j}{2}}}\wt{h}_j.
    \]
    As a result,
    \[
    \Delta_{m}^{\v}(\gamma)\leq \frac{|c_{m,0}|}{(1+2\gamma)^{\frac{m}{2}}}\Delta_0^{\v}(\gamma)+\sum_{j=1}^{m}\frac{|c_{m,j}|}{(1+2\gamma)^{\frac{m-j}{2}}}\wt{\Delta}_{j}^{\v}(\gamma).
    \]
    \end{lemma}
    \begin{proof}
        The decomposition follows by induction on $m$. For $m=0$, it holds trivially with $c_{0,0}=1$. Assuming the claim up to $m$, note that $h_{m+1}\equiv \wt{h}_{m+1}+\frac{m}{1+2\gamma} h_{m-1}$, so setting $c_{m+1,m+1}=1$ and $c_{m+1,j}=mc_{m-1,j}$ yields the induction step. The inequality follows directly from the triangle inequality.
    \end{proof}\vspace{-\baselineskip}
  To pass from bounds on $(\wt{\Delta}_m^{\v}(\gamma))_{m\geq 1}$ to $\Delta_m^{\v}(\gamma)$, we need an additional step of controlling $\Delta_0^{\v}(\gamma)$ separately. This is achieved by the following lemma, which links the $\Delta^{\v}_0(\gamma)$ to $\wt{\Delta}^{\v}_2(\gamma)$.
	\begin{lemma} \label{lem:zerototwotilde}
		For every $\gamma > 0$,
		$$\Delta^\v_0(\gamma) \leq\frac{1}{\sqrt{1+2\gamma}}\int_0^\gamma\sqrt{1+2t}\wt{\Delta}^\v_2(t)dt.$$	
	\end{lemma}
	\begin{proof}
		Observe that $\frac{d}{d\gamma}h_0(x) = \frac{d}{d\gamma}e^{-\gamma x^2} = -x^2e^{-\gamma x^2} = -h_2(x)$. Fix $Z\in \mathcal{S}^{\v}$ and define
		\begin{align*}
			g_0(\gamma) &= \EE\left[e^{-\gamma Z^2}\right] - \EE\left[e^{-\gamma G^2}\right],\\ 
			g_2(\gamma) &= \EE\left[Z^2e^{-\gamma Z^2}-\frac{1}{1+2\gamma}e^{-\gamma Z^2}\right] -  \EE\left[G^2e^{-\gamma G^2}-\frac{1}{1+2\gamma}e^{-\gamma G^2}\right].
		\end{align*}
		According to the observation, we have
		$g'_0(\gamma) = -g_2(\gamma)  -\frac{1}{1+2\gamma}g_0(\gamma).$
		Since $g(0) = 0$, solving this differential equation yields
		$$g_0(\gamma) = \frac{1}{\sqrt{1+2\gamma}}\int_0^\gamma -\sqrt{1+2t}g_2(t)dt.$$
		By definition, $|g_2(t)| \leq \wt{\Delta}^\v_2(t)$ holds, so by taking absolute value above, we have
		$$\left|\EE\Big[e^{-\gamma Z^2}\Big] - \EE\Big[e^{-\gamma G^2}\Big]\right| \leq \frac{1}{\sqrt{1+2\gamma}}\int_0^\gamma\sqrt{1+2t}\widetilde{\Delta}^\v_2(t)dt.$$
		Taking a supremum over all $Z \in \mathcal{S}^\v$ yields the desired bound.
	\end{proof}\vspace{-\baselineskip}
	For the induction, we also need a result that controls the deficits with respect to different variances.
	\begin{lemma} \label{lem:variancedec}
		Let $Z = \frac{1}{\sqrt{\v}}\sum\limits_{i=1}^n Y_i \in \mathcal{S}^\v$ , and for $i \in [n]$ consider $Z_i = Z - \frac{1}{\sqrt{\v}}Y_i$. Set $\v' = \v\mathrm{Var}(Z_i)$ and let $G_i \sim \mathcal{N}\mathcal(0,\frac{\v'}{\v})$. Then,
		$$\Big|\EE\left[h_m(Z_i)\right]-\EE\left[h_m(G_i)\right]\Big| \leq \left(\frac{\v'}{\v}\right)^\frac{m}{2}\Delta_m^{\v'}\left(\frac{\gamma \v'}{\v}\right).$$
	\end{lemma} 
	\begin{proof}
     Define $Z'_i=\frac{1}{\v'}\sum_{j\neq i} Y_j$ so that $Z'_i \in \mathcal{S}^{\v'}$ and $Z_i = \sqrt{\frac{\v'}{\v}}Z'_i$. Writing $h_m(x;\gamma)=x^m e^{-\gamma x^2}$, we have 
        \[
        h_m(Z_i; \gamma)=\left(\frac{\v'}{\v}\right)^{m/2}h_m\left(Z'_i\,;\,\frac{\gamma \v'}{\v}\right).
        \]
       Let $G\sim \mathcal{N}(0,1)$ and set $G_i=\sqrt{\frac{\v'}{\v}}G$. The same identity holds with $(Z_i,Z'_i)$ by $(G_i,G)$. Since $Z'_i \in \mathcal{S}^{\v'}$, taking expectations and applying the definition of $\Delta_m^{\v}(\cdot)$ yields the claim.
	\end{proof}\vspace{-\baselineskip}
	
    We will need to bound the Gaussian expectations of the derivatives of $f_m\equiv \cP \wt{h}_m$. A key feature is that these Gaussian expectations exhibit cancellations that improve the decay in $\gamma$ beyond what is suggested by Lemma~\ref{lem:hdervs}. To illustrate, consider $\E[f_2'(G)]$ for $G\sim \mathcal{N}(0,1)$. By a direct calculation,
    \[
    \E[f_2'(G)]=-\frac{1}{1+2\gamma}\E[e^{-\gamma G^2}]+\frac{2\gamma}{1+2\gamma} \E[G^2 e^{-\gamma G^2}].
    \]
   Each term is individually of order $(1+2\gamma)^{-3/2}$. However, the next lemma shows that these terms cancel at leading order, yielding the sharper bound $\E[f_2'(G)]=O((1+2\gamma)^{-5/2})$. This additional factor of $(1+2\gamma)^{-1}$ is crucial for obtaining the optimal dependence on $\gamma$ in Theorem~\ref{thm:mainapprox}.

	\begin{lemma} \label{lem:gausexpec}
		Let $m,\ell \in \mathbb{N}$, and $a \in [\frac{1}{2},1]$ . Let $G_a \sim \mathcal{N}(0,a)$, and recall $f_m\equiv \cP \wt{h}_m$. Then,
		\begin{itemize}
			\item If $m + \ell$ is even,
            \[\EE[f_m^{(\ell)}(G_a)] = 0.
            \]
			\item If $m + \ell$ is odd and $\ell$ is even
			$$\Big|\EE[f_m^{(\ell)}(G_a)]\Big| = O\left((1+2\gamma)^{-\frac{m+2}{2}}\right).$$
			\item If $m + \ell$ is odd and $\ell$ is odd
			$$\Big|\EE[f_m^{(\ell)}(G_a)]\Big| = O\left((1+2\gamma)^{-\frac{m+3}{2}}\right).$$
		\end{itemize}	
		Above, the $O$ notation hides constants which depend on $m$ and $\ell$, but not on $\gamma$ and $a$.
	\end{lemma}
	\begin{proof}
    Recall that $f_m=-\frac{1}{1+2\gamma}h_{m-1}$ by Lemma~\ref{lem:poissol}. Let $H_{\ell} = (-1)^{\ell}e^{\frac{x^2}{2a}}\frac{d^{\ell}}{dx^{\ell}} e^{-\frac{x^2}{2a}}$ denote the $\ell$'th Hermite polynomial associated with $\cN(0,a)$. Gaussian integration by parts gives
    \[
    \E[f_m^{(\ell)}(G_a)]=-\frac{1}{1+2\gamma}\E[h_{m-1}^{(\ell)}(G_a)]=-\frac{1}{1+2\gamma}\E[H_{\ell}(G_a)h_{m-1}(G_a)].
    \]
    By a Gaussian change of measure, letting $G_{\gamma}\sim \cN(0,\frac{a}{1+2a\gamma})$, we further have that
    \[
     \E[f_m^{(\ell)}(G_a)]=-\frac{1}{(1+2\gamma)\sqrt{1+2a\gamma}}\E[H_{\ell}(G_{\gamma})G_{\gamma}^{m-1}].
    \]
    Since $H_{\ell}$ is an odd (resp. even) function when $\ell$ is odd (resp. even), the integrand $x\mapsto x^{m-1}H_{\ell}(x)$ is odd whenever $m+\ell$ is even, proving the first claim.
    
    When $m+\ell$ is odd, expand $H_{\ell}(x)=\sum_{i=0}^{\ell} b_i x^i$, with coefficients $b_i$ depending only on $a$.  Since $a\in [1/2,1]$, these coefficients are uniformly bounded. Using
    \[
    \E[G_\gamma^k]=O\big((1+2\gamma)^{-k/2}\big),
    \]
    and noting that $b_0=0$ when $\ell$ is odd, yields the stated bounds.
	\end{proof}\vspace{-\baselineskip}
	
	\subsubsection{An inductive Taylor approximation scheme}
	At this stage, we have the necessary tools to implement an iterative Taylor's approximation within Stein's method. As discussed above, the main thrust of the proof will be to prove Theorem \ref{thm:mainapprox} in the centered case where $\mu = 0$. Throughout, we assume without loss of generality that $\bound =1$.
\begin{proposition}\label{prop:centeredapprox}
   For every $m,M\in \mathbb{N}$, there exists a constant $C_{m,M}>0$ such that if $\v>\max(M+1,\gamma)$, then
		$$\Delta_{m}^{\v}(\gamma) \leq C_{m,M}\left(\frac{\log(2+2\gamma)^{M}}{\sqrt{\v}(1+2\gamma)^{\frac{m+2}{2}}} +  \frac{1}{\gamma^{\frac{m-1}{2}}\sqrt{1+2\gamma}}\left(\frac{\gamma}{\v}\right)^{\frac{M}{2}}\right),$$
		when $m$ is odd, and an improvement
		$$\Delta_{m}^{\v}(\gamma) \leq C_{m,M}\left(\frac{\log(2+2\gamma)^{M}}{\v(1+2\gamma)^{\frac{m+1}{2}}} +    \frac{1  }{\gamma^{\frac{m-1}{2}}\sqrt{1+2\gamma}}\left(\frac{\gamma}{\v}\right)^{\frac{M}{2}}\right),$$
		when $m$ is even.
\end{proposition}
	\begin{proof}
    We proceed by induction on $M$. For a fixed $M$ we assume the desired claim holds for all $\v, m, \gamma$ such that $\v>\max(M+1,\gamma)$, and prove the corresponding statements with $M$ replaced by $M+1$.
    
    The proof is organized into several steps. In {\bf Step 0}, we derive a preliminary bound on $\wt{\Delta}_m^{\v}(\gamma)$ for $m\geq 1$ by applying a Taylor expansion to the Poisson solution $\cP h_m$ in the identity~\eqref{eq:taylorBasis}. We then establish the base case $M=0$ in {\bf Step 1}. For the induction step, assuming the bounds hold at level $M$, we obtain refined estimates on $\wt{\Delta}^{\v}_m(\gamma)$ in {\bf Step 2 -  4}, and finally convert these bounds into the desired estimate for $\Delta^{\v}_m(\gamma)$ at level $M+1$ in {\bf Step 5}.
    
    Throughout, $C_{m,M}, C_{m,M}'$ denote constants that only depends on $m,M$, and may change line to line. We drop the subscript $M$ (resp. $m$) if they do not depend on $M$ (resp. $m$).

		\medskip
	\noindent{\bf Step 0 -  Taylor approximation to bound $\wt{\Delta}_m$:} For $m\geq 1$, recall that we denoted $f_m=\cP \wt{h}_m$. Applying a $M+2$-order Taylor approximation to $\cP \wt{h}_m$ to the right-hand side of \eqref{eq:taylorBasis} yields 
		\begin{align} \label{eq:taylor}
			\bigg|&\EE\Big[\wt{h}_m(Z)\Big]-\EE\Big[\wt{h}_m(G)\Big]\bigg| = \Big|\EE\left[f'_m(Z)\right] -\EE\left[Zf_m(Z)\right]\Big|\nonumber\\ 
			&= \Bigg|\frac{1}{\sqrt{\v}}\sum\limits_{i=1}^n \EE\bigg[Y_i\bigg(\sum\limits_{\ell=2}^{M+1} \frac{(Z-Z_i)^\ell}{\ell!}f_m^{(\ell)}(Z_i) + \frac{(Z-Z_i)^{M+2}}{(M+2)!}f_m^{(M+2)}(T_i)\bigg)\bigg]\nonumber\\
			& \quad\quad + \frac{1}{\v}\sum\limits_{i=1}^n\EE\left[Y_i^2\right]\EE\bigg[\sum\limits_{\ell=2}^{M+1} \frac{(Z-Z_i)^{\ell-1}}{(\ell-1)!}f_m^{(\ell)}(Z_i) + \frac{(Z-Z_i)^{M+1}}{(M+1)!} f_m^{(M+2)}(T'_i)\bigg]\Bigg|,
		\end{align}		
        where $T_i$, and $T'_i$ are some random variables arising from the Lagrange form of the Taylor remainder.
		Since $Z-Z_i = \frac{Y_i}{\sqrt{\v}}$,  and $|Y_i| \leq 1$ almost surely,
		\begin{align*}
			\Bigg|\frac{1}{\sqrt{\v}}&\sum\limits_{i=1}^n\left\{\EE\left[\frac{Y_i(Z-Z_i)^{M+2}}{(M+2)!}f_m^{(M+2)}(T_i)\right]+ \frac{1}{\sqrt{\v}}\EE\left[Y_i^2\right]\EE\left[\frac{(Z-Z_i)^{M+1}}{(M+1)!}f_m^{(M+2)}(T_i)\right]\right\}\Bigg|\\&=	\frac{1}{\v^{\frac{M+3}{2}}}\Bigg|\sum\limits_{i=1}^n\left\{\EE\left[\frac{Y^{M+3}_i}{(M+2)!}f_m^{(M+2)}(T_i)\right]+ \EE\left[Y_i^2\right]\EE\left[\frac{Y_i^{M+1}}{(M+1)!}f_m^{(M+2)}(T_i)\right]\right\}\Bigg|\\
			&\leq \frac{2}{(M+2)!}\frac{1}{\v^{\frac{M+3}{2}}}\sum\limits_{i=1}^n\Bigg\{\EE\Big[Y^{2}_i\cdot \big|f_m^{(M+2)}(T_i)\big|\Big]+ \EE\left[Y_i^2\right]\EE\Big[\big|f_m^{(M+2)}(T_i)\big|\Big]\Bigg\}\,.
		\end{align*}
		Recalling $\sum_i \E[Y_i^2]=\v$, this is at most
		\begin{equation}\label{eq:I1bound}
        \begin{split}
		    \mathcal{I}_1
            &:= \frac{2}{(M+2)!}\frac{1}{\v^{\frac{M+1}{2}}}\Big\|f_m^{(M+2)}\Big\|_\infty\\
            &= \frac{2}{(M+2)!}\frac{1}{\v^{\frac{M+1}{2}}} \frac{1}{1+2\gamma}\|h_{m-1}^{(M+2)}\|_{\infty}\leq C_{m,M}\frac{1}{(1+2\gamma)\gamma^{\frac{m}{2}-1}}\left(\frac{\gamma}{\v}\right)^{\frac{M+1}{2}},
        \end{split}
		\end{equation}
        where we used $f_m=-\wt{h}/(1+2\gamma)$ (cf.~ Lemma~\ref{lem:poissol}) and $\|h_{m-1}^{(M+2)}\|_{\infty}\leq C_{m,M}\gamma^{\frac{M+3-m}{2}}$ from Lemma~\ref{lem:hdervbounds}.
        Similarly, the rest of the term in the right-hand side of \eqref{eq:taylor} is bounded by
		\begin{align*}
			&\Bigg|\frac{1}{\sqrt{\v}}\sum\limits_{i=1}^n\Bigg\{ \EE\bigg[Y_i\sum\limits_{\ell=2}^{M+1} \frac{(Z-Z_i)^\ell}{\ell!}f_m^{(\ell)}(Z_i)\bigg] + \frac{1}{\sqrt{\v}}\EE\left[Y_i^2\right]\EE\bigg[\sum\limits_{\ell=2}^{M+1} \frac{(Z-Z_i)^{\ell-1}}{(\ell-1)!}f_m^{(\ell)}(Z_i)\bigg]\Bigg\}\Bigg|\\
			&=	\Bigg|\frac{1}{\sqrt{\v}}\sum\limits_{i=1}^n\Bigg\{ \sum\limits_{\ell=2}^{M+1} \frac{\EE\big[Y_i^{\ell+1}\big]}{\v^{\frac{\ell}{2}}\ell!}\EE\left[f_m^{(\ell)}(Z_i)\right] + \EE\left[Y_i^2\right]\sum\limits_{\ell=2}^{M+1}\frac{\EE\big[Y_i^{\ell-1}\big]}{\v^{\frac{\ell}{2}}(\ell-1)!}\EE\left[ f_m^{(\ell)}(Z_i)\right]\Bigg\}\Bigg|\\
			&\leq \frac{2}{\sqrt{\v}}\sum\limits_{i=1}^n\EE\left[Y_i^2\right]\sum\limits_{\ell=2}^{M+1} \frac{1}{\v^{\frac{\ell}{2}}}\bigg|\EE\Big[f_m^{(\ell)}(Z_i)\Big]\bigg|,
		\end{align*}
        where the equality holds since $Z-Z_i$ is independent of $Y_i$, and the inequality holds by the triangle inequality and $|Y_i|\leq 1$.
		For $i \in [n]$ we let $G_i = \mathcal{N}(0,\mathrm{Var}(Z_i))$ and set
		\begin{equation}\label{eq:def:I2:I3}
        \begin{split}
			\mathcal{I}_2 &:=\frac{2 }{\sqrt{\v}}\sum\limits_{i=1}^n\EE\left[Y_i^2\right] \sum\limits_{\ell=2}^{M+1} \frac{1}{\v^{\frac{\ell}{2}}}\bigg|\EE\Big[f_m^{(\ell)}(Z_i)\Big]-\EE\Big[f_m^{(\ell)}(G_i)\Big]\bigg|\\
			\mathcal{I}_3 &:= \frac{2 }{\sqrt{\v}}\sum\limits_{i=1}^n\EE\left[Y_i^2\right] \sum\limits_{\ell=2}^{M+1} \frac{1}{\v^{\frac{\ell}{2}}}\bigg|\EE\Big[f_m^{(\ell)}(G_i)\Big]\bigg|
        \end{split}
		\end{equation}
		Combining the above bounds and collecting terms in \eqref{eq:taylor} gives the bound
		\begin{equation} \label{eq:zerotaylorbound}
			\bigg|\EE\Big[\wt{h}_m(Z)\Big]-\EE\Big[\wt{h}_m(G)\Big]\bigg| \leq \mathcal{I}_1+ \mathcal{I}_2+ \mathcal{I}_3.
		\end{equation}
		After establishing the base case, we shall bound each term separately.

        \medskip
	\noindent{\bf Step 1 - The base case $M=0$:} Fix $\v>\max( 1,\gamma)$. To prove the base case, we apply \eqref{eq:zerotaylorbound} with $M=0$. In this case, $\mathcal{I}_2,\mathcal{I}_3 = 0$ by definition. Thus, the bound \eqref{eq:I1bound} for $\cI_1$ yields that for every $m\geq 1$,
		$$\bigg|\EE\left[\wt{h}_m(Z)\right]-\left[\wt{h}_m(G)\right]\bigg|\leq \frac{C_{m}}{(1+2\gamma)\gamma^{\frac{m}{2}-1}} \sqrt{\frac{\gamma}{\v}}.$$
		Taking a supremum over all $Z \in \mathcal{S}^\v$ then shows
		\begin{equation*}\label{eq:tilde:base}
		    \wt{\Delta}^\v_m(\gamma) \leq \frac{C_{m}}{(1+2\gamma)\gamma^{\frac{m}{2}-1}} \sqrt{\frac{\gamma}{\v}}.
		\end{equation*}
        We then use Lemma~\ref{lem:zerototwotilde} to bound $\Delta_0^{\v}(\gamma)$. Since the previous estimate applies for all $\gamma\in (0,\v)$, 
		\begin{align*}
			\Delta^\v_0(\gamma) \leq\frac{1}{\sqrt{1+2\gamma}}\int_0^\gamma\sqrt{1+2t}\wt{\Delta}^\v_2(t)dt&\leq 	\frac{C}{\sqrt{1+2\gamma}}\frac{1}{\sqrt{\v}}\int_0^\gamma \sqrt{\frac{t}{1+2t}}dt\\
			&\leq  \frac{ C'\gamma}{\sqrt{1+2\gamma}} \frac{1}{\sqrt{\v}}.
		\end{align*}
      Therefore, applying these estimates to the bound on $\Delta_m^{\v}(\gamma)$ in Lemma~\ref{lem:wtDelta:to:Delta} yields
        \[
        \begin{split}
            \Delta_m^{\v}(\gamma)
            \leq \frac{C_{m,M}}{(1+2\gamma)^{\frac{m}{2}}}\Delta_0^{\v}(\gamma)+\sum_{j=1}^{m}\frac{C_{m,M}}{(1+2\gamma)^{\frac{m-j}{2}}}\wt{\Delta}_j^{\v}(\gamma)\leq \frac{C_{m,M}'}{\gamma^{\frac{m}{2}}}\sqrt{\frac{\gamma}{1+2\gamma}}\sqrt{\frac{\gamma}{\v}}.
        \end{split}
        \]
		Since $\v>\gamma$, this concludes the base case. 
		
        Next, we carry out the induction step. Assume that the desired claim holds at level $M$. For $\v>\max(M+2,\gamma)$, we prove the desired bound holds for $M+1$.
        \medskip
		
	\noindent {\bf Step 2 - Bounding the Gaussian approximation term $\mathcal{I}_2$ through induction:}  We bound $\mathcal{I}_2$ using the inductive hypothesis. Fix $i\in [n]$ and $2\leq \ell \leq M+1$. Recall that $G_i$ is a Gaussian random variable with the same variance as $Z_i$ and set $\v' = \v - \mathrm{Var}(Y_i)$. Since $f_m=-h_{m-1}/(1+2\gamma)$ (cf. \eqref{eq:Phtilde}), Lemma~\ref{lem:hdervs} gives the expansion of $f_{m}^{(\ell)}$ in terms of $(h_{m-1+i})_{-\ell\leq i\leq \ell}$, so by a triangle inequality
	\begin{equation}\label{eq:poisson:derivative}
        \begin{split}
			\left|\EE\left[f_m^{(\ell)}(Z_i)\right] - \EE\left[f_m^{(\ell)}(G_i)\right]\right|
			&\leq \frac{C_{m,M}}{1+2\gamma}\sum_{i=-\ell}^{\ell}\gamma^{\frac{\ell+i}{2}}\Big|\EE\left[h_{m-1+i}(Z_i)\right] - \EE\left[h_{m-1+i}(G_i)\right] \Big|\\
		 &\leq C_{m,M}\frac{\gamma^{\frac{\ell}{2}}}{1+2\gamma}\sum_{i= -\ell}^{\ell}\gamma^{\frac{i}{2}}\left(\frac{\v'}{\v}\right)^{\frac{m-1+i}{2}}\Delta^{\v'}_{m-1+i}\left(\frac{\v'\gamma}{\v}\right),
        \end{split}
		\end{equation}
        where we used Lemma~\ref{lem:variancedec} in the last step. Here, we adopted the convention from Lemma~\ref{lem:hdervs} that $h_k\equiv$ whenever $k<0$, and correspondingly set $\Delta_k\equiv 0$ in this range.
        
        To apply the inductive hypothesis, we need to guarantee that $\v'$ lies in the admissible range at level $M$. Since $\v>M+2$ and $|Y_i|\leq 1$ a.s.,
        \[
        \v'\equiv \v-\Var(Y_i)\geq \v-1>M+1,
        \]
       and $\frac{\v'}{\v}>\frac{M+1}{M+2}\geq \frac{1}{2}$.
        Letting $\gamma'=\frac{\v'}{\v}\gamma$, we also have $\frac{\gamma'}{\v'}=\frac{\gamma}{\v}<1$.
        Thus, the inductive hypothesis applied with $(\v',\gamma')$. Writing $j=m-1+i$, for any $-\ell\leq i\leq \ell$, we obtain from the inductive hypothesis 
		\begin{align*}
			\left(\frac{\v'}{\v}\right)^{\frac{j}{2}}\Delta^{\v'}_j\left(\frac{\v'\gamma}{\v}\right) &\leq C_{m,M}\left(\frac{\v'}{\v}\right)^{\frac{j}{2}}\left( \frac{\log(2+2\frac{\v'}{\v}\gamma)^{M}}{\sqrt{\v'}\left(1+2\frac{\v'}{\v}\gamma\right)^{\frac{j+2}{2}}} +  \frac{1}{\left(\frac{\v'}{\v}\gamma\right)^{\frac{j}{2}}}\sqrt{\frac{\gamma}{1+2\gamma}}\left(\frac{\gamma}{\v}\right)^{\frac{M}{2}}\right)\\
			&\leq C_{m,M}'\left(\frac{\log(2+2\gamma)^{M}}{\sqrt{\v}(1+2\gamma)^{\frac{j+2}{2}}} + \frac{1}{\gamma^{\frac{j}{2}}}\sqrt{\frac{\gamma}{1+2\gamma}}\left(\frac{\gamma}{\v}\right)^{\frac{M}{2}}\right),
		\end{align*}
         where we used $\frac{\v'}{\v}\in [\frac{1}{2},1]$ in the last step. Substituting this bound into \eqref{eq:poisson:derivative} yields
         \[
         \left|\EE\left[f_m^{(\ell)}(Z_i)\right] - \EE\left[f_m^{(\ell)}(G_i)\right]\right|\leq C_{m,M}
        \frac{\gamma^{\frac{\ell}{2}}}{1+2\gamma}\left(\frac{\log(2+2\gamma)^{M}}{\sqrt{\v}(1+2\gamma)^{\frac{m+1}{2}}}+\frac{1}{\gamma^{\frac{m-1}{2}}}\sqrt{\frac{\gamma}{1+2\gamma}}\left(\frac{\gamma}{\v}\right)^{\frac{M}{2}}\right).
         \]
         Multiplying $\v^{-\ell/2}$ and summing over $\ell=2,\ldots, M+1$ and using $\gamma<\v$, we obtain
		\begin{align*}
		\sum_{\ell=2}^{M+1}\frac{1}{\v^{\frac{\ell}{2}}}\left|\EE\left[f_m^{(\ell)}(Z_i)\right]-\EE\left[f_m^{(\ell)}(G_i)\right]\right|
        &\leq C_{m,M}\frac{\gamma}{\v}\left(\frac{\log(2+2\gamma)^{M}}{\sqrt{\v}(1+2\gamma)^{\frac{m+3}{2}}} + \frac{1}{\gamma^{\frac{m+1}{2}}}\sqrt{\frac{\gamma}{1+2\gamma}} \left(\frac{\gamma}{\v}\right)^{\frac{M}{2}}\right)\\
            &= C_{m,M}\frac{1}{\sqrt{\v}}\left(\frac{\gamma \log(2+2\gamma)^{M}}{\v(1+2\gamma)^{\frac{m+3}{2}}} +  \frac{1}{\gamma^{\frac{m}{2}}}\sqrt{\frac{\gamma}{1+2\gamma}}\left(\frac{\gamma}{\v}\right)^{\frac{M+1}{2}}\right).
		\end{align*}
	A key point here is that we gained an extra factor of $\sqrt{\frac{\gamma}{\v}}$. Recalling that $\v = \sum\limits_{i=1}^n \mathrm{Var}(Y_i)$, this yields
		\begin{align} \label{eq:I2bound}
			\mathcal{I}_2 &\equiv  \frac{2}{\sqrt{\v}}\sum\limits_{i=1}^n\EE\left[Y_i^2\right]\sum\limits_{\ell=2}^{M+1}\frac{1}{\v^\frac{\ell}{2}}\left|\EE\left[f_m^{(\ell)}(Z_i)\right]-\EE\left[f_m^{(\ell)}(G_i)\right]\right|\nonumber\\
			&\leq C_{m,M}\left(\frac{\log(2+2\gamma)^M}{\v(1+2\gamma)^{\frac{m+1}{2}}}+\frac{1}{\gamma^{\frac{m}{2}}}\sqrt{\frac{\gamma}{1+2\gamma}}\left(\frac{\gamma}{\v}\right)^{\frac{M+1}{2}}\right).
		\end{align}
        \medskip

		\noindent{\bf Step 3 - Bounding the Gaussian expectation term $\mathcal{I}_3$:}
		Recalling the definition of $\mathcal{I}_3$ in \eqref{eq:def:I2:I3},
        \[
     \cI_3\equiv \frac{2 }{\sqrt{\v}}\sum\limits_{i=1}^n\EE\left[Y_i^2\right] \sum_{\ell=2}^{M+1} \frac{1}{\v^{\frac{\ell}{2}}}\bigg|\EE\Big[f_m^{(\ell)}(G_i)\Big]\bigg|\leq 2M\max_{1\leq i\leq n}\Bigg\{\sum_{\ell=2}^{M+1} \frac{1}{\v^{\frac{\ell-1}{2}}}\bigg|\EE\Big[f_m^{(\ell)}(G_i)\Big]\bigg|\Bigg\},
        \]
        where the inequality follows from $\v=\sum_i \Var(Y_i)$. To this end, fix $i\in [n]$, and set $\v'\equiv \Var(Z_i)=\v-\Var(Y_i)$. Since $|Y_i|\leq 1$ a.s. and $\v >M+2$, we have $\v'/v\geq 1/2$. Thus, recalling $G_i\sim \cN(0,\v')$, 
        we can estimate the Gaussian expectations using Lemma \ref{lem:gausexpec}. In particular, we have $\E[f_m^{(\ell)}(G_i)]=0$ whenever $m+\ell$ is even.
        
        If $m$ is odd, the first non-vanishing term in the sum corresponds to $\ell = 2$. Thus, Lemma~\ref{lem:gausexpec} yields
		\begin{equation*} 
			\sum\limits_{\ell=2}^{M+1}\frac{1}{\v^{\frac{\ell-1}{2}}}\EE\left[f_m^{(\ell)}(G_i)\right] \leq \frac{C_{m,M}}{\sqrt{\v}(1+2\gamma)^{\frac{m+2}{2}}}.
		\end{equation*}
		If $m$ is even, the first non-vanishing term occurs at $\ell=3$ which is odd, and Lemma~\ref{lem:gausexpec} gives
		\begin{equation*} \label{eq:evenexpec}\sum\limits_{\ell=2}^{M+1}\frac{1}{\v^{\frac{\ell-1}{2}}}\EE\left[f_m^{(\ell)}(G_i)\right] \leq \frac{C_{m,M}}{\v(1+2\gamma)^{\frac{m+3}{2}}}.
		\end{equation*}
        Combining these estimates, we conclude that
        \begin{equation}\label{eq:I3bound}
        \cI_3\leq  \frac{C_{m,M}}{\sqrt{\v}(1+2\gamma)^{\frac{m+2}{2}}}\quad \textnormal{if $m$ is odd},\qquad 
             \cI_3\leq  \frac{C_{m,M}}{\v(1+2\gamma)^{\frac{m+3}{2}}}\quad \textnormal{if $m$ is even}
        \end{equation}
	\medskip
		\noindent{\bf Step 4 - Combining the bounds for $\wt{\Delta}^\v_m(\gamma)$:}
		We now combine the estimates obtained in the previous steps.
        
        When $m$ is odd, we substitute the bounds for $\cI_1$ from \eqref{eq:I1bound}, $\cI_2$ from \eqref{eq:I2bound}, and $\cI_3$ from \eqref{eq:I3bound} into  \eqref{eq:zerotaylorbound}. Taking the supremum over $Z\in \mathcal{S}^{\v}$ and using $\v>\gamma$, we obtain
	\begin{equation}\label{eq:tilde:Delta:odd}
	   \wt{\Delta}_m^{\v}(\gamma) \leq C_{m,M}\left(\frac{\log(2+2\gamma)^{M}}{\sqrt{\v}(1+2\gamma)^{\frac{m+2}{2}}} +  \frac{1}{\gamma^{\frac{m}{2}}}\sqrt{\frac{\gamma}{1+2\gamma}}\left(\frac{\gamma}{\v}\right)^{\frac{M+1}{2}}\right),
	\end{equation}
   Similarly, when $m$ is even the bounds \eqref{eq:I1bound}, \eqref{eq:zerotaylorbound}, \eqref{eq:I2bound}, and \eqref{eq:I3bound} yield
     \begin{equation}\label{eq:tilde:Delta:even}
           \wt{\Delta}_m^{\v}(\gamma) \leq C_{m,M}\left(\frac{\log(2+2\gamma)^{M}}{\v(1+2\gamma)^{\frac{m+1}{2}}} +  \frac{1}{\gamma^{\frac{m}{2}}}\sqrt{\frac{\gamma}{1+2\gamma}}\left(\frac{\gamma}{\v}\right)^{\frac{M+1}{2}}\right).
     \end{equation}
	\medskip
	\noindent {\bf Step 5 - Finishing the induction step:}
	We now transfer the bounds of $\wt{\Delta}_m^{\v}(\gamma)$ for $m\geq 1$ into $\Delta_m^{\v}(\gamma)$ for $m\geq 0$. First, combining Lemma~\ref{lem:zerototwotilde} with the estimate \eqref{eq:tilde:Delta:even} for $m=2$, we obtain
		\begin{equation*}\label{eq:zero:Delta}
        \begin{split}
			\Delta^\v_0(\gamma) &\leq\frac{1}{\sqrt{1+2\gamma}}\int_0^\gamma\sqrt{1+2t}\wt{\Delta}^\v_2(t)dt\\
			&\leq \frac{C_{M}}{\sqrt{1+2\gamma}}\int_0^\gamma\sqrt{1+2t}\left(\frac{\log(2+2\gamma)^{M}}{\v(1+2t)^{\frac{3}{2}}} +  \frac{\sqrt{t}}{t\sqrt{(1+2t)}}\left(\frac{t}{\v}\right)^{\frac{M+1}{2}}\right)dt\\
			&\leq C_M'\left(\frac{\log(2+2\gamma)^{M+1}}{\v\sqrt{1+2\gamma}}+\sqrt{\frac{\gamma}{1+2\gamma}}\left(\frac{\gamma}{\v}\right)^{\frac{M+1}{2}}\right).
        \end{split}
		\end{equation*}
        We now combine this estimate with Lemma~\ref{lem:wtDelta:to:Delta} and the bound \eqref{eq:tilde:Delta:even} for even $m$. This yields
    \[
    \begin{split}
    \Delta_{m}^{\v}(\gamma)
    &\leq  \frac{C_{m,M}}{(1+2\gamma)^{\frac{m}{2}}}\Delta_0^{\v}(\gamma)+\sum_{\substack{0\leq j\leq m\\ j:\textnormal{ even}}} \frac{C_{m,M}}{(1+2\gamma)^{\frac{m-j}{2}}} \wt{\Delta}_j^{\v}(\gamma)\\
    &\leq C'_{m,M}\left(\frac{\log(2+2\gamma)^{M}}{\v(1+2\gamma)^{\frac{m+1}{2}}} +  \frac{1}{\gamma^{\frac{m}{2}}}\sqrt{\frac{\gamma}{1+2\gamma}}\left(\frac{\gamma}{\v}\right)^{\frac{M+1}{2}}\right).
    \end{split}
    \]
    The case where $m$ is odd follows by the same argument, replacing \eqref{eq:tilde:Delta:odd} with \eqref{eq:tilde:Delta:even}. This completes the inductive step and the proof.
	\end{proof}\vspace{-\baselineskip}
    With Proposition \ref{prop:centeredapprox} we can finally prove Theorem \ref{thm:mainapprox}. 
    Having established the result for $\Delta_m^\v(\gamma)$, the final step is to reintroduce the mean $\mu$ and establish the same result $\Delta_{m,\mu}^\v(\gamma)$ provided that $\mu$ is small enough. 
    
\begin{proof}[Proof of Theorem \ref{thm:mainapprox}]
  The proof follows by using the bounds established in Proposition~\ref{prop:centeredapprox} for $\mu=0$, and then Taylor expanding $h_m(Z+\mu)$ with respect to $\mu$. Applying $M$'th order Taylor expansion gives
	\begin{equation} \label{eq:mean}
    \begin{split}
		\EE[h_m(Z+\mu)] - \EE[h_m(Z)]  = \sum\limits_{\ell=1}^{M-1} \frac{\mu^\ell}{\ell!}\EE\left[h^{(\ell)}_m(Z)\right] + \frac{\mu^{M}}{M!}\EE\left[h^{(\ell)}_m(T)\right],
        \end{split}
	\end{equation}
    for some $T\in [Z,Z+\mu]$. We use
	$|\EE[h^{(\ell)}_m(Z)]| \leq |\EE[h^{(\ell)}_m(Z)] - \EE[h^{(\ell)}_m(G)]|+|[h^{(\ell)}_m(G)]|$ to bound the first summmand.
	A calculation analogous to Lemma \ref{lem:gausexpec} shows that for any fixed $\ell\geq 0$,
	$$\bigg|\EE\Big[h^{(\ell)}_m(G)\Big]\bigg| =\bigg|\E\Big[H_{\ell}(G)h_m(G)\Big]\bigg|= O\bigg(\frac{1}{(1+2\gamma)^{\frac{m+1}{2}}}\bigg),$$
    where $H_{\ell}(x)=(-1)^{\ell}e^{\frac{x^2}{2}}\frac{d^{\ell}}{dx^{\ell}}e^{\frac{x^2}{2}}$ is the Hermite polynomial and $O(\cdot)$ hides the dependence on $\ell$. When $m$ is odd, we have the improvement
    \begin{equation} \label{eq:improvedexpec}
        \bigg|\EE\Big[h^{(\ell)}_m(G)\Big]\bigg| = O\bigg(\frac{1}{(1+2\gamma)^{\frac{m+2}{2}}}\bigg) 
    \end{equation}
    Moreover, invoking the established bounds for $\Delta_m^\v(\gamma)$ from Proposition~\ref{prop:centeredapprox}, along with Lemma \ref{lem:hdervs} gives that for any $0\leq \ell\leq M$,
	\begin{align*} 
		\bigg|\EE\Big[h^{(\ell)}_m(Z)\Big] - \EE\Big[h^{(\ell)}_m(G)\Big]\bigg| &\leq \sum\limits_{i=-\ell}^\ell c_{m,\ell,i}\gamma^{\frac{\ell+i}{2}}\Delta_{m+i}^\v(\gamma) \nonumber\\
        &\leq C_{m}\left(\frac{\gamma^{\frac{\ell}{2}}\log(2+2\gamma)^M}{\sqrt{\v}(1+2\gamma)^{\frac{m+2}{2}}} + \frac{\gamma^{\frac{\ell+1}{2}}}{\gamma^{\frac{m}{2}}\sqrt{1+2\gamma}}\left(\frac{\gamma}{\v}\right)^{\frac{M}{2}}\right).
	\end{align*} 
   Here and below, we denote $C_{m}$ as a constant that only depends on $m,M$, and $\bound$, which may change from line to line. Consequently, since $|\mu| \leq C/\sqrt{\v}$ and $\omega >\max(\gamma, (M+1)\bound)$, we have for even $m$ that
	\begin{align*}
		\sum\limits_{\ell=1}^{M-1} \frac{|\mu|^\ell}{\ell!}&\left(\left|\EE\left[h^{(\ell)}_m(Z)\right] - \EE\left[h^{(\ell)}_m(G)\right]\right| + \left|\EE\left[h^{(\ell)}_m(G)\right]\right|\right)\\
        &\leq C_{m} \left(\frac{|\mu|\sqrt{\gamma}\log(2+2\gamma)^M}{\sqrt{\v}(1+2\gamma)^\frac{m+2}{2}} + \frac{\mu^2}{(1+2\gamma)^{\frac{m+1}{2}}} + \frac{\sqrt{\gamma}}{\gamma^{\frac{m}{2}}\sqrt{1+2\gamma}}\left(\frac{\gamma}{\v}\right)^{\frac{M}{2}}\right)\\
        &\leq C_{m}\left(\frac{\log(2+2\gamma)^M}{\v(1+2\gamma)^\frac{m+1}{2}} + \frac{1}{\gamma^{\frac{m}{2}}}\left(\frac{\gamma}{\v}\right)^{\frac{M}{2}}\right),
	\end{align*}
	where we also used that $h_m'$ is an odd function and so $\EE\left[h_m'(G)\right] = 0$.
    When $m$ is odd, the first derivative does not vanish, and we instead use the improved bound in \eqref{eq:improvedexpec}.
    	Then for odd $m$ the bounds on $\mu$ imply,
	\begin{align*}
		\sum\limits_{\ell=1}^{M-1} \frac{|\mu|^\ell}{\ell!}&\left(\left|\EE\left[h^{(\ell)}_m(Z)\right] - \EE\left[h^{(\ell)}_m(G)\right]\right| + \left|\EE\left[h^{(\ell)}_m(G)\right]\right|\right)\\
        &\leq C_m \left(\frac{|\mu|\sqrt{\gamma}\log(2+2\gamma)^M}{\sqrt{\v}(1+2\gamma)^\frac{m+2}{2}} + \frac{|\mu|}{(1+2\gamma)^{\frac{m+2}{2}}} + \frac{1}{\gamma^{\frac{m}{2}}\sqrt{1+2\gamma}}\left(\frac{\gamma}{\v}\right)^{\frac{M}{2}}\right)\\
        &\leq C_m \left(\frac{\log(2+2\gamma)^M}{\sqrt{\v}(1+2\gamma)^{\frac{m+2}{2}}} + \frac{1}{\gamma^{\frac{m}{2}}}\left(\frac{\gamma}{\v}\right)^{\frac{M}{2}}\right).\\
	\end{align*}
	As for the remainder term, by Lemma \ref{lem:hdervbounds}, and since $|\mu| \leq C/\sqrt{\v}$,
	$$\frac{|\mu|^{M}}{M!}\bigg|\EE\left[h^{(M)}_m(T)\right] \bigg|\leq \frac{|\mu|^{M}}{ M!}\|h^{( M)}_m\|_{\infty} \leq C_m \frac{\left(\mu\sqrt{\gamma}\right)^{M}}{\gamma^\frac{m}{2}}\leq C_m'\frac{1}{\gamma^\frac{m}{2}}\left(\frac{\gamma}{\v}\right)^{\frac{M}{2}}.$$ Combining the three displays above with \eqref{eq:mean}, and using Proposition~\ref{prop:centeredapprox} to bound $|\E[h_m(Z)]-\E[h_m(G)]|$ concludes the proof.
\end{proof}\vspace{-\baselineskip}

    \section{Slow mixing in Erd\"os-R\'enyi and random $d$-regular graphs}
    In this section we will colloquially use $G$ to denote a random graph drawn from one of the following two distributions. The type of distribution will determine the possible degree behavior, which we now define.
    \begin{definition} \label{def:admis}
    We call the following two distributions \emph{admissible} for our lower bound:
    \begin{itemize}
    	\item Either $G\sim G(n,d/n)$ is an Erd\"os-R\'enyi random graph with average degree $d<\frac{n}{2}$. In particular \emph{$d$ can depend on $n$}.
    	\item Or $G$ is a random $d$-regular graphs, where \emph{$d$ is fixed} and potentially large.
    \end{itemize}
        \end{definition}
    Other than some specific estimates, which we will handle separately below, the arguments are virtually identical for these two random graph ensembles. 
For $G$ as above, denote the anti-ferromagnetic Ising model on $G$ with inverse temperature $\beta$ by $\nu_{\beta,d}$. That is, $\nu_{\beta,d}$ is a probability measure on $\{-1,+1\}^{n}$ such that
    \[
    \nu_{\beta, d}(x)\propto \exp\left(\beta\sqrt{d} H_{n,d}(x)\right)\,,\quad H_{n,d}(x):=-\frac{1}{\sqrt{d}}\langle x,A_Gx \rangle,
    \]
    where $A_G=(A_{ij})_{i,j\in [n]}$ denotes the adjacency matrix of $G$. 
    We prove slow mixing of $\nu_{\beta,d}$ for $\beta \gtrsim d^{-1/2}$ by following the proof of slow mixing of SK model at low temperature by \cite{Sellke25}, which builds upon the existence of \textit{gapped states} by~\cite{minzer2024perfect, dandi2025maximal}. To define this notion, consider the \textit{local field} of a configuration $x\in \{-1,+1\}^n$ at site $i\in [n]$
   \begin{equation}\label{eq:def:local:field}
  \partial_i H_{n,d}(x)=x_i\cdot (H_{n,d}(x)-H_{n,d}(x^{(i)}))/2=-\frac{2}{\sqrt{d}}(A_G x)_i,
   \end{equation}
   where $x^{(i)}$ denotes the configuration obtained by flipping $i$'th coordinate of $x$.

   For $\kappa,\delta>0$, we say $x=(x_i)_{i\in [n]}\in \{-1,+1\}^n$ is $(\kappa,\delta)$-gapped state for $H_{n,d}$ if
   \begin{equation}\label{eq:def:gapped}
   |I_x|\equiv |\left\{i\in [n]:x_i   \partial_i H_{n,d}(x)<\kappa \right\}|\leq \delta n.
   \end{equation}
   That is, flipping any $i\in [n]$ coordinate except $\delta n$ coordinates decreases energy by at least $\kappa$. Note that for $G \sim G(n,d/n)$,  it is essential to allow for $\delta > 0$. Indeed, when $d = \Theta(1)$, $G(n,d/n)$ has isolated vertices which cannot contribute to any gapped state. Below, we will increase $d$ as a function of $\delta$ to lessen this effect.

    Our results also extend immediately to the \emph{Kawasaki dynamics on the central slice}, which we now introduce briefly. These dynamics are associated to the Gibbs measure conditional on the slice $\{x:\langle \one, x\rangle/n = m\}$ for some $m\in [-1,1]$. The main difference in the Kawasaki dynamics is that in each update, two randomly chosen neighbors resample their spins according to the restricted Gibbs measure, conditional on the other vertices. This is in contrast to Glauber dynamics, where only a single vertex may change its spin, and the choice of the vertex does not depend on the structure of the graph. When $m = 0$ in the restriction, we shall refer to the process as Kawasaki dynamics on the central slice.
   
 \begin{proposition}\label{prop:slow:mixing:ER}
        Suppose that there exists an absolute constant $\kappa>0$ such that for any fixed $\delta>0$, if $d\geq d(\delta)$ then with probability $1-o(1)$, there exists a $(\kappa,\delta)$-gapped state for $H_{n,d}$ satisfying $\langle \one , x\rangle =0$.
     
        Then, there exist absolute constants $d_0, C,c>0$ such that if $d\geq d_0$ and $\beta> C/\sqrt{d}$, then with probability $1-o(1)$ over the randomness of the admissible distribution (as in Definition \ref{def:admis}) of $G$, we have that the mixing time of Glauber dynamics on $\nu_{\beta,d}$ is at least 
        $e^{cn\beta \sqrt{d}}$. The same result applies when the Glauber dynamics is replaced with the Kawasaki dynamics on the central slice.
    \end{proposition}
    For large enough, but fixed $d$, the assumption of Proposition~\ref{prop:slow:mixing:ER} is established in \cite[Theorem 2]{dandi2025maximal} under a numerical hypothesis that was verified in~\cite{minzer2024perfect} using interval arithmetics. The paper also shows the existence of a gapped state in for $G(n,1/2)$, and in Subsection \ref{sec:intermediate} we explain how their result extends to the entire intermediary regime $1\ll d\ll n$. The existence of gapped states in random $d$-regular is more delicate. Since the edges in this model are not independent, the required result does not follow from previous works. Instead, in Section \ref{sec:regulargapped} we built upon a method introduced in \cite{Dembo2017extremal} to couple between the so-called configuration model and an Erd\"os-R\'enyi graph. The coupling allows us to show the existence of gapped states in the configuration model, which then implies the same for random $d$-regular graphs, for any fixed $d$. 
     \begin{remark} \label{rmk:gapped}
    	It is plausible that a direct approach can show the existence of gapped states in random $d$-regular graphs when $d$ grows with $n$. In that case, as will become evident below, our proof can easily extend to more general regimes. 
    \end{remark}

   In light of the above existence results, Proposition~\ref{thm:explowerbound} follows immediately by combining Proposition \ref{prop:slow:mixing:ER}, as well as the analog result for Kawasaki dynamics follows immediately. We thus focus on Proposition~\ref{prop:slow:mixing:ER} for the rest of this section.
   
	As emphasized above, the proof strategy is similar to~\cite{Sellke25}, which considers the Wigner ensemble $W$ instead of  $A_G$. The only non-trivial change to make is to adapt the result \cite[Proposition 2.3]{Sellke25} proving that with high probability, any principal submatrix of $W$ of size $\rho n\times \rho n$, for some $\rho\in (0,1/2)$, has operator norm at most $O(\sqrt{\rho\log(1/\rho)})$ which is $o_{\rho}(1)$. The same is \textit{false} when $W$ is replaced by $A_G$ due to potentially high-degree nodes in the Erd\"os-R\'enyi case, or the existence of many disjoint stars in random $d$-regular graphs. Instead, we will use the fact that for \emph{all} $x,y\in \{-1,+1\}^n$, $x-y$ is delocalized when $\|x-y\|=\Omega(n)$ to replace the operator norm by another norm that reduces the effect of high-degree nodes. Define the symmetric matrix $W$:
    \[
    W:= -\frac{1}{\sqrt{d}}(A_G-\E[A_G]),
    \]
    as a normalization of $A_G$.
    For $p\in (0,1/2)$, define the set of configurations
    \begin{equation} \label{eq:sparsedef}
    	 \mathcal{X}_{p}:= \left\{x\in \{-1/\sqrt{pn}, 1/\sqrt{pn},0\}^n:\|x\|_0=pn\right\}.
    \end{equation}
    The point of $\mathcal{X}_{p}$ is that if $x,y\in \{-1,+1\}^n$ differ in exactly in $pn$ coordinates, then $x-y/\|x-y\|_2\in \mathcal{X}_{p}$. So, instead of bounding $(x-y)W(x-y)$ by $\|W\|_{\op}\|x-y\|^2$, which can incur a large error due to high-degree vertices, we use the bound below.
  	\begin{lemma}\label{lem:bernstein}
		There exist universal constants $c,C>0$ such that the following holds with probability at least $1-n^{-c}$. For all $p\in (0,1/2)\cap n^{-1}\mathbb{Z}$, we have
		\[
		\begin{split}
			\sup_{x,y \in \mathcal{X}_p}\big|\langle x,Wy \rangle\big|
			&\leq C\left(\sqrt{p\log (1/p)}+\frac{1}{\sqrt{d}}\log(1/p)\right)\,,\\
			\sup_{x \in \mathcal{X}_p, y\in \mathcal{X}_1}\big|\langle x,Wy \rangle\big|
			&\leq C\left(1+\frac{1}{\sqrt{dp}}\right)\,.
		\end{split}
		\]
	\end{lemma}
	\begin{proof}
		Fix $p\in (0,1/2)\cap n^{-1}\mathbb{Z}$ and assume first that $G \sim G(n,d/n)$ is an Erd\"os-R\'enyi graph. We'll take a union bound over such $p$ since there are at most $n/2$ many of them. Note that since $p<1/2$, we have by Stirling approximation, 
		\[
		|\cX_p|=\binom{n}{pn}2^{pn}\leq \exp\Big(n\big(\mathcal{H}(p)+p\log 2)\Big) \leq \exp(3np\log(1/p)).
		\]
		Meanwhile, for each fixed pair $x,y\in \mathcal{X}_p$, since $\|x\|_{\infty}\cdot \|y\|_{\infty}= 1/(pn)$ and $\|x\|_2=\|y\|_2=1$. Since we've assumed for now that the entries of $W$ are independent, we have by Bernstein's inequality
		\begin{equation}\label{eq:bernstein}
			\begin{split}
				\P\left(\big|\langle x,Wy\rangle \big|\geq t\right)=\P\left(\Big|\sum_{i<j} W_{ij}(x_iy_j+x_jy_i)\Big|\geq t\right)
				&\leq \exp\left(-\frac{\frac{1}{2}t^2}{\frac{2}{n}+\frac{2}{3\sqrt{d}pn}t}\right)\\
				&\leq \exp\left(-\frac{n}{8}\min\Big\{t^2,3\sqrt{d}p t\Big\}\right)
			\end{split}
		\end{equation}
		where we used $W_{ij}$ are centered, independent with variance at $\frac{1}{n}(1-\frac{d}{n})$ and $|W_{ij}|\leq \frac{1}{\sqrt{d}}$. Thus, by a union bound over $x,y\in \mathcal{X}_p$ and $p\in (0,1/2)\cap n^{-1}\mathbb{Z}$, the first estimate holds with probability at least $1-\exp(-2np\log(1/p))\geq 1-n^{-3}$ by taking $C=50$. The second estimate follows by the same argument where we use $|\mathcal{X}_1|=2^n, |\mathcal{X}_p|\leq 2^{3n/2}$ and the analog of \eqref{eq:bernstein} holds where $3\sqrt{d}pt$ term is replaced with $3\sqrt{dp}t$.
		
		If $G$ is instead a random $d$-regular graph, then $W$ has dependencies between its entries and \eqref{eq:bernstein} does not follow from Bernstein's inequality. To handle this issue we replace the Bernstein inequality with a concentration bound from \cite{cook2018size} to arrive at the same result, see Lemma \ref{lem:edgeregularbernstein} immediately below. In particular, up to constants, Lemma \ref{lem:edgeregularbernstein} is identical to the estimate in \eqref{eq:bernstein} and so the proof proceeds in the same fashion.
	\end{proof}\vspace{-\baselineskip}
		\begin{lemma}\label{lem:edgeregularbernstein}
		Let $d$ be fixed, and suppose $G \sim G_{\reg}(n,d)$ is a random $d$-regular graph. Let $A_G$ be the adjacency matrix of $G$ and set $$W = -\frac{1}{\sqrt{d}}\left(A_G - \E[A_G]\right).$$
		Suppose that $n$ is large enough. Then, there exists a universal constant $C>0$, such that for any $x,y \in \RR^n$ with $\|x\|_2,\|y\|_2 \leq 1$,
		$$\PP\left(|\langle x, Wy\rangle| \geq t\right) \leq C\exp\left( \frac{1}{C}\frac{t^2}{\frac{1}{n} + \frac{a}{\sqrt{d}}t}\right),$$
		where $a = \|x\|_{\infty}\cdot \|y\|_{\infty}$.
	\end{lemma}
	\begin{proof}
		The proof is based on \cite[Theorem 5.1]{cook2018size}. We first assume that $x,y \in \RR_+^n$, have positive coordinates. For such vectors set  $\tilde{\sigma}^2 = \frac{d}{n-1}\sum\limits_{i\neq j}x^2_iy^2_j$ as well as $\mu = \E\left[\langle x,A_Gy\rangle\right]$. With these notation,  \cite[Theorem 5.1]{cook2018size} states, that there exists a constant $c_d > 0$, depending only on $d$, such that for any $t \geq \frac{c_d}{n}$,
		$$\P\left(\left|\langle x,A_Gy\rangle - \mu\right| \geq t\right) \leq \exp\left(-\frac{1}{C'}\frac{t^2}{\tilde{\sigma}^2+at}\right),$$
		where $C' > 1$ is some absolute constant, and where $a = \|x\|_{\infty}\|y\|_{\infty}$ is as in the statement of the result. 
		
		We now note that  $\left|\langle x, Wy\rangle\right| = \frac{1}{\sqrt{d}}\left|\langle x,A_Gy\rangle-\mu\right|$ and that $\tilde{\sigma}^2 \leq \frac{2d}{n}$, where we used $\|x\|_2,\|y\|_2 \leq 1$.
		Combining the above, and replacing $t$ by $\sqrt{d}t$, we get for $t \geq \frac{c_d}{n}$,
		\begin{equation} \label{eq:positivebernstein}
			\P\left(\left|\langle x,Wy\rangle\right| \geq t\right) \leq C'\exp\left(-\frac{1}{C'}\frac{t^2}{ \frac{1}{n}+\frac{a}{\sqrt{d}}t}\right).
		\end{equation}
		On the other hand if $t  < \frac{c_d}{n}$ and $n$ is large enough, then we can obtain the desired result by making $C'$ larger if necessary. This concludes the case when $x$ and $y$ have positive coordinates.
		
		The general case follows from a simple decomposition. For $x \in \RR^n$, write $x^+ = \max(0,x)$ and $x^- = -\min(0,x)$, i.e. separate the positive and negative entries of $x$, so that $x = x^+ - x^-$. Define $y^+$ and $y^-$ similarly. We have
		$$\big|\langle x,Wy \rangle\big| \leq \big|\langle x^+,Wy^+ \rangle\big| + \big|\langle x^-,Wy^- \rangle\big| + \big|\langle x^+,Wy^- \rangle\big| +\big|\langle x^-,Wy^+ \rangle\big|.$$
		Further, for any $\xi_1,\xi_2 \in \{\pm\}$, $\|x^{\xi_1}\|_{\infty}\cdot\|y^{\xi_2}\|_{\infty} \leq \|x\|_\infty\cdot\|y\|_{\infty}  = a$. From \eqref{eq:positivebernstein} it now follows that
		$$\P\left(\big|\langle x,Wy \rangle\big|
		\geq t\right) \leq 4C'\exp\left(-\frac{1}{16C'}\frac{t^2}{\frac{1}{n}+ \frac{a}{4\sqrt{d}}t}\right),$$
		concluding the proof.
	\end{proof}\vspace{-\baselineskip}

	From now on, we will fix a small $\kappa>0$ and choose $\rho>0$ small enough compared to $\kappa$, and $\delta>0$ small enough (resp. $d$ large enough) depending on both $\kappa$ and $\rho$. 
	
	\begin{lemma}\label{lem:energy:estimate}
		Fix $\kappa > 0$. There exists a constant $C_\kappa>0$ such that for any $n\geq C_\kappa$ and $\rho \leq 1/C_{\kappa}$, there exist $\delta>0$ and $d_0>0$, depending only on $\kappa$ and $\rho$ 
		such that on the event in Lemma~\ref{lem:bernstein}, the following holds: if $d>d_0$ and $x$ is a $(\kappa,\delta)$-gapped state with $\langle \one ,x \rangle =0$, then for any $y\in \{-1,+1\}^n$ with $\|x-y\|_1=\rho n$, we have
		\[
		H_{n,d}(y)\leq H_{n,d}(x)-\rho\kappa n/4.
		\]
	\end{lemma}
	\begin{proof}
		Recalling $W=-(A_G-\E[A_G])/\sqrt{d}$ and $H_{n,d}(x)=-\langle x, A_G x\rangle/\sqrt{d}$, define $\wt{H}_{n,d}(x)=\langle x, Wx \rangle$. Since $\E[A_G]=d/n(\one \one^{\sT}-I_n)$\footnote{When $G$ is a random $d$-regular graph we have $\frac{d}{n-1}$ instead of $\frac{d}{n}$ in the expectation, but the proof is the same otherwise.}, we have that
		\begin{equation}\label{eq:H:to:H:tilde}
			H_{n,d}(z)=\wt{H}_{n,d}(z)-\frac{\sqrt{d}}{n}\langle \one, z\rangle^2+\frac{\sqrt{d}}{n}\|z\|_2^2.
		\end{equation}
		Using $\|z\|_2^2=n$, and since we assumed $\langle \one, x\rangle =0$, it follows from $\langle \one, y\rangle^2\geq 0$ that
		\[
		\begin{split}
			H_{n,d}(x)-H_{n,d}(y)
			&\geq \wt{H}_{n,d}(x)-\wt{H}_{n,d}(y)\\
			&=\langle \nabla \wt{H}_{n,d}(x),x-y\rangle-\langle x-y, W(x-y)\rangle,
		\end{split}
		\]
		where the last equality holds since $\wt{H}_{n,d}(x)$ is a quadratic form so $2$nd Taylor order expansion around $x$ is exact. By differentiating \eqref{eq:H:to:H:tilde} and using $\langle \one, x\rangle =0$ and $x_i^2=1$, $ x_i \partial_i H_{n,d}(x)=x_i \partial_i \wt{H}_{n,d}(x)+\frac{2\sqrt{d}}{n}$. Since $d\leq n$, the last term is negligible for large enough $n\geq C_{\kappa}$. In particular, by definition of gapped state in \eqref{eq:def:gapped},
		\[
		x_i \partial_i \wt{H}_{n,d}(x)\geq \frac{\kappa}{2}\,,\qquad\textnormal{for}\quad i\notin I_x,
		\]
		and $|I_x|\leq \delta n$. Define $J\equiv J_{x,y}:=\{j\in [n]: x_j\neq y_j\}$. Since $\|x-y\|_1=\rho n$, $|J|=\rho n/2$, we have
		\[
		\langle \nabla \wt{H}_{n,d}(x),x-y\rangle = 2\sum_{i\in J}x_i\partial_i \wt{H}_{n,d}(x)\geq \kappa(\rho/2 - \delta)n - 2\bigg|\sum_{i\in J\cap I_x}x_i \partial_i \wt{H}_{n,d}(x)\bigg|, 
		\]
		where the $\kappa(\rho/2 - \delta)n$ term comes from the contribution of $|J\setminus I_x| \geq (\rho/2-\delta)n$ by definition of a gapped state. Using Lemma~\ref{lem:bernstein}, the final sum is bounded by
		\[
		4\bigg|\sum_{i\in J\cap I_x}\sum_{j=1}^{n}W_{ij}x_i y_j\bigg|\leq C\left(1+\frac{1}{\sqrt{d\frac{|J\cap I_x|}{n}}}\right)\sqrt{|J\cap I_x|}\sqrt{n}\leq C\left(\sqrt{\delta}+\frac{1}{\sqrt{d}}\right)n,
		\]
		where we used $|J\cap I_x|\leq |I_x|\leq \delta n$ in the final step. Combining with the previous display, setting the constant $\delta$ small enough compared to $\rho, \kappa$ and $d$ to be large enough so that $1/\sqrt{d}\ll \kappa\rho$, we have $\langle \nabla \wt{H}_{n,d}(x),x-y\rangle\geq \kappa\rho n/3$. For the quadratic term, using the first estimate in Lemma~\ref{lem:bernstein}, 
		\[
		\begin{split}
			\langle x-y, W(x-y)\rangle =\sum_{i,j\in J} (x_i-y_i)W_{ij}(x_j-y_j)
			&\leq C|J|\left(\sqrt{\frac{|J|}{n}\log (n/|J|)}+\frac{1}{\sqrt{d}}\log(n/|J|)\right)\\
			&\leq C'n\left(\sqrt{\rho^3 \log (1/\rho)}+\frac{\rho\log(1/\rho)}{\sqrt{d}}\right),
		\end{split}
		\]
		so setting $\rho$ small enough and $1/\sqrt{d}\ll \kappa/\log(1/\rho)$, we have $H_{n,d}(x)-H_{n,d}(y)\geq \rho\kappa n/4$.
	\end{proof}\vspace{-\baselineskip}
	\begin{proof}[Proof of Proposition~\ref{prop:slow:mixing:ER}]    
		Fix $\kappa > 0$ that satisfies the assumption of Proposition~\ref{prop:slow:mixing:ER}, and set $\rho>0$ which is admissible by Lemma 
		\ref{lem:energy:estimate}. Further let $\delta>0$ be small enough and $d > d_0$ be large enough, as determined by Lemma \ref{lem:energy:estimate}. Set $x$ to be a $(\kappa,\delta)$-gapped state of $H_{n,d}$ satisfying $\langle \one,x \rangle =0$, which exists with high probability by the premise of the proposition. 
		
		From here, the proof is identical to \cite[Theorem 1]{Sellke25}. Consider
		\[
		B(\rho) := \{y\in \{-1,+1\}^n:\|x-y\|_1\leq\rho n\},\qquad S(\rho)=\{y\in \{-1,+1\}^n:\|x-y\|_1=\rho n\}.
		\]
		Then, as long as $\beta \geq C/\sqrt{d}$ for large enough $C>0$, by Lemma \ref{lem:energy:estimate}, $\frac{\nu_{\beta,d}(S(\rho))}{\nu_{\beta,d}(x)}\leq 2^n e^{-\rho\kappa n/4}\leq e^{-c\beta\sqrt{d} n}$, and for $\rho < \frac{1}{2}$, $\nu_{\beta,d}(B(\rho)) \leq \frac{1}{2}$. The claim now follows from Cheeger's inequality. The proof for Kawasaki dynamics on the central slice follows exactly in the same manner. The main point is that the only states through which the Kawasaki chain can leave $B(\rho)$ lie in $\cup_{\rho'>\rho/2}S(\rho')$. Thus, by taking $\delta$ small enough and $d$ large enough as determined by Lemma~\ref{lem:energy:estimate} for $\rho/2$ instead of $\rho$, the same argument applies. We note in passing that the same argument applies to any other reversible Markov chains with respect to either $\nu_{\beta,d}$ or $\nu_{\beta,d}$ conditioned on the centered slice, which flips at most $cn$ coordinates at a time, where $c>0$ is a small enough constant.
	\end{proof}\vspace{-\baselineskip}
	
	\subsection{Gapped states in Erd\"os-R\'enyi graphs in the intermediate regime} \label{sec:intermediate}
	Here we consider the Erd\"os-R\'enyi graph $G \sim G(n,d/n)$ in the intermediate regime where $d:= d(n)$ satisfies $1\ll d\ll n$. We establish that with high probability $G$ possess gapped, which follows from the proof of \cite{dandi2025maximal} for the dense case $d = \frac{n}{2}.$ The main idea is to replace $A_G$, the adjacency matrix of $G$, by a Gaussian matrix, and invoke universality results. For simplicity we will work with the centered version satisfying $\tilde{A}_G = A_G - \E[A_G]$. This is with no loss of generality, since the existence of gapped space is invariant to shifting by a constant.
	
	Following \cite{dandi2025maximal}, for an $n\times n$ matrix $W$, we define the bisection deficit function
	$$D_d(W,\kappa,x) = \sum\limits_{i=1}^n\max\left(\kappa-\frac{1}{\sqrt{d}}\sum\limits_{j} W_{i,j}x_ix_j,0\right).$$
	Above $x \in \{\pm 1\}^n$ is a signing of the vertices which should be thought as a bisection, $\kappa$ is the potential gap of the state and $d$ is parameter acting as the density. Taking the minimum over of all bisections, we arrive at
	$$D^*_d(W,\kappa) := \min\limits_{x: \langle x , {\bf1}\rangle = 0} D_d(W,\kappa,x).$$
	The key point, used below, is that having $D^*_d(W,\kappa) = o(n)$ implies the existence of a gapped state.
	
	Now, let $J$ be random $n\times n$ matrix with independent standard Gaussian entries. For this matrix \cite[Lemma 16 and Lemma 17]{dandi2025maximal} imply 
	$$D_n^*(J,\kappa) = o(n),$$
	with high probability, provided $\kappa$ is small enough. 
	Our next step is to consider the re-normalized matrix $\tilde J = \sqrt{\frac{d}{n}}J$, for which it is clear that for every $x \in \{-1,1\}^n$, $D_d(\tilde{J},\kappa,x) = D_n(J,\kappa,x)$ holds, so
	\begin{equation} \label{eq:gappedgaus}
		D^*_d(\tilde{J},\kappa) = o(n).
	\end{equation}
	We will now show that $\left| D^*_d(\tilde{J},\kappa) - D^*_d(\tilde A_G,\kappa)\right| = o(n)$ as well.
	
	Since $\EE[\tilde A_G] = 0$, we note that, for every $i,j \in [n]$, 
	$$\EE[(\tilde A_G)_{i,j}^2] =\frac{d}{n}\left(1-\frac{d}{n}\right), \quad  \EE[\tilde{J}^2_{i,j}] =\frac{d}{n}, \quad \text{ and } \quad \EE[|(\tilde A_G)_{i,j}|^3], \EE[|\tilde{J}_{i,j}|^3] \leq C\frac{d}{n},$$
	for some universal constant $C>0$.
	Under these conditions, according to \cite[Equation (6.8)]{dandi2025maximal} as long as $1\ll d$, 
	\begin{equation} \label{eq:gappedapprox}
		\left| D^*_d(\tilde{J},\kappa) - D^*_d(\tilde A_G,\kappa)\right| = o(n).
	\end{equation}
	We make the important remark that all computations leading to \cite[Equation (6.8)]{dandi2025maximal}, and especially \cite[Equations (5.15),(5.16), (5.20) (6.2), (6.4), and (6.7)]{dandi2025maximal} are not asymptotic and hold for arbitrary values of $d$. In particular, $d$ could depend on $n$ for these estimates. Combining \eqref{eq:gappedgaus} and \eqref{eq:gappedapprox} we can thus see $D^*_d(\tilde A_G,\kappa) = o(n).$ Finally, for the existence of a gapped state, define 
	$$N_d(\tilde A_G,\kappa,x) = \sum\limits_{i}{\bf 1}\left(\frac{1}{\sqrt{d}}\sum\limits_j (\tilde A_G)_{i,j}x_ix_j - \kappa\geq 0\right), \text{ and } \quad N^*_d(\tilde A_G,\kappa) := \max\limits_{x: \langle x , {\bf1}\rangle = 0} N_d(\tilde A_G,\kappa,x).$$
	In words, $n-N^*_d(\tilde A_G,\kappa)$ is the number of vertices violating the gapped condition, as in \eqref{eq:def:gapped}.
	Then, according to \cite[Lemma 11]{dandi2025maximal} (see also \cite[Lemma 6]{dandi2025maximal}), the condition $D^*_d(\tilde A_G,\kappa) = o(n)$, implies
	$$N^*_d(\tilde A_G,\kappa') = n(1-o(n)),$$
	for any $\kappa' < \kappa$. This is then equivalent to the existence of a $(\kappa',\delta)$ state for any $\delta > 0$, provided $n$ is large enough.
	
	\subsection{Gapped states for random regular graphs} \label{sec:regulargapped}
	In this section, we show  that the existence of a balanced gapped state in the sparse Erd\"{o}s-R\'{e}nyi graph $G(n,d/n)$ for large enough $d$, implies the same property for random $d$-regular configuration model $G^{\textnormal{c}}_{\textnormal{reg}}(n,d)$. In the configuration model each vertex is assigned $d$ incident half-edges and a multigraph is formed by a uniform random matching among all $nd$ half-edges (assuming $nd$ is even). The probability of obtaining a simple graph is roughly $e^{-d^2}$. So, when $d$ is fixed, high probability events in the configuration model $G^{\textnormal{c}}_{\textnormal{reg}}(n,d)$, immediately transfer to random $d$-regular graphs $G_{\textnormal{reg}}(n,d)$.
	
	Throughout this section we adopt the following conventions. For a multigraph $G$ with adjacency matrix $A$, we set $A_{ij}=\ell$ if there are $\ell$ edges between vertices $i\neq j\in [n]$ and $A_{ii}=2\ell$ if there are $\ell$ loops at vertex $i$. We write $O(\cdot), o(\cdot)$ and $\Theta(\cdot)$ for the usual $n\to\infty$ asymptotics, while $O_d(\cdot), o_d(\cdot)$, and $\Theta_d(\cdot)$ denote the $d\to\infty $ asymptotics. A sequence of events $A_n$ holds with high probability (w.h.p.) if $\P(A_n)\to 1$ as $n\to\infty$. For random $\{X_n\}$ and non-random $f:\R^{+}\to\R^{+}$, we write $X_n=o_d(f(d))$ w.h.p. as $n\to\infty$ if there exists non-random $g(d)=o_d(f(d))$ such that the sequence $A_n =\{|X_n|\leq g(d)\}$ holds w.h.p. as $n\to\infty$.

	\begin{proposition}\label{prop:gapped:random:regular}
		Suppose there exists a constant $\kappa>0$ such that for any $\delta>0$, if $d\geq d(\delta)$ then w.h.p. as $n\to\infty$ there exists $(\kappa,\delta)$-gapped state for $G(n,d/n)$ satisfying $\langle \one, x\rangle =0$. Then for any $\eta>0$ and $\delta>0$, if $d\geq d'(\eta,\delta)$ then w.h.p. as $n\to\infty$, there exists $(\kappa-\eta, \delta)$-gapped state for $G^{\textnormal{c}}_{\reg}(n,d)$. The same holds with the roles of $G(n,d/n)$ and $G^{\textnormal{c}}_{\reg}(n,d)$ interchanged. 
		
		In particular, when $d$ is fixed and sufficiently large, since by \cite{dandi2025maximal} $G(n,d/n)$ has gapped states, so does $G_{\reg}(n,d)$.
	\end{proposition}
	\begin{remark}
		Although we only need the implication from $G(n,d/n)$ to $G_{\reg}(n,d)$ in order to establish slow mixing of anti-ferromagnetic Ising model on $G_{\reg}(n,d)$, together with the reverse implication this proposition shows that the so-called ``maximal stability threshold'' (see \cite[Definition 1]{dandi2025maximal}) is the same for both ensembles. In particular, since this threshold for $G(n,d/n)$ was identified in \cite{dandi2025maximal} as the root $h_\star$ of a certain nonlinear equation, Proposition~\ref{prop:gapped:random:regular} shows that the same value $h_\star$ also characterizes the threshold for $G_{\reg}(n,d)$.
	\end{remark}
	The proof relies on the coupling of $G(n,d/n)$ and $G^{\textnormal{c}}_{\reg}(n,d)$ introduced in \cite[Section 3]{Dembo2017extremal} (see also~\cite{Frieze1992independence} for a similar construction). Note that moving from $G(n,d/n)$ to $G^{\textnormal{c}}_{\reg}(n,d)$ is non-trivial since the degree of the typical vertices differ by $\Theta(\sqrt{d})$, thus any coupling of these two graphs leads to $\Theta(1)$ difference in their local field distributions (recall the $1/\sqrt{d}$ factor in $\partial_i H_{n,d}$~\eqref{eq:def:local:field}). Instead, we use the coupling from \cite{Dembo2017extremal} to embed $G(n,d_{-}), d_{-}=d-\sqrt{d}\log d$ into $G^{\textnormal{c}}_{\reg}(n,d)$. The additional randomness in this coupling allows us to show that the difference in their local fields is negligible for all but an $o_d(1)$ fraction of vertices.

	We first recall the coupling by \cite{Dembo2017extremal}. Let $d_{-}=d-\sqrt{d}\log d$. Draw i.i.d. $X_i\sim \textnormal{Poisson}(d_{-})$ for $i\in[n]$, and let $B_i =(d-X_i)_{+}$. Color $B_i$ of the half-edges incident to vertex $i$ {\small \sf{BLUE}}, and the remaining $d-B_i$ half-edges {\small \sf{RED}}. Matching the half-edges uniformly to obtain $G_1\sim G^{\textnormal{c}}_{\reg}(n,d)$, it is decomposed into $G_{\rr}$ consisting of all $\rr$ edges, $G_{\rb}$ consisting of all multi-color edges (i.e. $\rb$ and $\br$), and $G_{\bb}$ consisting of all $\bb$ edges. To construct the second graph $G_2$, remove all $\bb$ edges, disconnect all multi-colored $\rb$ edges, and delete all ${\small \sf{BLUE}}$ half-edges that as a result became unmatched. Finally, form a new subgraph $\widetilde{G}_{\rr}$ by a uniform re-matching between all unmatched ${\small \sf{RED}}$ half-edges (if odd number of them, leave one of them as self loop). The graph $G_2$ has vertex set $[n]$ and the edge set $E(G_2)=E(G_{\rr})\sqcup E(\widetilde{G}_{\rr})$.

	A key observation is that $G_2$ is approximately distributed as $G(n,d_{-}/n)$. More precisely, \cite{Dembo2017extremal} (see their proof of Lemma 3.1) showed that there exists a random graph $\widetilde{G}_2$ and a coupling between $G_2$ and $\widetilde{G}_2$ such that
	\begin{equation}\label{eq:G:2:contiguous}
		\textnormal{$\widetilde{G}_2$ is mutually contiguous with $G(n,(d-\sqrt{d}\log d)/n)$, and }|E(\widetilde{G}_2)\triangle E(G_2)|/n=o_d(\sqrt{d}).
	\end{equation}
	Let $A_1$ and $A_2$ denote the adjacency matrices of $G_1$ and $G_2$, respectively. Define
	\[
	\Delta:=A_1-A_2\,,\qquad \Omega_n:=\{x\in \{-1,+1\}^n:\langle \one, x\rangle =0\}.
	\]
	The following is our main approximation result from which we deduce Proposition \ref{prop:gapped:random:regular}.
	\begin{proposition}\label{prop:gapped:random:regular:2}
		For any $\eps, \delta>0$, there exists $d(\eps,\delta)$ such that if $d\geq d(\eps,\delta)$, then w.h.p. as $n\to\infty$
		\[
		\sup_{x\in \Omega_n}F_{\eps}^{+}(x)\leq \delta\quad\textnormal{where}\quad F_{\eps}^{+}(x):= \frac{1}{n}\sum_{i=1}^{n}\one\left\{\frac{x_i(\Delta x)_i}{\sqrt{d}}\geq \eps\right\}.
		\]
		The same conclusion holds with $F_{\eps}^{+}(x)$ replaced by $ F_{\eps}^{-}(x):=\frac{1}{n}\sum_{i=1}^{n}\one\left\{\frac{x_i(\Delta x)_i}{\sqrt{d}}\leq -\eps\right\}$.
	\end{proposition}
	\begin{proof}[Proof of Proposition~\ref{prop:gapped:random:regular}]
		Fix $\kappa>0$ as in the assumption. For a matrix $A\in \R^{n\times n}$ and $q>0$, define
		\[
		T_{A,q}(x):=\frac{1}{n}\sum_{i=1}^{n}\one\left\{x_i(Ax)_i \geq -q/2\right\}\,,\quad x\in \Omega_n.
		\]
		Recalling the definition of gapped states~\eqref{eq:def:gapped}, by assumption on $G(n,d/n)$, if $A\sim G(n,d/n)$ and $d\geq d(\delta)$, then w.h.p. $\inf_{x\in \Omega_n}T_{A,\kappa\sqrt{d}}(x)\leq \delta$. We first show that an analogous statement holds for $A_2$. Let $\tilde{A}_2$ be the adjacency matrix of $\tilde{G}_2$ and fix $\xi>0$, to be chosen below. For any $x\in \Omega_n$, a union bound yields
		\[
		T_{A_2,\kappa\sqrt{d_{-}}-\xi}(x)
		\leq T_{\tilde{A}_2,\kappa\sqrt{d_{-}}}(x)+\frac{1}{n}\sum_{i=1}^{n}\one \left\{\big|x_i(\tilde{A}_2 x)_i-x_i(A_2 x)_i\big|\geq \xi/2\right\}.
		\]
		Using Markov's inequality and the fact $\sum_{i}|(\tilde{A}_2-A_2)x_i|\leq 2|E(\tilde{G}_2)\triangle E(G_2)|$, this gives
		\[
		T_{A_2,\kappa\sqrt{d_{-}}-\xi}(x)\leq  T_{\tilde{A}_2,\kappa\sqrt{d_{-}}}(x)+\frac{4}{n\xi}\left|E(\tilde{G}_2)\triangle E(G_2)\right|.
		\]
		By the assumption on $G(n,d/n)$ together with the mutual contiguity of $\tilde{G}_2$, if $d$ is large enough so that $d-\sqrt{d}\log d \geq d(\delta/4)$, we have $\inf_{x\in \Omega_n}T_{\tilde{A}_2,\kappa\sqrt{d_{-}}}(x)\leq \delta/4$ with high probability. Moreover, $|E(\tilde{G}_2)\triangle E(G_2)|/n=o_d(\sqrt{d})$ by \eqref{eq:G:2:contiguous}, thus
		\[
		\inf_{x\in \Omega_n}T_{A_2,\kappa\sqrt{d_{-}}-\xi}(x)\leq \frac{\delta}{4}+\frac{o_d(\sqrt{d})}{\xi}.
		\]
		Given any fixed $\eta, \delta>0$, choosing $\xi\equiv \kappa\sqrt{d_{-}}-(\kappa-\eta)\sqrt{d}=\Omega_d(\sqrt{d})$ shows that if $d\geq d(\eta, \delta)$, w.h.p. as $n\to\infty$
		\begin{equation}\label{eq:intermediate:claim}
			\inf_{x\in \Omega_n}T_{A_2,(\kappa-\eta)\sqrt{d}}(x)\leq \delta/2.
		\end{equation}
		Finally, observe that for any $x\in \Omega_n$,
		\begin{equation*}\label{eq:gapped:triangle}
			T_{A_1,(\kappa-\eta)\sqrt{d}}(x)\leq T_{A_2,(\kappa-\eta/2)\sqrt{d}}(x)+F^{+}_{\eta/2}(x).
		\end{equation*}
		Note that by Proposition~\ref{prop:gapped:random:regular:2}, if $d$ is large enough, w.h.p., $\sup_{x\in \Omega_n} F_{\eps}^{+}(x)\leq \delta/2$. Therefore, combining with \eqref{eq:intermediate:claim} shows that $\inf_{x\in \Omega_n}  T_{A_1,(\kappa-\eta)\sqrt{d}}(x)\leq \delta$ holds w.h.p., concluding the proof of the existence of gapped states for $A_1\sim G^{\textnormal{c}}_{\reg}(n,d)$.
		
		For the reverse direction, observe that the map $d\mapsto d-\sqrt{d}\log d$ is a one-to-one correspondence on $[1,\infty)$. Using again the contiguity relation for $\tilde{G}_2$ and repeating the preceding argument with the roles of $A_1$ and $\tilde{A}_2$ interchanged yields the existence of $(\kappa-\eta,\delta)$-gapped states for $G(n,d_{-})$. 
	\end{proof}\vspace{-\baselineskip}
	We prove Proposition~\ref{prop:gapped:random:regular:2} via three lemmas below. Let $B=(B_1,...,B_n)$ denote the counts of ${\small \sf{BLUE}}$ edges of all vertices. 
	\begin{lemma}\label{lem:F:concentration}
		For any fixed $\eps>0$, with probability $1-o(1)$ as $n\to\infty$, 
		\[
		\sup_{x\in \Omega_n}F_{\eps}^{+}(x)\leq \sup_{x\in \Omega_n}\E\left[F_{\eps/2}^{+}(x)\Big| B\right]+o_d(1)\,,\quad\textnormal{and}\quad \sup_{x\in \Omega_n}F_{\eps}^{-}(x)\leq \sup_{x\in \Omega_n}\E\left[F_{\eps/2}^{-}(x)\Big| B\right]+o_d(1).
		\]
	\end{lemma}
	\begin{proof}
		Let $\phi_{\eps}^{+}(x)$ be a Lipschitz approximation of $x\mapsto \one\{x\geq \eps\}$, where $\phi_{\eps}^{+}(x)=0$ for $x\leq \eps/2$ and $\phi_{\eps}^{+}(x)=1$ for $x\geq \eps$, and linear in between. Further, let $\phi_{\eps}^{-}(x):=\phi_{\eps}^{+}(-x)$. Then, $\phi_{\eps}^{\pm}$ are both $2/\eps$-Lipschitz. For any bounded $2/\eps$-Lipschitz function $g:\R\to [0,1]$, we claim that with high probability over $B$,
		\begin{equation}\label{eq:claim:lem:F:concentration}
			\sup_{x\in \Omega_n} \left|\Phi_n g(x)-\E\Big[\Phi_n g(x)\Big| B\Big]\right|=o_d(1)\,,~~\textnormal{where}~~ \Phi_n g(x) :=\frac{1}{n}\sum_{i=1}^{n}g\left(\frac{x_i(\Delta x)_i}{\sqrt{d}}\right).
		\end{equation}
		Observe that \eqref{eq:claim:lem:F:concentration} then concludes the proof since we can apply it to $g=\phi_{\eps}^{+}$ and $g=\phi_{\eps}^{-}$, and combine with $\one\{x\geq \eps\}\leq \phi_{\eps}^{+}(x)\leq \one\{x\geq \eps/2\},$ which holds true for every $x \in \R$. 
		
		We are left to prove~\eqref{eq:claim:lem:F:concentration}, following similar arguments as in \cite[Lemma 3.3]{Dembo2017extremal}. For $z>0$, let
		\begin{equation}\label{eq:bound:p:n}
			p(n):=\P\left( \sup_{x\in \Omega_n} \left|\Phi_n g(x)-c(x,B)\right|\geq z\right)\,,\quad\textnormal{where}\quad c(x,B):=\E\Big[\Phi_n g(x)\Big| B\Big].
		\end{equation}
		Since the $B=(B_1,\ldots, B_n)$ are i.i.d., it holds, by Chernoff's inequality, that $B\in \mathcal{E}_n$ w.h.p. where $\mathcal{E}_n:=\{b:|S_n(b)-n\E B_1|\leq \sqrt{n\log n}\}$ and $S_n(b):=\sum_{i=1}^{n}b_i$. Thus by a union bound over $x\in \Omega_n$,
		\[
		p(n)\leq 2^{n}\sup_{x\in \Omega_n}\sup_{b\in \mathcal{E}_n}\P\Big(\big|\Phi_n g(x)-c(x,B)\big|\geq z\Big|B=b\Big)+o_n(1).
		\]
		To control the right-hand side, we will invoke the Azuma-Hoeffding inequality. To apply the inequality we will set a filtration generating the graph and consider the Doob martingale with respect to this filtration. Thus, fix $b\in \mathcal{E}_n$ and half-edge colors following $\{B=b\}$. Observe that $\Delta=A_1-A_2$ are determined by $G_{\br}\cup G_{\bb}$ and $\tilde{G}_{\rr}$, which in turn can be generated as follows: Sequentially match a candidate half-edge to a uniformly chosen half-edge, first using ${\small \sf{BLUE}}$ half-edges for the candidates until all of them are exhausted. Once all ${\small \sf{BLUE}}$ half-edges are matched, cut all $\br$ half-edges and re-match all free ${\small \sf{RED}}$ half-edges. Let $\mathcal{F}_k$ be the sigma-algebra generated by all half-edge colors and first $k\geq 0$ edges to have been paired. Fixing $x\in \Omega_n$, consider the Doob's martingale $M_k= \E[\Phi_n g(x)\mid \mathcal{F}_k]$. This martingale starts with $M_0=\E[\Phi_n g(x)\mid B]$ and since $G_{\rb}\cup G_{\bb}$ and $\tilde{G}_{\rr}$ is determined after $2S_n(b)$ steps in this process, $M_k=\Phi_n g(x)$ for $k\geq 2 S_n(b)$. Moreover, if we let $m$ denote a realization of a matching of all ${\small \sf{BLUE}}$ half-edges and ${\small \sf{RED}}$ freed half-edges, and let $\Phi_n g(x,m)$ be the value of $\Phi_n g(x)$ realized by $m$, then
		\begin{equation}\label{eq:martingale:diff}
			\sup_{k\geq 0}|M_k-M_{k-1}|\leq \sup_{m,m'} \left|\Phi_n g(x,m)-\Phi_n g(x,m')\right|,
		\end{equation}
		where the latter supremum is taken with respect to all neighboring matchings $m,m'$, i.e. $m'$ can be obtained from $m$ switching two edges. Write $\Delta$ and $\Delta'$ for the value of $A_1-A_2$ realized by $m$ and $m'$ respectively, and let $i_1,i_2,i_3,i_4$ be vertices in which $(i_1,i_2), (i_3,i_4)\in m$ whereas $(i_1,i_3),(i_2,i_4)\in m'$. Then,
		\begin{equation}\label{eq:Phi:diff:bound}
			\begin{split}
				\left|\Phi_ng(x,m)-\Phi_ng(x,m')\right|
				&\leq \frac{1}{n}\sum_{\ell=1}^{4}\left|g\left(\frac{x_{i_{\ell}}(\Delta x)_{i_{\ell}}}{\sqrt{d}}\right)-g\left(\frac{x_{i_{\ell}}(\Delta'x)_{i_{\ell}}}{\sqrt{d}}\right)\right|\\
				&\leq \frac{2}{n\eps}\sum_{\ell=1}^{4}\left|\frac{((\Delta-\Delta')x)_{i_{\ell}}}{\sqrt{d}}\right|\leq \frac{16}{n\eps\sqrt{d}},
			\end{split}
		\end{equation}
		where the last inequality holds since $\Delta_{i_{\ell}j}=\Delta'_{i_{\ell}j}$ for all $j\in [n]$ except $2$ vertices, in which case they differ by at most $1$. Thus, $\sup_{k\geq 0}|M_k-M_{k-1}|\leq 16/(n\eps\sqrt{d})$. By the Azuma-Hoeffding inequality, we get that for $b\in \mathcal{E}_n$,
		\[
		\P\Big(\big|\Phi_n g(x)-c(x,B)\big|\geq z_n\Big|B=b\Big)\leq 2\exp\left(-\frac{z_n^2n^2\eps^2 d}{CS_n(b)}\right)\leq 2\exp\left(-\frac{z_n^2n^2\eps^2 d}{C(n\E B_1+\sqrt{n\log n})}\right).
		\]
		Since $\E[B_1]=\sqrt{d}\log d+ O_d(1)$ (cf.~\cite[Eq. (3.15)]{Dembo2017extremal}), choosing $z=C'd^{-1/4}\sqrt{\log d}/\eps$ in \eqref{eq:bound:p:n}, for large enough absolute constant $C'>0$ shows $p(n) = o_n(1)$. Moreover, since with this choice $z=o_d(1)$, and since $\Phi_ng$ is bounded, this also establishes~\eqref{eq:claim:lem:F:concentration}, thus concluding the proof.
	\end{proof}\vspace{-\baselineskip}
	\begin{lemma}\label{lem:concentration:for:most:i}
		For any fixed $\eps>0$, with probability $1-o(1)$ as $n\to\infty$,
		\begin{equation}\label{eq:lem:concentration:for:most:i}
			\sup_{x\in \Omega_n}\Bigg\{\frac{1}{n}\sum_{i=1}^{n}\P\left(\bigg|\frac{x_i(\Delta x)_i-\E\big[x_i(\Delta x)_i\big| B\big]}{\sqrt{d}}\bigg|\geq \eps \,\Bigg|\,B\right)\Bigg\}= o_d(1).
		\end{equation}
	\end{lemma}
	\begin{proof}
		Fix $i\in [n], x\in \Omega_n$, and we condition on half-edge colors such that $\{B=b\}$ for some $b$. Abbreviate $f_i(x)=x_i(\Delta x)_i/\sqrt{d}$. Recalling the filtration $(\mathcal{F}_k)_{k\geq 0}$ from the proof of Lemma~\ref{lem:F:concentration}, consider the Doob's martingale $N_k=\E[f_i(x)\mid \mathcal{F}_k]$. As before, this martingale starts with $N_0=\E[f_i(x)\mid B]$ and $N_k=f_i(x)$ holds for $k\geq 2S_n(b)$, where $S_n(b):=\sum_{i=1}^{n}b_i$.  For $k\geq 2 S_n(b)$, we claim that its predictable quadratic variation satisfies
		\begin{equation}\label{eq:quadratic:variation:upper}
			\langle N\rangle_k\equiv \sum_{\ell=1}^{k} \E[(N_{\ell}-N_{\ell-1})^2\mid \mathcal{F}_{\ell-1}]\leq \frac{64}{d}(B_i+2RB_i),
		\end{equation}
		where $RB_i$ denote the number of ${\small \sf{RED}}$ half-edges incident to $i$ paired with ${\small \sf{BLUE}}$ half-edges. To this end, recall the sequential matching process generating $(\mathcal{F}_k)_{k\geq 0}$. By definition of $f_i(x)$, given $\mathcal{F}_{\ell-1}$, if the $\ell$'th half-edge to be paired is not incident to $i$ and is not matched to an half-edge incident to $i$, then $N_{\ell}=N_{\ell-1}$. Otherwise, following the same switching argument as in \eqref{eq:martingale:diff} and \eqref{eq:Phi:diff:bound}, we have $|N_{\ell}-N_{\ell-1}|\leq 8/\sqrt{d}$. Since the number of steps in the matching process that involves vertex $i$ is equal to $B_i + 2RB_i$, where the factor $2$ stems from rewiring, the inequality in~\eqref{eq:quadratic:variation:upper} holds.
		
		Thus, by Freedman's martingale version of Bernstein's inequality~\cite{freedman75} (see e.g.~\cite{dzhaparidze01}), we have for $v>0$,
		\[
		\P\left(\left|f_i(x)-\E[f_i(x)]\mid B]\right|\geq \eps\,,\, \frac{64(B_i+2RB_i)}{d}\leq v\,\Bigg|\, B\right)\leq 2\exp\left(-\frac{\frac{\eps^2}{2}}{v+\frac{8\eps}{3\sqrt{d}}}\right).
		\]
		We take $v=C\eps^2/\log d$ so that the right-hand side is $d^{-\Omega(1)}$ since $\eps>\log d/\sqrt{d}$ for large enough $d$. Averaging over $i\in [n]$ and taking supremum over $x\in \Omega_n$, we find that the left-hand side of \eqref{eq:lem:concentration:for:most:i} is at most
		\[
		\begin{split}
			\frac{1}{n}\sum_{i=1}^{n}\P\left(\frac{B_i+2RB_i}{d}\geq \frac{C'\eps^2}{\log d}\,\Bigg|\,B\right)+d^{-C}&
			\leq \frac{1}{C'\eps^2 }\frac{\log d}{nd}\sum_{i=1}^{n}\E[B_i+2RB_i\mid B]+ d^{-C}\\
			&\leq \frac{1}{C'\eps^2}\frac{3\log d}{nd}\sum_{i=1}^{n}B_i+d^{-C},
		\end{split}
		\]
		where we used Markov's inequaltiy in the first step and $\sum_{i=1}^{n} RB_i \leq \sum_{i=1}^{n} B_i$ in the second step. Finally, note that since $(B_i)_{i\leq n}$ are i.i.d. $\sum_{i=1}^{n}B_i/n\leq 2 \E[B_1]=2\sqrt{d}\log d+O_d(1)$ with high probability, concluding the proof.
	\end{proof}\vspace{-\baselineskip}
	\noindent We will require one final bound, which will follow from estimates on the max-cut size of $G^{\textnormal{c}}_{\reg}(n,d)$, established in \cite{Dembo2017extremal}. Recall that in a graph $G$ with adjacency matrix $A$ and $x\in \Omega_n$, the cut size of $x$ w.r.t. $A$ is
	\[
	\cut_A(x)\equiv \frac{|E|}{2}-\frac{1}{2}\sum_{(i,j)\in E}x_ix_j=\frac{1}{4}\left(\langle \one, A\one \rangle - \langle x, Ax\rangle \right),
	\]
	where $E$ denote the edge set of the graph $G$.
	\begin{lemma}\label{lem:local:field:to:cut}
		With probability $1-o(1)$ as $n\to\infty$
		\[
		\sup_{x\in \Omega_n}\left\{\frac{1}{n}\sum_{i=1}^{n}\E\left[\frac{x_i(\Delta x)_i}{\sqrt{d}}\,\bigg|\, B\right]\right\}=o_d(1)\,,\quad\textnormal{and}\quad 
		\inf_{x\in \Omega_n}\left\{\frac{1}{n}\sum_{i=1}^{n}\E\left[\frac{x_i(\Delta x)_i}{\sqrt{d}}\,\bigg|\, B\right]\right\}=o_d(1).
		\]
	\end{lemma}
	\begin{proof}
		With abuse of notations, let $\cut_{\Delta}(x)\equiv \cut_{A_1}(x)-\cut_{A_2}(x)$. By definition of $\cut_A(x)$, we have
		\begin{equation*}
			\frac{1}{n}\sum_{i=1}^{n}\E\left[\frac{x_i(\Delta x)_i}{\sqrt{d}}\,\bigg|\, B\right]=\frac{1}{n\sqrt{d}}\E \Big[\langle \one, \Delta \one \rangle - 4\cut_{\Delta}(x)\,\Big|\, B\Big].
		\end{equation*}
		Note that $\langle \one, \Delta \one\rangle=\sum_i \{\deg_{G_{\rb\cup \bb}}(i)-\deg_{\tilde{G}_{\rr}}(i)\}=\sum_i B_i$. Since $\sum_{i} B_i/n=\E[B_1]+o(1)$ w.h.p., and $\E[B_1]=\sqrt{d}\log d+O_d(1)$, it follows that w.h.p., uniformly over all $x\in \Omega_n$,
		\begin{equation}\label{eq:local:field:to:cut}
			\frac{1}{n}\sum_{i=1}^{n}\E\left[\frac{x_i(\Delta x)_i}{\sqrt{d}}\,\bigg|\, B\right]=\log d+O_d(d^{-1/2})-\frac{4}{n\sqrt{d}}\E[\cut_{\Delta}(x)\mid B].
		\end{equation}
		Meanwhile, it was shown in the proof of \cite[Lemma 3.3]{Dembo2017extremal} that for both $\mathcal{A}= G_{\rb\cup \bb}$ and $\mathcal{A}=\tilde{G}_{\rr}$, w.h.p.
		\[
		\sup_{x\in \Omega_n}\Big|\E[\cut_{\mathcal{A}}\mid B]-\E[\cut_{\mathcal{A}}(x)]\Big|\leq Cnd^{1/4}\sqrt{\log d}.
		\]
		Moreover, by \cite[Lemma 3.2]{Dembo2017extremal}, uniformly over all $x\in \Omega_n$,
		\[
		\E[\cut_{G_{\rb\cup \bb}}(x)]=n\left(\frac{\sqrt{d}\log d}{2}+O_d(\log ^2 d)\right)+o(n)\,,~~\E[\cut_{\tilde{G}_{\rr}}(x)]=n\left(\frac{\sqrt{d}\log d}{4}+O_d(1)\right)+o(n).
		\]
		As a result, we have that w.h.p., uniformly over all $x\in \Omega_n$,
		\[
		\begin{split}
			\frac{1}{n}\E[\cut_{\Delta}(x)\mid B]
			&=\frac{1}{n}\left(\E[\cut_{{G_{\rb\cup \bb}}}(x)]-\E[\cut_{\tilde{G}_{\rr}}(x)]\right)+O_d(d^{1/4}\sqrt{\log d})\\
			&=\frac{\sqrt{d}\log d}{4}+O_d(d^{1/4}\sqrt{\log d})+o(1).
		\end{split}
		\]
		Therefore, plugging this estimate into \eqref{eq:local:field:to:cut} shows that uniformly over $x\in \Omega_n$, the left-hand side of \eqref{eq:local:field:to:cut} equals $O_d(d^{-1/4}\log d)$, concluding the proof.
	\end{proof}\vspace{-\baselineskip}
	\begin{proof}[Proof of Proposition~\ref{prop:gapped:random:regular:2}]
		Fix $\eps>0$. Abbreviate $f_i(x)=x_i(\Delta x)_i/\sqrt{d}$. By Lemma~\ref{lem:F:concentration}, we have that w.h.p., $\sup_{x\in \Omega_n} F_{\eps}^{+}(x)$ is at most
		\begin{equation*}\label{eq:gapped:random:regular:union:bound}
			\begin{split}
				\sup_{x\in \Omega_n} \E\Big[F_{\eps/2}^{+}(x)\Big|B\Big]
				\leq \sup_{x\in \Omega_n}\Bigg(\frac{1}{n}\sum_{i=1}^{n}\left\{\P\bigg(\Big|f_i(x)-\E[f_i(x)\mid B]\Big|\geq \frac{\eps}{4}\,\bigg|\,B\bigg)+\one\left\{\E[f_i(x)\mid B]\geq \frac{\eps}{4}\right\}\right\}\Bigg),
			\end{split}
		\end{equation*}
		where we used a union bound. Analogous upper bound holds for $\sup_{x\in \Omega_n} F_{\eps}^{-}(x)$, where we replace the above indicator by $\one\{\E[f_i(x)|B]\leq -\eps/4\}$. By Lemma~\ref{lem:concentration:for:most:i}, the supremum of the first term in the right-hand side is $o_d(1)$ with high probability. Thus, applying Markov's inequality to the indicator shows that w.h.p., we have
		\[
		\sup_{x\in \Omega_n} F_{\eps}^{\pm}(x)\leq o_d(1)+\sup_{x\in \Omega_n}\left\{\pm\frac{4}{n\eps}\sum_{i=1}^{n}\E\left[f_i(x)|B\right]\right\}.
		\]
		By Lemma~\ref{lem:local:field:to:cut}, the final supremum is $o_d(1)$ w.h.p., thus $\sup_{x\in \Omega_n} F_{\eps}^{\pm} (x)=o_d(1)$ w.h.p..
	\end{proof}\vspace{-\baselineskip}
     	
     	\section{Remaining proofs} \label{sec:app}
     	In this section, we prove lower bounds on the anti-ferromagnetic Curie-Weiss model, and the corollaries to Theorem \ref{thm:main} from Section \ref{sec:corrolaries}. 
     	\begin{proof}[Proof of Corollary \ref{cor:dregular}]
     		Let $G$ be any $d$-regular graph, with $A_G$ its adjacency matrix with eigenvalues $d = \lambda_1\geq \lambda_2 \geq \dots \lambda_n \geq -d$. Since $A_G {\bf{1}} = d{\bf{1}}$ and $d$ is the maximal eigenvalue of $A_G$, by the spectral theorem $A_G = d{\bf{1}}{\bf{1}}^T + M,$ where $\mathrm{image}(M)\subset {\bf{1}}^\perp$. In particular, because $\Tr(A_G) = 0$ necessarily $\lambda_n < 0$, and $\lambda_{\min}(M) = \lambda_n$. Also, $\lambda_{\max}(M) = \lambda_2 \geq 0$ holds unless $d = n-1$, and $G$ is the complete graph \cite{torgaev1985graphs}. 
     		Assume for now $d < n-1$. Since the Ising measure is supported on $\{-1,1\}^n$, we can add a multiple of the identity and rewrite $\mu_{J_G, h}$ as
     		$$\mu_{J_G, h}(x) \propto \exp\left(-\frac{\beta}{n} d\langle {\bf{1}},x\rangle^2 + \beta\langle x, \tilde{M}x\rangle + \langle h,x\rangle\right),$$
     		where $\tilde{M}$ is positive semi-definite with $\|M\|_{\op} = \lambda_2 - \lambda_n$. Taking $u = {\bf 1}$ and $J = \beta\tilde{M}$, Theorem \ref{thm:main} applies provided that $\beta d \leq n^{1-\eps}$ and $\beta\left(\lambda_2 - \lambda_n\right)< \frac{1}{2},$ for some $\eps >0$. This condition is equivalent to $\beta < \min\left(\frac{1}{2(\lambda_2-\lambda_n)}, \frac{n^{1-\eps}}{d}\right)$, and so we can conclude the proof when $d < n-1$. When $d = n-1$, for the complete graph, $\|\tilde M\| = 1$ instead, but now $\min\left(\frac{1}{2(\lambda_2-\lambda_n)}, \frac{n^{1-\eps}}{d}\right) = \min\left(\frac{1}{2}, \frac{n^{1-\eps}}{d}\right) = \frac{n^{1-\eps}}{d}$, and we arrive at the same conclusion. 
     	\end{proof}\vspace{-\baselineskip}
     	\begin{proof}[Proof of Corollary \ref{cor:randomdregular}]
     		By Friedman's inequality \cite{friedman2009proof, bordnave2020proof}, and its analogs for dense graphs \cite{sarid2023spectral, he2024spectral}, for any $d \leq \frac{n}{2}$ if $A_G$ is the adjacency matrix of a random $d$-regular, then with high probability
     		$$\max(\lambda_2,|\lambda_n|) \leq M_d:=2\sqrt{d\left(1-\frac{d}{n}\right)}(1+o(1)),$$
     		and the same holds for $d > \frac{n}{2}$ by symmetry. 
     		For $d = \Theta(1)$, the bound sharpens to $M_d = 2\sqrt{d-1}$.
     		Thus,  $\lambda_2 - \lambda_n \leq 2M_d$ and if $\beta < \frac{1}{2M_d}$ then with high probability $\beta < \frac{1}{2(\lambda_2-\lambda_n)}$ as well, at least when $n$ is large. The corollary now follows from Corollary~\ref{cor:dregular}. 
     		Note further that for $d = \Theta(1)$ or even $d < (1-\eps)n$, then the condition $\beta < \frac{1}{2M_d}$ automatically enforces $\beta \leq \frac{n^{1-\eps}}{d}$, for large enough $n$.
     	\end{proof}\vspace{-\baselineskip}
     	\begin{proof}[Proof of Corollary \ref{thm:denseER}]
     		The proof is virtually identical to the proof of $d$-regular graphs. Again, let $\lambda_1\geq \lambda_2\geq \lambda_n$ be the eigenvalues of $A_G$. When $d \geq C\log(n)$ for large enough $C>0$, $G$ is connected with overwhelming probability and the eigenvector $u$, corresponding to $\lambda_1$ satisfies $\|u\|_{\infty}\leq n^{-1/2+o(1)}\|u\|_2$, see \cite[Theorem 2.7]{he2019local} for example. Furthermore, we have that $\lambda_1 \leq d(1+o(1))$ \cite{krivelevich2003largest}, and  $\max(\lambda_2,|\lambda_n|) \leq \frac{2}{c}\sqrt{d}(1+o(1))$, where $c > 0$ depends only on $C$, see \cite[Theorem 3.2]{benaych2020spectral} and \cite[Corollary 2.3]{alt2021extremal}. To refine the bound, \cite[Corollary 2.3]{alt2021extremal} also shows that as $C \to \infty$, $c \to 1$.
     		Therefore, we have the following expression
     		$$\mu_{J_G, h}(x) \propto \exp\left(-\frac{\beta'}{n} d\langle u,x\rangle^2 + \beta\langle x, \tilde{M}x\rangle + \langle h,x\rangle\right),$$
     		where $\beta' = \beta + o(1)$, $u$ is an eigenvector with $\|u\|_\infty = 1$ and $\tilde{M}$ is a positive semi-definite matrix satisfying $\|\tilde{M}\|_{\op} \leq \frac{4}{c}\sqrt{d}(1+o(1)).$
     		Again, the condition on $\beta$ allows to invoke Theorem \ref{thm:main} from whence the result follows.
     	\end{proof}\vspace{-\baselineskip}

     	\medskip\noindent\textbf{Slow mixing of anti-ferromagnetic Curie-Weiss when $\beta = \Theta(n)$:}
     	Here, we consider the measure
     	$$\pi(x)  = \frac{1}{Z_n(\beta_0)}\exp\bigg(-\beta_0\Big(\sum\limits_{i=1}^n x_i\Big)^2\bigg) \qquad x \in \{-1,+1\}^n,$$
     	and show that the associated Glauber dynamics has a mixing time which scales by $e^{\beta_0}$. This example rules out possible extensions of Theorem \ref{thm:main} to the regime $\beta = \Theta(n)$.
     	\begin{lemma} \label{lem:lowerbound}
     		Let $\beta_0 > 0$ and let $\pi$ be as above, then $\rho_{\mathrm{LS}}\left(\pi\right) \leq \frac{C}{ne^{4\beta_0}}$.
     	\end{lemma}
     	\begin{proof}
     		For simplicity, we assume that $n$ is even and define the sub-cube $S:= \left\{x \in \{-1,+1\}^n: x_1 = 1\right\}.$
     		Consider the conductance $Q(S,S^c):= \sum\limits_{x\in S, y\notin S} \pi(x)P(x,y),$
     		where $P$ is the transition kernel for Glauber dynamics. 
     		For any $x \in S$, and $y \in S^C$, $P(x,y) \neq 0$ if only if $y$ is obtained from $x$ by flipping the first coordinate, which we denote as $x'$.
     		Now, if $x$ is balanced, in the sense that $\sum x_i = 0$, then
     		$$P(x,x') = \frac{1}{n}\frac{\pi(x')}{\pi(x)+\pi(x')} = \frac{1}{n}\frac{e^{-4\beta_0}}{e^{-4\beta_0} + 1}\leq \frac{1}{n}\frac{1}{e^{4\beta_0}}.$$
     		Otherwise, we use the crude bound $P(x,x')\leq \frac{1}{n}$ and we estimate the conductance by
     		$$Q(S,S^c) \leq  \frac{1}{Z_n(\beta_0)n} \left(\binom{n}{n/2}\frac{1}{e^{4\beta_0}} + 2\sum\limits_{m=1}^{n/2}\binom{n}{n/2- m}e^{-4\beta_0 m^2} \right).$$
     		Focusing on the sum,
     		\begin{align*}
     			\sum\limits_{m=1}^{n/2}\binom{n}{n/2- m}e^{-4\beta_0 m^2} &\leq \sum\limits_{m=0}^{n/2}\binom{n}{n/2- m}e^{-4\beta_0 (m+1)^2}\\
     			&= \sum\limits_{m=0}^{n/2}\binom{n}{n/2- m} e^{-4\beta_0 m^2}\frac{e^{-4\beta_0 (m+1)^2}}{e^{-4\beta_0 m^2}} \\
     			&\leq \frac{1}{e^{4\beta_0}}\sum\limits_{m=0}^{n/2}\binom{n}{n/2- m} e^{-4\beta_0 m^2} = \frac{1}{e^{4\beta_0}}\frac{Z_n(\beta_0)}{2}
     		\end{align*}
     		Clearly, $Z_n(\beta_0) \geq \binom{n}{n/2}$ and $\pi(S) = \frac{1}{2}$. We have thus shown
     		$\frac{Q(S,S^c)}{\pi(S)} \leq  \frac{4}{n e^{4\beta_0}}$.
     		The bound on $\rho_{\mathrm{LS}}(\pi)$ now follows from Cheeger's inequality.
     	\end{proof}\vspace{-\baselineskip}
	\bibliographystyle{alpha}
	\bibliography{bib}{}
\end{document}